\documentclass[12pt,a4paper]{amsart}
\usepackage{amssymb,amsxtra}
\usepackage[all,cmtip]{xy} 

\calclayout

\newcommand{\+}{\protect\nobreakdash-}
\renewcommand{\:}{\colon}

\newcommand{\rarrow}{\longrightarrow}
\newcommand{\ot}{\otimes}
\newcommand{\la}{\leftarrow}

\newcommand{\lrarrow}{\mskip.5\thinmuskip\relbar\joinrel\relbar\joinrel
 \rightarrow\mskip.5\thinmuskip\relax}

\DeclareMathOperator{\Hom}{Hom}
\DeclareMathOperator{\coker}{coker}
\DeclareMathOperator{\im}{im}
\DeclareMathOperator{\coim}{coim}

\newcommand{\Sets}{\mathsf{Sets}}
\newcommand{\Top}{\mathsf{Top}}
\newcommand{\Ab}{\mathsf{Ab}}
\newcommand{\Vect}{\mathsf{Vect}}
\newcommand{\Pro}{\mathsf{Pro}}

\newcommand{\plim}
 {\mathop{\text{\normalfont``$\varprojlim$''\!\!}}\nolimits}
\renewcommand{\cot}{\mathbin{\widehat\ot}}

\newcommand{\s}{\mathsf s}
\newcommand{\scc}{\mathsf{sc}}
\newcommand{\su}{\mathsf{su}}
\newcommand{\lie}{\mathsf{le}}
\newcommand{\se}{\mathsf{se}}
\newcommand{\jsm}{\mathsf{jsm}}

\newcommand{\sA}{\mathsf A}
\newcommand{\sB}{\mathsf B}
\newcommand{\sC}{\mathsf C}
\newcommand{\sE}{\mathsf E}
\newcommand{\sF}{\mathsf F}
\newcommand{\sG}{\mathsf G}

\newcommand{\fA}{\mathfrak A}
\newcommand{\fB}{\mathfrak B}
\newcommand{\fC}{\mathfrak C}
\newcommand{\fD}{\mathfrak D}
\newcommand{\fK}{\mathfrak K}
\newcommand{\fL}{\mathfrak L}
\newcommand{\fP}{\mathfrak P}
\newcommand{\fQ}{\mathfrak Q}
\newcommand{\fR}{\mathfrak R}
\newcommand{\fU}{\mathfrak U}
\newcommand{\fV}{\mathfrak V}
\newcommand{\fW}{\mathfrak W}

\newcommand{\cB}{\mathcal B}
\newcommand{\cF}{\mathcal F}

\newcommand{\boZ}{\mathbb Z}

\theoremstyle{plain}
\newtheorem{thm}{Theorem}[section]
\newtheorem{prop}[thm]{Proposition}
\newtheorem{lem}[thm]{Lemma}
\newtheorem{cor}[thm]{Corollary}
\theoremstyle{definition}
\newtheorem{qst}[thm]{Question}
\newtheorem{conc}[thm]{Conclusion}
\newtheorem{rem}[thm]{Remark}
\newtheorem{ex}[thm]{Example}
\newtheorem{exs}[thm]{Examples}

\newcommand{\Section}[1]{\bigskip\section{#1}\medskip}
\begin{document}

\title{Exact categories of topological vector spaces \\
with linear topology}

\author{Leonid Positselski}

\address{Institute of Mathematics of the Czech Academy of Sciences \\
\v Zitn\'a~25, 115~67 Praha~1 (Czech Republic); and
\newline\indent Laboratory of Algebra and Number Theory \\
Institute for Information Transmission Problems \\
Moscow 127051 (Russia)}

\email{positselski@math.cas.cz}

\begin{abstract}
 We explain why the na\"\i ve definition of a natural exact category
structure on complete, separated topological vector spaces with linear
topology fails.
 In particular, contrary to~\cite{Beil}, the category of such
topological vector spaces is not quasi-abelian.
 We present a corrected definition of exact category structure
which works~OK\@.
 Then we explain that the corrected definition still has a shortcoming
in that a natural tensor product functor is not exact in it, and discuss
ways to refine the exact category structure so as to make the tensor
product functors exact. 
\end{abstract}

\maketitle

\tableofcontents

\section*{Introduction}
\medskip

 Topological algebra is a treacherous ground.
 So is general topology, of which topological algebra is a part.
 Well-behaved classes of topological algebraic structures are few
and far between, and hard to come by.
 The aim of this paper is to warn the reader about some of
the dangers.

 Topological vector spaces as a subject have a functional analysis
flavour.
 In this paper we consider the most ``algebraic'' class of topological
vector spaces, viz., topological vector spaces with linear topology
(VSLTs).
 Hence the results we discuss tend to be true or false irrespectively
of the ground field~$k$ (which is discrete).

 An optimistic vision of topological vector spaces with linear topology
was expressed in Beilinson's paper~\cite{Beil}.
 In the present paper we explain that some of most basic assertions
in~\cite{Beil} are actually not true.
 The construction of counterexamples which we operate with is taken
from the book~\cite{AGM} by Arnautov \emph{et al.}; in fact, it goes
back to the much earlier book of Roelcke and Dierolf~\cite{RD}.

 The present author became aware of the failure of straightforward
attempts to prove some assertions from~\cite{Beil} back in May~2008,
when working on the book~\cite{Psemi}.
 So I~tried to exercise extra caution in~\cite[Section~D.1]{Psemi},
in order to avoid the traps.
 The conceptual and terminological system of  the July~2018
preprint~\cite[Sections~1\+-10]{BP0} reflected this understanding.

 Still we did not know whether the problematic assertions in~\cite{Beil}
were true or not true.
 It was only in October~2018 that I~learned about the book~\cite{AGM}
and the counterexamples in~\cite[Theorem~4.1.48]{AGM}.
 Subsequently the reference to~\cite[Theorem~4.1.48]{AGM} appeared in
the July~2019 version of~\cite{Pproperf}.
 In fact, the counterexamples in~\cite[Problem~20D]{KN} are already
sufficient for some purposes. 

 Surprisingly, the categories of incomplete (arbitrary or Hausdorff)
topological vector spaces have better exactness properties than
the category of complete, separated topological vector spaces.
 In particular, the categories of incomplete topological vector
spaces are quasi-abelian (unlike the category of complete ones).

 Weakening the axioms of abelian category in order to develop
homological functional analysis is a natural idea, which was tried,
e.~g., in the long paper~\cite{Pal}.
 The observation that the categories of incomplete topological vector
spaces are quasi-abelian, while the category of complete ones is not,
was elaborated upon (in the context of locally convex spaces) in
the 2000 paper~\cite{Pros2}.

 The problem is that \emph{the quotient space of a complete topological
vector space by a closed subspace need not be complete}.
 In the functional analysis context, this observation seems to go back
to K\"othe's 1947 paper~\cite{Koeth},
see also~\cite[\S\S19.5 and~31.6]{Koeth2}.
 Moreover, in the same context Dierolf proved that \emph{any}
topological vector space is a quotient of a complete
one~\cite{Die1,Die2}, \cite[Section~2.1]{FW}.
 This result was generalized to topological abelian groups of quite
general nature in the 1981 book~\cite[Proposition~11.1]{RD};
the same construction is presented in~\cite[Theorem~4.1.48]{AGM}.

 Many a mathematician unfamiliar with the subject would guess nowadays
that the embedding of topological vector spaces with linear topology
into the abelian category of pro-vector spaces resolves many problems.
 We discuss the idea in detail, and the conclusion is that this is
rather not the case.
 The category of complete, separated topological vector spaces with
linear topology, viewed as a full subcategory in the abelian category of
pro-vector spaces, is not closed under extensions, and does not inherit
an exact category structure.

 In the final sections of the paper we study Beilinson's constructions
of three tensor product operations on topological vector spaces with
linear topology.
 These topological tensor product functors have a number of good
properties which we verify, but once again we show that there are
subtle issues involved.
 The problem is that, even when a quotient space $\fC$ of a complete
vector space $\fV$ is complete, it is, generally speaking, impossible
to lift a given zero-convergent sequence or family of elements in $\fC$
to a zero-convergent family of elements in~$\fV$.

 All in all, it is not obvious from our discussion (as well as from
the one in~\cite{Pros2}) that the language of exact categories provides
the most suitable category-theoretic point of view on topological
algebra, particularly if one is interested in \emph{complete}
topological abelian groups or vector spaces (which is usually the case).
 Nevertheless, we hope that the present paper will serve as a useful
basic reference source on topological algebraic structures with linear
topology and their categorical properties.

 Let us say a few words about the \emph{linear topology} terminology
used throughout this paper starting from the title.
 It may be unfamiliar to some people in functional analysis
(even though in appears in K\"othe's book~\cite[\S10]{Koeth2}).
 This terminology, going back to Lefschetz'
classics~\cite[Definitions~II.1.2 and~II.25.1]{Lef}, is a standard
nomenclature in algebra.
 See, e.~g., Fuchs' book~\cite[Section~7 in Chapter~I]{Fuchs}, or
the book of Beilinson and Drinfeld~\cite[Section~3.6.1]{BeDr}, or
the present author's book~\cite[Section~D.1.1]{Psemi}.

 One speaks about linear topologies on abelian groups or vector spaces,
meaning topologies with a base of neighborhoods of zero formed by
subgroups/subspaces; linear topologies on modules over rings, meaning
topologies with a base of neighborhoods of zero formed by submodules;
left and right linear ring topologies (or left and right linear
topological rings), meaning topological rings with a base of
neighborhoods of zero consisting of left/right ideals, etc.
 Topological vector spaces with linear topology form a natural class
of topological vector spaces \emph{over discrete fields}, analogous to
the class of locally convex topological vector spaces over the normed
fields of real or complex numbers in functional analysis.

 Let us describe the content of the paper section by section.
 An introductory discussion of topological algebraic groups with linear
topology is presented in Section~\ref{top-abelian-secn}.
 We proceed to describe the main construction of counterexamples to
completeness of quotients, which will play a key role in the rest of
the paper, in Section~\ref{counterex-secn}.
 Topological vector spaces with linear topology are introduced and
the basic properties of them discussed in
Section~\ref{top-vector-secn}.

 The language of exact categories in Quillen's sense, which forms
the category-theoretic background for the discussion in the rest of
the paper, is elaborated upon in Section~\ref{exact-categories-secn}.
 Two important particular cases, namely the quasi-abelian categories
and the maximal exact category structures (on weakly idempotent-complete
additive categories) are considered in
Sections~\ref{quasi-abelian-secn} and~\ref{maximal-exact-struct-secn}.

 The properties of the categories of topological abelian groups and
vector spaces with linear topology, in the context of the general
theory of additive categories, are discussed
in Sections~\ref{incomplete-VSLTs-secn}
and~\ref{maximal-exact-VSLTs-secn}.
 The reader can find a brief summary in
Conclusion~\ref{maximal-exact-conclusion}.

 The interpretation of complete, separated topological vector spaces
as objects of the abelian category of pro-vector spaces is
discussed in Section~\ref{pro-vector-spaces-secn}.
 For incomplete topological vector spaces, what we call
\emph{supplemented} pro-vector spaces are needed, and we consider
these in Section~\ref{suppl-pro-secn}.
 The theory which we develop in these sections provides a kind of
category-theoretic generalization of topological algebra.
 See Conclusion~\ref{proobjects-conclusion} for a brief summary.

 The construction of the abelian group/vector space of zero-convergent
families of elements, together with the related notions of strongly
surjective maps and the strong exact category structure, are presented
in Section~\ref{strong-exact-structure-secn}.
 These play an important role in the theory of contramodules over
topological rings, as developed in our recent
works~\cite{Pweak,PR,PS,Pcoun} and particularly~\cite{Pproperf,PS3}.
 Besides, this construction occurs as a particular case of some
of Beilinson's topological tensor products (with a discrete vector
space), and serves as a source of our counterexamples in
this context.

 The uncompleted versions of Beilinson's topological tensor products
(for incomplete topological vector spaces) are studied in
Section~\ref{tensor-product-yoga-secn}, and the completed topological
tensor products (of complete VSLTs) are discussed in
Section~\ref{refined-exact-struct-secn}.
 Some natural questions arising in this context we were unable to
answer; they are formulated as open questions in
Section~\ref{refined-exact-struct-secn}.
 The reader can find the final summary in
Conclusion~\ref{final-conclusion}.

\subsection*{Acknowledgment}
 I~am grateful to Jan Trlifaj for showing me the book~\cite{AVM}
in October~2018, which helped me to discover the book~\cite{AGM}.
 This made the present research possible.
 I~learned about the book~\cite{RD} from Jochen Wengenroth during
the Zoom conference ``Additive categories between algebra and
functional analysis'' in March~2021; I would like to thank both him
and the organizers of the conference.
 I~also wish to thank Jan \v St\!'ov\'\i\v cek for helpful discussions.
 The author is supported by the GA\v CR project 20-13778S and
research plan RVO:~67985840.

\Section{Abelian Groups with Linear Topology}  \label{top-abelian-secn}

 A \emph{topological abelian group} $A$ is a topological space with
an abelian group structure such that the summation map ${+}\,\:A\times A
\rarrow A$ is continuous (as a function of two variables) and
the inverse element map ${-}\,\:A\rarrow A$ is also continuous.
 A topological abelian group $A$ is said to have \emph{linear topology}
if open subgroups form a base of neighborhoods of zero in~$A$.
 In this paper, all the ``topological abelian groups'' will be presumed
to have linear topology.

 A linear topology on an abelian group $A$ is determined by
the collection of all open subgroups in~$A$.
 A collection of subgroups in an abelian group $A$ is the collection
of all open subgroups in a (linear) topology if and only if it is
a filter, i.~e., the following conditions are satisfied:
\begin{itemize}
\item $A$ is an open subgroup in itself;
\item if $U'\subset U''\subset A$ are subgroups in $A$ and $U'$ is
an open subgroup, then so is~$U''$;
\item the intersection of any two open subgroups in $A$ is open.
\end{itemize}

 A collection of subgroups $\cB$ in $A$ is a base of neighborhoods of
zero in some linear topology on $A$ if and only if
\begin{itemize}
\item $\cB$ is nonempty; and
\item for any $U'$, $U''\in\cB$ there exists $U\in\cB$ such that
$U\subset U'\cap U''$.
\end{itemize}
 In this case, a subgroup in $A$ is open if and only if it contains
a subgroup belonging to~$\cB$.

 The \emph{morphisms} of topological abelian groups $f\:A\rarrow B$ are
the continuous additive maps, that is, the abelian group homomorphisms
such that $f^{-1}(U)$ is an open subgroup in $A$ for every open subgroup
$U\subset B$.

 The \emph{completion} of a topological abelian group $A$ is defined
as the projective limit
$$
 \fA=A\sphat\,=\varprojlim\nolimits_{U\subset A}A/U,
$$
where $U$ ranges over the poset of all open subgroups in $A$ (or
equivalently, any base of open subgroups in~$A$).
 The projective limit is taken in the category of abelian groups.
 The abelian group $\fA$ is endowed with the topology in which
the open subgroups $\fU\subset\fA$ are precisely the kernels of
the projection maps $\fU=\ker(\fA\to A/U)$, where $U$ ranges over
the open subgroups in~$A$.
 When $U\in\cB$ ranges over a base of open subgroups in $A$,
the related subgroups $\fU$ form a base of open subgroups in~$\fA$.

 The collection of projection maps $A\rarrow A/U$ defines a natural
continuous abelian group homomorphism $\lambda_A\:A\rarrow\fA$,
which is called the \emph{completion map}.
 A topological abelian group $A$ is said to be \emph{separated}
(or \emph{Hausdorff}) if the map~$\lambda_A$ is injective.
 Equivalently, this means that the intersection of all open subgroups
in $A$ is the zero subgroup.
 A topological abelian group $A$ is said to be \emph{complete} if
the map~$\lambda_A$ is surjective.
 A topological abelian group $A$ is separated and complete if and only
if $\lambda_A$~is an isomorphism of topological abelian groups.

 It is clear from the definitions that the completion $\fA=A\sphat$
of any topological abelian group $A$ is a complete, separated
topological abelian group.

 Let $A$ be a topological abelian group and $K\subset A$ be a subgroup.
 Then the \emph{induced topology} on $K$ is defined by the rule that
the open subgroups in $K$ are the intersections $K\cap U\subset K$,
where $U\subset A$ ranges over the open subgroups of~$A$.

 Let $A$ be a topological abelian group and $p\:A\rarrow C$ be
a surjective homomorphism of abelian groups.
 Then the \emph{quotient topology} on $C$ is defined by the rule
that a subgroup $W\subset C$ is open in $C$ if and only if its
preimage $p^{-1}(W)\subset A$ is an open subgroup in~$A$.
 Given a surjective homomorphism of topological abelian groups
$p\:A\rarrow C$, the topology on $C$ is the quotient topology of
the topology on $A$ if and only if $p$~is a continuous open map.
 The latter condition means that, for any open subgroup $U\subset A$,
its image $p(U)\subset C$ is an open subgroup in~$C$.

 Let $A$ be a topological abelian group.
 A subgroup $K\subset A$ is closed in the topology of $A$ if and only
if the set-theoretical complement $A\setminus K$ is a union of cosets
with respect to open subgroups in~$A$.
 Equivalently, a subgroup in $A$ is closed if and only if it is
an intersection of open subgroups.
 In particular, $A$ is separated if and only if the zero subgroup is
closed in~$A$.
 The quotient group $A/K$ is separated in the quotient topology if
and only if $K$ is a closed subgroup in~$A$.
 The topological closure $\overline{K}_A\subset A$ of a subgroup
$K\subset A$ is a subgroup equal to the intersection of all the open
subgroups in $A$ containing~$K$.

 Given an injective homomorphism of topological abelian groups
$i\:K\rarrow A$, the subgroup $i(K)\subset A$ is closed in $A$
\emph{and} the topology on $K$ is induced by the topology of $A$
(via the embedding~$i$) if and only if $i$~is a continuous closed map.
 The latter condition means that, for any closed \emph{subset}
$Z\subset K$, its image $i(Z)\subset A$ is a closed subset in~$A$.
 It suffices to consider closed subsets of the form $Z=K\setminus W$,
where $W$ ranges over the open subgroups of~$K$.

 The following lemma explains the connection between closures
and completions.
 Notice that any subgroup of a separated abelian group is separated
in the induced topology.
 As a particular case of the lemma, one obtains the assertion that
any closed subgroup of a complete, separated abelian group is
complete in the induced topology.

\begin{lem} \label{closure-completion-lemma}
 Let\/ $\fB$ be a complete, separated topological abelian group,
and let $A\subset\fB$ be a subgroup, endowed with the induced
topology.
 Then the morphism of completions $A\sphat\,\rarrow\fB\sphat=\fB$
induced by the inclusion $A\rarrow\fB$ identifies the topological
abelian group $A\sphat\,$ with the closure $\overline{A}_\fB\subset
\fB$ of the subgroup $A$ in\/ $\fB$, where $\overline{A}_\fB$ is
endowed with its induced topology as a subgroup in\/~$\fB$.
\end{lem}

\begin{proof}
 The following assertions are straightforward to check from
the definitions:
\begin{enumerate}
\renewcommand{\theenumi}{\alph{enumi}}
\item for any topological abelian group $B$ and a subgroup
$A\subset B$ endowed with the induced topology, the map of
completions $A\sphat\,\rarrow B\sphat\,$ induced by
the injective morphism $A\rarrow B$ is also injective;
\item moreover, in the context of~(a), \,$A\sphat\,$ is a closed
subgroup in~$B\sphat\,$;
\item moreover, in the context of~(a), the topology on $A\sphat\,$
as the completion of $A$ coincides with the topology induced on
$A\sphat\,$ as a subgroup in $B\sphat\,$;
\item for any topological abelian group $A$, the image of
the completion morphism $\lambda_A\:A\rarrow A\sphat\,$ is
dense in $A\sphat\,$, that is, the closure of $\lambda_A(A)$
in $A\sphat\,$ coincides with $A\sphat\,$.
\end{enumerate}

 In the situation at hand, according to~(a), the map $A\sphat\,
\rarrow\fB$ is injective.
 Following~(b), \,$A\sphat\,$ is a closed subgroup in $\fB$; hence
$\overline{A}_\fB\subset A\sphat\,\subset\fB$.
 According to~(d), \,$A$ is dense in $A\sphat\,$; so
$A\sphat\,\subset\overline{A}_\fB\subset\fB$.
 Finally, (c)~tells us that the topology on $A\sphat\,$ coincides
with the one induced from~$\fB$.
 (Cf.~\cite[Lemma~2.1]{Pproperf}.)
\end{proof}

 Let $(A_i)_{i\in I}$ be a family of topological abelian groups.
 Then the \emph{product topology} (also called the \emph{Tychonoff
topology}) on the direct product $\prod_{i\in I} A_i$ of the abelian
groups $A_i$ is defined as follows.
 A base of open subgroups in $\prod_{i\in I} A_i$ is formed by
the subgroups of the form $\prod_{j\in J} U_j\times
\prod_{l\in I\setminus J}A_l\subset\prod_{i\in I}A_i$, where
$J\subset I$ is a finite subset of indices and $U_j\subset A_j$
are open subgroups in~$A_j$.

 Let $(A_\gamma)_{\gamma\in\Gamma}$ be a projective system of
topological abelian groups, indexed by a directed poset~$\Gamma$.
 Then the \emph{projective limit topology} on the projective limit
$\varprojlim_{\gamma\in\Gamma}A_\gamma$ of the abelian groups
$A_\gamma$ is defined as follows.
 A base of open subgroups in $\varprojlim_{\gamma\in\Gamma}A_\gamma$
is formed by the full preimages of open subgroups $U_\delta\subset
A_\delta$, \ $\delta\in\Gamma$, under the projection maps
$\varprojlim_{\gamma\in\Gamma}A_\gamma\rarrow A_\delta$.

 One can also construct the \emph{coproduct topology} on the direct
sum of abelian groups $A=\bigoplus_{i\in I} A_i$ in the following way.
 A subgroup $U\subset\bigoplus_{i\in I} A_i$ is open if and only if,
for every $j\in I$, the full preimage $U_j\subset A_j$ of the subgroup
$U$ under the natural inclusion map $A_j\rarrow\bigoplus_{i\in I} A_i$
is an open subgroup in~$A_j$.
 Equivalently, one can say that a base of open subgroups in
$\bigoplus_{i\in I}A_i$ is formed by the subgroups of the form
$\bigoplus_{i\in I}U_i\subset\bigoplus_{i\in I}A_i$, where
$U_i\subset A_i$ are open subgroups in~$A_i$.

 Denote by $\Top_\boZ$ the category of topological abelian groups and
continuous additive maps.
 As usually, we denote by $\Ab$ the category of abelian groups.
 Then $\Top_\boZ$ is an additive category with kernels and cokernels,
as well as set-indexed products and coproducts (so $\Top_\boZ$
actually has all set-indexed limits and colimits).
 The forgetful functor $\Top_\boZ\rarrow\Ab$ preserves all
the limits and colimits.

 Specifically, for any morphism $f\:A\rarrow B$ in $\Top_\boZ$,
the kernel of~$f$ in $\Top_\boZ$ is the kernel of~$f$ in $\Ab$
endowed with the induced topology as a subgroup in~$A$.
 The cokernel of~$f$ in $\Top_\boZ$ is the cokernel of~$f$ in $\Ab$
endowed with the quotient topology as an epimorphic image of~$B$.
 The products, projective limits, and coproducts in $\Top_\boZ$ are,
respectively, the products, projective limits, and coproducts in $\Ab$
endowed with the product, projective limit, and coproduct topologies,
as described above.

 Let $\Top^\s_\boZ$ denote the full subcategory of separated
topological abelian groups in $\Top_\boZ$.
 The full subcategory $\Top^\s_\boZ\subset\Top_\boZ$ is closed under
kernels and products (hence under all limits), and also under
coproducts (see below).
 However, $\Top^\s_\boZ$ is not closed under cokernels in $\Top_\boZ$.
 In fact, $\Top^\s_\boZ$ is a reflective full subcategory in
$\Top_\boZ$, i.~e., the inclusion functor $\Top^\s_\boZ\rarrow
\Top_\boZ$ has a left adjoint functor (the reflector).
 The latter functor assigns to any topological abelian group $A$ its
quotient group by the closure of zero subgroup in $A$, considered
in the quotient topology, that is $A\longmapsto A/\overline{\{0\}}_A$.

 The cokernels (and all colimits) in $\Top^\s_\boZ$ can be computed
by applying the reflector to the cokernel (resp., colimit) of
the same morphism/diagram taken in $\Top^\s_\boZ$.
 In other words, the cokernel of a morphism $f\:A\rarrow B$ in
$\Top^\s_\boZ$ is the quotient group $B/\overline{f(A)}_B$ of $B$
by the closure of the image of the morphism~$f$, where the quotient
group is endowed with the quotient topology.

 Let $\Top^\scc_\boZ$ denote the full subcategory of complete,
separated topological abelian groups in $\Top_\boZ$.
 The full subcategory $\Top^\scc_\boZ\subset\Top_\boZ$ is closed under
kernels and products (hence under all limits).
 In fact, $\Top^\scc_\boZ$ is a reflective full subcategory in
$\Top_\boZ$.
 The reflector (i.~e., the functor left adjoint to the inclusion
$\Top^\scc_\boZ\rarrow\Top_\boZ$) assigns to any topological abelian
group $A$ its completion $\fA=A\sphat\,$.
 The full subcategory $\Top^\scc_\boZ$ is also closed under coproducts
in $\Top_\boZ$, as the following lemma shows.

\begin{lem} \label{coproduct-topology-lemma}
\textup{(a)} The direct sum of a family of separated topological
abelian groups is separated in the coproduct topology. \par
\textup{(b)} The direct sum of a family of complete topological
abelian groups is complete in the coproduct topology.
\end{lem}

\begin{proof}
 Part~(a): let $(A_i)_{i\in I}$ be a family of separated topological
abelian groups.
 In order to show that the group $A=\bigoplus_{i\in I} A_i$ is
separated it suffices to find, for any nonzero element $a\in A$,
an open subgroup $U\subset A$ such that $a\notin u$.
 Indeed, the element~$a$ can be viewed as a collection of elements
$(a_i\in A_i)_{i\in I}$ such that $a_i=0$ for all but a finite
subset of indices $i\in I$.
 Since $a\ne0$, there exists an index $j\in I$ such that $a_j\ne0$.
 Since the topological group $A_j$ is separated, there exists an open
subgroup $U_j\subset A_j$ such that $a_j\notin U_j$.
 Now $U=U_j\oplus\bigoplus_{l\in I}^{l\ne j}A_l\subset
\bigoplus_{i\in I}A_i$ is an open subgroup such that $a\notin U$.

 Part~(b): we will consider the case of a family of complete,
separated topological abelian groups $A_i$; the general case can be
easily reduced to that.
 Put $A=\bigoplus_{i\in I}A_i$ and $\fA=A\sphat\,$.
 The projection map $A\rarrow A_i$ is continuous, so it induces
a continuous map $q_i\:\fA\rarrow A_i\sphat\,=A_i$.
 Let $b\in\fA$ be an element; we want to show that $q_i(b)=0$
for all but a finite subset of indices $i\in I$.

 Choose an open subgroup $U_i\subset A_i$ for every $i\in I$ such that
$U_i=A_i$ if $q_i(b)=0$ and $q_i(b)\notin U_i$ if $q_i(b)\ne0$.
 Then $U=\bigoplus_{i\in I}U_i$ is an open subgroup in
$\bigoplus_{i\in I}A_i$.
 Let $\fU=U\sphat\,$ be the related open subgroup in $\fA$,
that is, the kernel of the projection map $\fA\rarrow A/U$
(so $A/U\simeq\fA/\fU$).
 Then, viewing $A$ as a subgroup in $\fA$, we have $\fA=A+\fU$.
 Let $a\in A$ be an element such that $b-a\in\fU$.
 The subset $J\subset I$ of all $j\in I$ such that $q_j(a)\ne0$
is finite.
 The square diagram of projections $\fA\rarrow A/U\rarrow A_i/U_i$,
\ $\fA\rarrow A_i\rarrow A_i/U_i$ is commutative.
 Hence we have $q_i(\fU)\subset U_i\subset A_i$ and therefore
$q_i(b)-q_i(a)\in U_i$ for every $i\in I$. 
 For all $i\notin J$, it follows that $q_i(b)\in U_i$, thus
$q_i(b)=0$.

 We have shown that the collection of elements
$(q_i(b)\in A_i)_{i\in I}$ defines an element of
$A=\bigoplus_{i\in I}A_i$.
 Denote this element by $c\in A$.
 Let us show that $b=c$ in $\fA$.
 It suffices to check that, for every open subgroup $U\subset A$,
the image of $b-c$ under the projection map $\fA\rarrow A/U$
vanishes.
 Let $U_i\subset A_i$ denote the full preimage of $U$ under
the inclusion map $A_i\rarrow A$.
 Then $U'=\bigoplus_{i\in I} U_i\subset\bigoplus_{i\in I}A_i$ is
an open subgroup and $U'\subset U$.
 So it suffices to check that the image of $b-c$ in $A/U'=
\bigoplus_{i\in I}A_i/U_i$ vanishes.
 This follows from the fact that $q_i(b)=q_i(c)$ for every $i\in I$.
\end{proof}

 The cokernels (and all colimits) in $\Top^\scc_\boZ$ can be computed
by applying the completion functor $A\longmapsto A\sphat\,$ (i.e.,
the reflector) to the cokernel (resp., colimit) of the same
morphism/diagram computed in $\Top_\boZ$.
 In other words, the cokernel of a morphism $f\:A\rarrow B$ in
$\Top^\scc_\boZ$ is the completion $(B/f(A))\sphat\,$ of
the quotient group $B/f(A)$ in its quotient topology.
 The most important fact which we discuss in this paper is that
\emph{the quotient group of a complete, separated topological abelian
group by a closed subgroup need not be complete in the quotient
topology}.
 The construction of counterexamples, following~\cite{RD}
and~\cite{AGM}, will be presented in the next section.

\begin{rem}
 Let $\fA$ be a complete, separated topological abelian group, and
let $\fK\subset\fA$ be a closed subgroup.
 Denote by $Q=\fA/\fK$ the quotient group, endowed with the quotient
topology.
 What does it mean that the topological abelian group $Q$ is (or is
not) complete?

 Denote by $\fQ=Q\sphat\,$ the completion of the topological abelian
group~$Q$.
 By the definition, one has $\fQ=\varprojlim_{W\subset Q}Q/W$, where
$W$ ranges over all the open subgroups in~$Q$.
 Let $p\:\fA\rarrow Q$ denote the projection map.
 Then the map $W\longmapsto p^{-1}(W)=\fU$ is an ordered isomorphism
between the poset of all open subgroups $W\subset Q$ and the poset of
all open subgroups $\fU\subset\fA$ such that $\fK\subset\fU$.
 Thus the composition $\fA\rarrow Q\rarrow\fQ$ can be interpreted as
the natural map of projective limits
\begin{equation} \label{quotient-completion-map}
 \varprojlim\nolimits_{\fU\subset\fA}\fA/\fU\lrarrow
 \varprojlim\nolimits_{\fK\subset\fU\subset\fA}\fA/\fU.
\end{equation}
 Here the left-hand side is the projective limit of the quotient
groups $\fA/\fU$ taken over the poset of all open subgroups
$\fU\subset\fA$.
 The right-hand side is the projective limit of the same quotient
groups taken over the subposet of all open subgroups $\fU\subset\fA$
containing~$\fK$.

 Why should the map~\eqref{quotient-completion-map} be surjective?
 An element of the right-hand side is a compatible family of elements
$b=(\bar a_\fU\in\fA/\fU)_{\fK\subset\fU\subset\fA}$ specified for
all the open subgroups $\fU$ in $\fA$ containing~$\fK$.
 To find a preimage of~$b$ in the left-hand side means to extend
the given family of elements $\bar a_\fU$ to a similar family of
elements defined for all the open subgroups $\fU\subset\fA$ (where
$\fU$ no longer necessarily contains~$\fK$).
 How to produce the missing elements $\bar a_\fU$, simultaneously for
all $\fU\not\supset\fK$ and in a compatible way?
 There is no apparent way to do it, and indeed we will see in
the next section that this cannot be done (generally speaking).
\end{rem}

 A positive assertion in the desired direction holds under
a countability assumption.

\begin{prop} \label{countable-base-kernel}
 Let\/ $\fA$ be a complete, separated topological abelian group, and
let\/ $\fK\subset\fA$ be a closed subgroup.
 Assume that the topological abelian group\/ $\fK$ has a countable
base of neighborhoods of zero.
 Then the quotient group\/ $\fA/\fK$ is (separated and) complete in
the quotient topology.
\end{prop}

\begin{proof}
 This is a generalization of~\cite[Lemma~2.2(b)]{Pproperf}, provable
in a similar way with the following additional argument.
 Let $\Delta$ denote the poset of all open subgroups in $\fA$ with
respect to the inverse inclusion, and let $\Gamma$ be the similar
poset of all open subgroups in~$\fK$.
 Let $\psi\:\Delta\rarrow\Gamma$ be the map taking any open
subgroup $\fU\subset\fA$ to the open subgroup $\psi(\fU)=\fK\cap\fU
\subset\fK$.
 Then $\psi$ is a cofinal map of posets, and $\Gamma$ has a cofinal
countable subposet.
 Over a countable directed poset, the derived projective limit
$\varprojlim^1$ of any projective system of surjective maps of
abelian groups vanishes.

 We need to check that
$\varprojlim^1_{\fU\in\Delta}(\fK/(\fK\cap\fU))=0$.
 The key observation is that the inverse image with respect to
a cofinal map of directed poset preserves the derived projective limits
of projective systems of abelian groups, that is
$\varprojlim^n_{\gamma\in\Gamma} K_\gamma=
\varprojlim^n_{\delta\in\Delta} K_{\psi(\delta)}$ for any cofinal map
of directed posets $\psi\:\Delta\rarrow\Gamma$, every projective system
of abelian groups $(K_\gamma)_{\gamma\in\Gamma}$,  and all $n\ge0$.
 This is provable using the notion of a weakly flabby (faiblement
flasque) projective system, see~\cite[Th\'eor\`eme~1.8]{Jen}.
 Moreover, the derived projective limit of a directed projective system
in an abelian category with exact products is an invariant of
the related pro-object~\cite[Corollary~7.3.7]{Pros1}.
\end{proof}

\Section{The Construction of Counterexamples} \label{counterex-secn}

 This section is based on~\cite[Theorem~4.1.48]{AGM} (see
\cite[Proposition~11.1]{RD} for the original exposition).
 The exposition in~\cite{RD} and~\cite{AGM} is very general;
we specialize it to the particular case of abelian groups
with linear topology.
  In the context of functional analysis, for topological vector spaces
over the field of real numbers with its real topology, counterexamples
to completeness of quotients were found in \cite{Koeth},
\cite[\S\S19.5 and~31.6]{Koeth2} and mentioned
in~\cite[Exercise~IV.4.10(b)]{Bour}
and~\cite[Proposition~11.2]{Pal}.
 For another rather general counterexample,
see~\cite[Problem~D to Section~20]{KN}.

 The construction of the coproduct topology on a direct sum of
topological abelian groups in Section~\ref{top-abelian-secn} is
related to the \emph{box topology} on a product of topological
abelian groups (cf.~\cite[Proposition~4.1.46]{AGM}).
 Let $(A_i)_{i\in I}$ be a family of topological abelian groups.
 The box topology on the abelian group $\prod_{i\in I}A_i$ is
defined as the topology with a base of neighborhoods of zero consisting
of the subgroups of the form $\prod_{i\in I}U_i\subset
\prod_{i\in I}A_i$, where $U_i\subset A_i$
are open subgroups in the topology of~$A_i$.
 The coproduct topology on $\bigoplus_{i\in I}A_i$ is induced from
the box topology on $\prod_{i\in I}A_i$ via the natural embedding
$\bigoplus_{i\in I}A_i\rarrow\prod_{i\in I}A_i$.

\begin{lem} \label{box-topology-lemma}
\textup{(a)} The direct product of a family of separated topological
abelian groups is separated in the box topology. \par
\textup{(b)} The direct product of a family of complete topological
abelian groups is complete in the box topology.
\end{lem}

\begin{proof}
 Similar to, but much simpler than, the proof of
Lemma~\ref{coproduct-topology-lemma}.
 For a more general result (which is also an ``if and only if'' result)
applicable to both the Tychonoff and box topologies, as well as to
a family of intermediate topologies between these indexed by cardinal
numbers, see~\cite[Corollary~4.1.44 and Proposition~4.1.47]{AGM}.
\end{proof}

 Let $(A_i)_{i\in I}$ be a family of topological abelian groups.
 The \emph{modified box topology} on the abelian group
$\prod_{i\in I}A_i$ is defined as follows.
 A base of neighborhoods of zero in the modified box topology is
formed by the subgroups of the form $\prod_{j\in J}\{0\}\times
\prod_{l\in I\setminus J} U_l\subset\prod_{i\in I} A_i$,
where $J\subset I$ is a finite subset of indices and $U_l\subset A_l$
are open subgroups.
 For example, when the set of indices $I$ is finite, the modified
box topology on $\prod_{i\in I}A_i$ is discrete.

\begin{lem} \label{modified-box-topology-lemma}
 For any family of topological abelian groups $(A_i)_{i\in I}$,
the direct product of abelian groups\/ $\prod_{i\in I}A_i$ is
separated and complete in the modified box topology.
\end{lem}

\begin{proof}
 To show that the modified box topology on $\prod_{i\in I}A_i$
is separated, it suffices to check that the intersection of all
open subgroups is zero.
 Indeed, let $a=(a_i)_{i\in I}$ be a nonzero element in
$\prod_{i\in I}A_i$.
 Then there exists $j\in I$ such that $a_j\ne 0$ in~$A_j$.
 Put $J=\{j\}$ and $U_l=A_l$ for all $l\in I$, \,$l\ne j$.
 Then the open subgroup $\prod_{j\in J}\{0\}\times
\prod_{l\in I\setminus J}U_l\subset\prod_{i\in I}A_i$ does not
contain~$a$.

 To prove that the group $A=\prod_{i\in I}A_i$ is complete in
the modified box topology, notice that, for every $s\in I$,
the projection map $p_s\:\prod_{i\in I}A_i \rarrow A_s$ is continuous
with respect to the modified box topology on $\prod_{i\in I}A_i$
and the discrete topology on~$A_s$.
 Put $\fA=A\sphat\,$; then we have a continuous homomorphism of
the completions $q_s\:\fA\rarrow A_s$ induced by the projection
map~$p_s$.
 Given an element $b\in\fA$, consider the element
$c=(q_i(b))_{i\in I}\in A$.
 We need to show that $b=\lambda_A(c)$ in~$\fA$.
 
 For this purpose, it suffices to check that, for any open subgroup
$U\subset A$ belonging to a chosen base of open subgroups in~$A$,
the images of $b$ and~$c$ are equal in~$A/U$.
 We can assume that $U=\prod_{j\in J}\{0\}\times
\prod_{l\in I\setminus J} U_l$, where $J\subset I$ is a finite subset
and $U_l\subset A_l$ are open subgroups.
 Put $U_j=0$ for $j\in J$.
 Then the square diagram of projections $\fA\rarrow A/U\rarrow A_i/U_i$,
\ $\fA\rarrow A_i\rarrow A_i/U_i$ is commutative for every $i\in I$.
 So is the square diagram of projections $A\rarrow A/U\rarrow A_i/U_i$,
\ $A\rarrow A_i\rarrow A_i/U_i$.
 The images of both $b$ and~$c$ in $A_i/U_i$ are equal to
the coset $q_i(b)+U_i$.
 Hence the images of $b$ and~$c$ coincide in
$A/U=\prod_{i\in I}A_i/U_i$.
\end{proof}

\begin{lem} \label{modified-top-sum-product-lemma}
 For any family of separated topological abelian groups
$(A_i)_{i\in I}$, the subgroup\/ $\bigoplus_{i\in I}A_i\subset
\prod_{i\in I}A_i$ is closed in the modified box topology
on\/ $\prod_{i\in I}A_i$.
\end{lem}

\begin{proof}
 In fact, $\bigoplus_{i\in I}A_i$ is closed already in the box
topology on $\prod_{i\in I}A_i$ (hence in the modified box
topology, which is finer than the box topology).
 Cf.\ the middle paragraph of the proof of
Lemma~\ref{coproduct-topology-lemma}(b).
 An even stronger and more general result can be found
in~\cite[Proposition~4.1.46]{AGM}.
\end{proof}

 Let $(A_i)_{i\in I}$ be a family of topological abelian groups.
 The \emph{modified coproduct topology} on the direct sum of
abelian groups $\bigoplus_{i\in I}A_i$ is defined as follows.
 A base of neighborhoods of zero in the modified coproduct topology
is formed by the subgroups of the form $\bigoplus_{j\in J}\{0\}
\oplus\bigoplus_{l\in I\setminus J}U_l\subset\bigoplus_{i\in I}A_i$,
where $J\subset I$ is a finite subset of indices and $U_l\subset A_l$
are open subgroups.

\begin{cor} \label{modified-coproduct-cor}
 For any family of separated topological abelian groups
$(A_i)_{i\in I}$, the direct sum of abelian groups
$\bigoplus_{i\in I}A_i$ is separated and complete in the modified
coproduct topology.
\end{cor}

\begin{proof}
 By Lemma~\ref{modified-box-topology-lemma}, the direct product
$\prod_{i\in I}A_i$ is separated and complete in the modified box
topology.
 According to Lemma~\ref{modified-top-sum-product-lemma},
\,$\bigoplus_{i\in I}A_i$ is a closed subgroup of $\prod_{i\in I}A_i$
in the modified box topology on the product.
 It is clear from the definitions that the modified coproduct topology
on $\bigoplus_{i\in I}A_i$ is induced from the modified box topology
on $\prod_{i\in I}A_i$ via the embedding of abelian groups
$\bigoplus_{i\in I}A_i\rarrow\prod_{i\in I}A_i$.
 Now it remains to apply Lemma~\ref{closure-completion-lemma}.
\end{proof}

\begin{thm}[{\cite[Proposition~11.1]{RD}, \cite[Theorem~4.1.48]{AGM}}]
\label{top-abelian-main-theorem}
 Any separated topological abelian group can be obtained as
the quotient group of a complete, separated topological abelian
group by a closed subgroup, endowed with the quotient topology.
\end{thm}

\begin{proof}
 Let $Q$ be a separated topological abelian group.
 Choose an infinite set~$I$, and denote by $\fA_I(Q)$ the abelian
group $Q^{(I)}=\bigoplus_{i\in I}Q$, endowed with the modified
coproduct topology (of the direct sum of copies of~$Q$).
 Consider the summation map $\Sigma\:\fA_I(Q)=Q^{(I)}\rarrow Q$,
defined by the rule that, for every $i\in I$, the composition
$Q\rarrow Q^{(I)}\rarrow Q$ of the natural inclusion map
$\iota_i\:Q\rarrow Q^{(I)}$ with the map $\Sigma\:Q^{(I)}\rarrow Q$
is equal to the identity map, $\Sigma\circ\iota_i=\mathrm{id}_Q$.

 According to Corollary~\ref{modified-coproduct-cor}, the topological
abelian group $\fA_I(Q)$ is separated and complete.
 In order to show that the topology of $Q$ is the quotient topology
of the topology on $\fA_I(Q)$, it remains to check that $\Sigma$
is an open continuous map.

 To show that $\Sigma$ is continuous, consider an open subgroup
$W\subset Q$.
 Denote by $\fU\subset\fA_I(Q)$ the subgroup $\fU=W^{(I)}
\subset Q^{(I)}$.
 By the definition of the modified coproduct topology, the subgroup
$\fU$ is open in~$\fA$.
 Clearly, $\Sigma(\fU)\subset W$; so $\fU\subset\Sigma^{-1}(W)$.
 Hence the subgroup $\Sigma^{-1}(W)\subset\fA_I(Q)$ is open.

 To show that $\Sigma$ is open, it suffices to check that the subgroup
$f(\fU)\subset Q$ is open in $Q$ for any open subgroup $\fU\subset\
\fA_I(Q)$ belonging to a chosen base of open subgroups in $\fA_I(Q)$.
 We can assume that $\fU=\bigoplus_{j\in J}\{0\}\oplus
\bigoplus_{l\in I\setminus J} W_l\subset Q^{(I)}$, where $J\subset I$
is a finite subset and $W_l\subset Q_l$ are open subgroups.
 Then we have $\Sigma(\fU)=\sum_{l\in I\setminus J}W_l\subset Q$.
 Since the set $I$ is infinite by assumption, and consequently
the set $I\setminus J$ is nonempty, it follows that $\Sigma(\fU)$ is
an open subgroup in~$Q$.
\end{proof}

\Section{Vector Spaces with Linear Topology}  \label{top-vector-secn}

 Throughout the rest of the paper, we denote by~$k$ a fixed ground
field.
 A \emph{topological vector space} $V$ is a topological abelian group
with a $k$\+vector space structure such that the multiplication map
$k\times V\rarrow V$ is continuous (where the topology on~$k$ is
discrete).
 In other words, this means that the multiplication with any fixed
element of~$k$ is a continuous endomorphism of~$V$.

 A topological vector space $V$ is said to have \emph{linear topology}
if open $k$\+vector subspaces form a base of neighborhoods of zero
in~$V$.
 In this paper, all the ``topological vector spaces'' will be presumed
to have linear topology.

 All the basic theory of topological abelian groups developed in
Sections~\ref{top-abelian-secn}\+-\ref{counterex-secn} extends or
specializes \emph{verbatim} to topological vector spaces.
 In particular, the completion (of the underlying topological abelian
group) of a topological vector space is a topological vector space.
 One can speak of complete and/or separated topological vector spaces,
the induced topology on a vector subspace, the quotient topology on
a quotient vector space, closed vector subspaces and closures of
vector subspaces, and product and coproduct topologies for topological
vector spaces in the same way as for topological abelian groups.
 In the same way as for topological abelian groups, one defines
the box topology and the modified box topology on a direct product of
topological vector spaces, and the modified coproduct topology on
a direct sum.

 Similarly to Section~\ref{top-abelian-secn}, we denote by
$\Top_k$ the category of topological vector spaces and continuous
$k$\+linear maps.
 We also denote by $\Vect_k$ the category of (abstract,
nontopological) $k$\+vector spaces.
 The full subcategory of separated topological vector spaces is
denoted by $\Top^\s_k\subset\Top_k$, and the full subcategory of
complete, separated topological vector spaces by $\Top^\scc_k\subset
\Top^\s_k\subset\Top_k$.

 Then all the three categories $\Top_k$, $\Top^\s_k$, and $\Top^\scc_k$
are additive categories with set-indexed limits and colimits.
 The full subcategories $\Top^s_k$ and $\Top^\scc_k$ are reflective
in $\Top_k$.
 The forgetful functor $\Top_k\rarrow\Vect_k$ preserves all limits
and colimits.
 The forgetful functors $\Top_k\rarrow\Top_\boZ$, \
$\Top^\s_k\rarrow\Top^\s_\boZ$, and $\Top^\scc_k\rarrow\Top^\scc_\boZ$
preserve all limits and colimits, and commute with the reflectors.

 All the assertions and results of Sections~\ref{top-abelian-secn}\+-%
\ref{counterex-secn} remain valid in the topological vector
space setting.
 Without repeating all the details, we restrict ourselves to restating
Theorem~\ref{top-abelian-main-theorem} for topological vector spaces.

\begin{thm} \label{top-vector-main-theorem}
 Any separated topological vector space $Q$ can be obtained as
the quotient vector space of a complete, separated topological vector
space by a closed vector subspace (endowed with the quotient topology).
 Specifically, for any infinite set~$I$, the open, continuous
summation map\/ $\Sigma\:\fA_I(Q)\rarrow Q$ makes $Q$ a topological
quotient vector space of the complete, separated topological vector
space\/ $\fA_I(Q)=Q^{(I)}=\bigoplus_{i\in I}Q$ with the modified
coproduct topology.  \qed
\end{thm}

 In the rest of this section we mostly discuss topological vector spaces
with a \emph{countable} base of neighborhoods of zero.
 These form a special, well-behaved class of topological vector spaces
(cf.\ Proposition~\ref{countable-base-kernel}).

\begin{lem} \label{countable-is-product-of-discrete}
 A complete, separated topological vector space has a countable base of
neighborhoods of zero if and only if it is isomorphic to the product of
a countable family of discrete vector spaces, endowed with the product
topology.
\end{lem}

\begin{proof}
 It is clear from the definition of the product topology that
the product of a countable family of discrete vector spaces has
a countable base of neighborhoods of zero.
 Such topological vector spaces are also complete and separated,
since the full subcategory of complete, separated topological
vector spaces is closed under infinite products in $\Top_k$ and
contains the discrete vector spaces.

 Conversely, let $\fV$ be a complete, separated topological vector
space with a countable base of neighborhoods of zero.
 Since any countable directed poset has a cofinal subposet
isomorphic to the poset of natural numbers, one can choose
a descending sequence of open vector subspaces $\fV=\fU_0\supset
\fU_1\supset\fU_2\supset\dotsb$ such that the set of all
subspaces $\fU_n$, \,$n\ge0$, is a base of neighborhoods of zero
in~$\fV$.
 Considering $\fU_n\supset\fU_{n+1}$ as an abstract vector space
with a vector subspace, choose a complementary vector subspace
$V_n\subset\fU_n$; so $\fU_n=V_n\oplus\fU_{n+1}$ for every $n\ge0$.
 Endow the vector spaces $V_n$ with the discrete topology.
 Since $\fV$ is complete and separated, we have
$$
 \fV\simeq\varprojlim\nolimits_{n\ge1}\fV/\fU_n\simeq
 \varprojlim\nolimits_{n\ge1}\bigoplus\nolimits_{i=0}^n V_i\simeq
 \prod\nolimits_{i=0}^\infty V_i
$$
as an abstract vector space; and it is clear from the definitions
of a topology base and the product topology that the topologies
on $\fV$ and $\prod_{i=0}^\infty V_i$ agree.
\end{proof}

 The next proposition shows that complete, separated topological vector
spaces with a countable base of neighborhoods of zero have a rather
strong injectivity property in $\Top_k$.

\begin{prop} \label{countable-injectivity}
 Let $V$ be a topological vector space and $U\subset V$ be a vector
subspace, endowed with the induced topology.
 Let\/ $\fW$ be a complete, separated topological vector space with
a countable base of neighborhoods of zero.
 Then any continuous linear map $U\rarrow\fW$ can be extended to
a continuous linear map $V\rarrow\fW$.
\end{prop}

\begin{proof}
 According to Lemma~\ref{countable-is-product-of-discrete},
the topological vector space $\fW$ is isomorphic to a countable product
$\prod_{i=0}^\infty W_i$ of discrete vector spaces $W_i$, endowed with
the product topology.
 According to the discussion above in this section and in
Section~\ref{top-abelian-secn}, this means that the object $\fW$ is
the categorical product of the objects $W_i$ in the category $\Top_k$.
 Therefore, it suffices to consider the case of a discrete vector
space $\fW=W$ in order to prove the proposition.

 Let $f\:U\rarrow W$ be a continuous linear map.
 Since $W$ is discrete, the zero subspace $\{0\}\subset W$ is open
in $W$; hence the subspace $K=\ker(f)=f^{-1}(0)$ is open in~$U$.
 By the definition of the induced topology on $U$, there exists
an open vector subspace $L\subset V$ such that $U\cap L=K$.
 The linear map~$f$ factorizes through the surjection $U\rarrow U/K$;
so we have a linear map (of discrete vector spaces)
$\bar f\:U/K\rarrow W$.
 The vector space $U/K$ is a subspace in the (discrete)
vector space $V/L$; so the map~$\bar f$ can be extended to a linear
map $\bar g\:V/L\rarrow W$.
 Now the composition $V\rarrow V/L\rarrow W$ provides a desired
continuous linear extension $g\:V\rarrow W$ of the map~$f$.
\end{proof}

\begin{cor} \label{countable-base-complete-subspace-splits}
 Let $V$ be a topological vector space and\/ $\fU\subset V$ be a vector
subspace, endowed with the induced topology.
 Assume that the topological vector space $\fU$ is complete, separated,
and has a countable base of neighborhoods of zero.
 Then there exists a topological vector space $W$ and an isomorphism
of topological vector spaces $V\simeq\fU\oplus W$ forming
a commutative triangle diagram with the inclusion $\fU\rarrow V$
and the direct summand inclusion $\fU\rarrow\fU\oplus W$.
\end{cor}

\begin{proof}
 In any additive category with kernels or cokernels, any retract
is a direct summand.
 In the situation at hand, denote the inclusion map
by $\iota\:\fU\rarrow V$.
 By Proposition~\ref{countable-injectivity}, the identity map
$\fU\rarrow\fU$ can be extended to a continuous linear map
$g\:V\rarrow\fU$ such that $g\circ\iota=\mathrm{id}_\fU$.
 Now the topological vector space $W$ can be constructed as
the kernel of~$g$ or the cokernel of~$\iota$ in the category $\Top_k$.
\end{proof}

 A complete, separated topological vector space $\fV$ (with linear
topology) is called \emph{linearly compact} (or \emph{pseudo-compact}, or
\emph{profinite-dimensional}) if all its open subspaces $\fU\subset\fV$ 
have finite codimension in $\fV$, i.~e., the quotient spaces $\fV/\fU$
are finite-dimensional.
 For any discrete vector space $V$, the dual vector space $V^*=
\Hom_k(V,k)$ has a natural linearly compact topology in which
the open subspaces are the annihilators of finite-dimensional subspaces
in~$V$.
 The vector space $V$ can be recovered as the space of all continuous
linear functions $\fV\rarrow k$ on its dual vector space
$\fV=\Hom_k(V,K)$.
 The correspondence $V\longleftrightarrow\fV$ is an anti-equivalence
between the category of discrete vector spaces $\Vect_k$ and
the category of linearly compact vector spaces.
 Hence the category of linearly compact vector spaces is abelian.

 A topological vector space is linearly compact if and only if it is
isomorphic to a product $\prod_{i\in I}k$ of copies of
the one-dimensional discrete vector space~$k$ over some index set $I$,
endowed with the product topology.
 Following the proofs of Proposition~\ref{countable-injectivity}
and Corollary~\ref{countable-base-complete-subspace-splits}, one can see
that the former holds for an arbitrary linearly compact vector space
$\fW$, and the latter holds for any linearly compact vector space~$\fU$
(not necessarily with a countable base of neighborhoods of zero).

\Section{Exact Categories} \label{exact-categories-secn}

 Reasonably modern expositions on exact categories (in Quillen's
sense) include Keller's~\cite[Appendix~A]{Kel},
\cite[Appendix]{DRSS}, the present author's~\cite[Appendix~A]{Partin},
and B\"uhler's long paper~\cite{Bueh}.
 In this section we use some material from~\cite{Partin}.

 An \emph{exact category} $\sE$ is an additive category endowed with
a class of \emph{short exact sequences} (known also in the recent
literature as \emph{conflations}) $0\rarrow E'\rarrow E\rarrow E''
\rarrow0$.
 A morphism $E'\rarrow E$ appearing in a short exact sequence
$0\rarrow E'\rarrow E\rarrow E''\rarrow0$ is said to be
an \emph{admissible monomorphism} (or an \emph{inflation}), and
a morphism $E\rarrow E''$ appearing in such a short exact sequence is
said to be an \emph{admissible epimorphism} (or a \emph{deflation}).
 The following axioms must be satisfied:

\begin{itemize}
\item[Ex0:] The zero short sequence $0\rarrow 0\rarrow 0\rarrow 0
\rarrow0$ is exact.
Any short sequence isomorphic to a short exact sequence is exact.
\item[Ex1:] In any short exact sequence $0\rarrow E'\rarrow E\rarrow
E''\rarrow0$, the morphism $E'\rarrow E$ is the kernel of
the morphism $E\rarrow E''$, and the morphism $E\rarrow E''$ is
the cokernel of the morphism $E'\rarrow E$ in the category~$\sE$.
\item[Ex2:] Let $0\rarrow E'\rarrow E\rarrow E''\rarrow0$ be
a short exact sequence.  Then (a)~for any morphism $E'\rarrow F'$
there exists a commutative diagram
\begin{equation} \label{Ex2a-diagram}
\begin{gathered}
\xymatrix{
 0\ar[r] & E' \ar[r] \ar[d] & E \ar[r] \ar[d] & E'' \ar[r] & 0 \\
 0\ar[r] & F' \ar[r] & F \ar[ru]
}
\end{gathered}
\end{equation}
with a short exact sequence $0\rarrow F'\rarrow F\rarrow E''\rarrow0$;
and dually, (b)~for any morphism $F''\rarrow E''$ there exists
a commutative diagram
\begin{equation} \label{Ex2b-diagram}
\begin{gathered}
\xymatrix{
 0\ar[r] & E' \ar[r] \ar[rd] & E \ar[r] & E'' \ar[r] & 0 \\
 & & F \ar[u]\ar[r] & F'' \ar[u] \ar[r] & 0
}
\end{gathered}
\end{equation}
with a short exact sequence $0\rarrow E'\rarrow F\rarrow F''\rarrow0$.
\item[Ex3:] (a)~The composition of any two admissible monomorphisms
is an admissible monomorphism.  (b)~The composition of any two
admissible epimorphisms is an admissible epimorphism.
\end{itemize}

 The above formulation of axiom~Ex2 differs from the usual one, but
is actually equivalent to it under the assumption of Ex0--Ex1.
 The usual formulation is:
\begin{itemize}
\item[Ex2$'$:]
 (a)~For any admissible monomorphism $E'\rarrow E$ and any morphism
$E'\rarrow F'$ in $\sE$, there exists a pushout object
$F=F'\sqcup_{E'}E$ in $\sE$, and the natural morphism $F'\rarrow F$
is an admissible monomorphism (in other words, the class of
admissible monomorphisms is closed under pushouts). \\
 (b)~For any admissible epimorphism $E\rarrow E''$ and any morphism
$F''\rarrow E''$ in $\sE$, there exists a pullback object
$F=F''\sqcap_{E''}E$ in $\sE$, and the natural morphism $F\rarrow F''$
is an admissible epimorphism (in other words, the class of
admissible epimorphisms is closed under pullbacks).
\end{itemize}

 In fact, for any additive category $\sE$ with a class of short
exact sequences satisfying Ex0--Ex1 and Ex2(a), and for any commutative
diagram~\eqref{Ex2a-diagram} with short exact sequences $0\rarrow E'
\rarrow E\rarrow E''\rarrow0$ and $0\rarrow F'\rarrow F\rarrow E''
\rarrow0$, the commutative square $E'\rarrow E\rarrow F$, \ 
$E'\rarrow F'\rarrow F$ is a pushout
square~\cite[Proposition~A.2]{Partin}.
 So one has $F=F'\sqcup_{E'}E$, and Ex2$'$(a) follows.
 Conversely, to deduce Ex2(a) from Ex2$'$(a) under Ex0--Ex1, it suffices
to observe that $F=F'\sqcup_{E'}E$ implies $\coker(F'\to F)=
\coker(E'\to E)=E''$.

 Dually, for any additive category $\sE$ with a class of short
exact sequences satisfying Ex0--Ex1 and Ex2(b), and for any commutative
diagram~\eqref{Ex2b-diagram} with short exact sequences $0\rarrow E'
\rarrow E\rarrow E''\rarrow0$ and $0\rarrow E'\rarrow F\rarrow F''
\rarrow0$, the commutative square $F\rarrow E\rarrow E''$, \
$F\rarrow F''\rarrow E''$ is a pullback
square~\cite[Section~A.4]{Partin}.
 So one has $F=F''\sqcap_{E''}E$, and Ex2$'$(b) follows.
 Conversely, to deduce Ex2(b) from Ex2$'$(b) under Ex0--Ex1, it suffices
to notice that $F=F''\sqcap_{E''}E$ implies $\ker(F\to F'')=
\ker(E\to E'')=E'$.

 Moreover, assuming Ex0--Ex2, the short sequence $0\rarrow E'\rarrow
F'\oplus E\rarrow F\rarrow0$ is exact in the context of
Ex2(a) or Ex2$'$(a), and dually, the short sequence $0\rarrow F\rarrow E
\oplus F''\rarrow E''\rarrow0$ is exact in the context of
Ex2(b) or Ex2$'$(b) \cite[Appendix~A]{Kel}, \cite[Appendix]{DRSS}.

 The following property is known as the ``obscure axiom''
(see~\cite[Remark~2.17]{Bueh} for a historical discussion).
 It follows from Ex0--Ex2 (see~\cite[2nd step of the proof of
Proposition~A.1]{Kel} or~\cite[Proposition~2.16]{Bueh}).
\begin{itemize}
\item[OEx:] (a)~If a morphism~$g$ has a cokernel in $\sE$ and
the composition~$fg$ is an admissible monomorphism for some
morphism~$f$, then $g$~is an admissible monomorphism. \\
(b)~If a morphism~$f$ has a kernel in $\sE$ and the composition~$fg$
is an admissible epimorphism for some morphism~$g$, then
$f$~is an admissible epimorphism.
\end{itemize}

 An additive category $\sA$ is said to be \emph{idempotent-complete}
if, for every object $A\in\sA$ and an endomorphism $e\:A\rarrow A$
such that $e^2=e$, there exists a direct sum decomposition
$A\simeq B\oplus C$ in $\sA$ such that $e$~is the composition of
the direct summand projection and the direct summand inclusion
$A\rarrow B\rarrow A$.
 In other words, this means that any idempotent endomorphism $e\in\sA$
has a kernel, or equivalently, a cokernel, or equivalently, any
idempotent endomorphism $e\in\sA$ has an image.
 An additive category $\sA$ is said to be \emph{weakly
idempotent-complete} if, for every pair of objects $A$, $B\in\sA$
and pair of morphisms $p\:A\rarrow B$, \ $i\:B\rarrow A$ such that
$pi=\mathrm{id}_B$, there exists a direct sum decomposition
$A\simeq B\oplus C$ in $\sA$ such that $i$~is the direct summand
inclusion $B\rarrow A$ and $p$~is the direct summand projection
$A\rarrow B$.
 In other words, this means that the morphism~$p$ (i.~e., any retraction
in $\sA$) has a kernel, or equivalently, the morphism~$i$ (i.~e., any
section in $\sA$) has a cokernel.

 The assumption of weak idempotent completeness simplifies the theory
of exact categories considerably.
 In particular, the following axiom for a class of short exact sequences
in an additive category $\sE$ is equivalent, under the assumption of
Ex0--Ex1, to $\sE$ being weakly idempotent-complete (as an additive
category) \emph{and} satisfying Ex2 or~Ex2$'$ (as an additive
category with a class of short exact sequences):

\begin{itemize}
\item[Ex2$''$:] (a)~If a composition~$fg$ is an admissible monomorphism
in $\sE$, then $g$~is an admissible monomorphism. \\
(b)~If a composition~$fg$ is an admissible epimorphism in $\sE$, then
$f$~is an admissible epimorphism. \\
(c)~If in the commutative diagram
\begin{equation} \label{Ex2secondc-diagram}
\begin{gathered}
\xymatrix{
 0\ar[r] & E' \ar[r] \ar[rd] & E_1 \ar[r] \ar[d] & E'' \ar[r] & 0 \\
 & & E_2 \ar[ru]
}
\end{gathered}
\end{equation}
both the short sequences $0\rarrow E'\rarrow E_1\rarrow E''\rarrow0$
and $0\rarrow E'\rarrow E_2\rarrow E''\rarrow0$ are exact, then
the morphism $E_1\rarrow E_2$ is an isomorphism.
\end{itemize}

 Indeed, conditions~Ex2$''$(a\+-b) form a stronger (and simpler)
version of the obscure axiom~OEx.
 Any one of these two conditions for a class of short exact sequences
satisfying (a small part of) Ex0--Ex2 easily implies that the category
$\sE$ is weakly idempotent-complete.
 Conversely, in a weakly idempotent-complete category $\sE$,
axioms~Ex0--Ex1 and Ex2$'$ imply Ex2$''$(a\+-b)
\cite[Appendix, part~C]{DRSS}, \cite[Proposition~7.6]{Kel}.
 Condition~Ex2$''$(c) is a common particular case of Ex2(a) and Ex2(b).
 It is shown in~\cite[proof of Proposition~A.4]{Partin} that
Ex0--Ex1 and Ex2$''$ imply~Ex2.

 Finally, we should mention that, under the assumption of Ex0--Ex2,
the conditions Ex3(a) and Ex3(b) are equivalent to each other, so it
suffices to require only one of them~\cite[3rd step of the proof of
Proposition~A.1]{Kel} (see also~\cite[the discussion of diagram~(A.3)
at the end of Section~A.5]{Partin}).

\begin{ex} \label{inherited-exact-structure}
 Let $\sE$ be an exact category and $\sG\subset\sE$ be a full additive
subcategory.
 We will say that the full subcategory $\sG\subset\sE$ is \emph{closed
under extensions} if, for any short exact sequence $0\rarrow E'\rarrow
E\rarrow E''\rarrow0$, the object $E$ belongs to $\sG$ whenever both
the objects $E'$ and $E''$ do.
 Similarly, we will say that $\sG$ is \emph{closed under the kernels
of admissible epimorphisms} in $\sE$ if in the same short exact
sequence the object $E'$ belongs to $\sG$ whenever both the objects $E$
and $E''$ do; and $\sG$ is \emph{closed under the cokernels of
admissible monomorphisms} in $\sE$ if the object $E''$ belongs to $\sG$
whenever both the objects $E'$ and $E$ do.

 Finally, we will say that the full subcategory $\sG\subset\sE$
\emph{inherits the exact category structure} if the class of all
short sequences in $\sG$ that are exact in $\sE$ forms an exact
category structure on~$\sG$.
 Assume that the full subcategory $\sG$ is \emph{either} closed
under extensions, \emph{or} it is closed under both the kernels of
admissible epimorphisms and the cokernels of admissible monomorphisms
in~$\sE$.
 Then $\sG$ inherits the exact category structure from~$\sE$.

 Indeed, if $\sG$ is closed under extensions in $\sE$, then the axioms
Ex0--Ex2 are clearly satisfied for the class of all short sequences
in $\sG$ that are exact in~$\sE$.
 To see that the composition $A\rarrow B\rarrow C$ of two admissible
epimorphisms in $\sG$ is an admissible epimorphism in $\sE$, one
observes that the kernel of $A\rarrow C$ is an extension of
the kernels of $A\rarrow B$ and $B\rarrow C$ in~$\sE$.

 If $\sG$ is closed under the kernels of admissible epimorphisms and
the cokernels of admissible monomorphisms, then the axioms Ex0--Ex1
and Ex3 are clearly satisfied in~$\sG$.
 To check the axiom Ex2(a), consider a short exact sequence
$0\rarrow E'\rarrow E\rarrow E''\rarrow 0$ in $\sG$ and a morphism
$E'\rarrow F'$ in~$\sG$.
 Build a commutative diagram~\eqref{Ex2a-diagram} in $\sE$ with a short
exact sequence $0\rarrow F'\rarrow F\rarrow E''\rarrow0$, and recall
that the short sequence $0\rarrow E'\rarrow F'\oplus E\rarrow F\rarrow0$
is exact in~$\sE$.
 Since the objects $E'$ and $F'\oplus E$ belong to $\sG$ and $\sG$ is
closed under the cokernels of admissible monomorphisms in $\sE$, it
follows that $F\in\sG$.

 Given an exact category $\sE$, a full subcategory $\sG$ closed under
extensions in $\sE$, endowed with the exact structure inherited from
$\sE$, is called a \emph{fully exact subcategory} in
the surveys~\cite[Sections~4 and~12]{Kel2}, \cite[Sections~10.5
and~13.3]{Bueh} and the paper~\cite[Definition~2.2(ii)]{DS}.
 More generally, a full subcategory $\sG$ inheriting an exact category
structure from~$\sE$, endowed with the inherited exact category
structure, is simply called an \emph{exact subcategory}
in~\cite[Definition~2.2(i)]{DS}.
 A characterization of exact subcategories (or in our terminology,
full subcategories inheriting an exact category structure) can be found
in~\cite[Theorem~2.6]{DS} (see also~\cite[Lemma~4.20]{Pedg}).
\end{ex}

 Let $\sE$ and $\sG$ be exact categories.
 An additive functor $\Psi\:\sE\rarrow\sG$ is said to be \emph{exact}
if it takes short exact sequences to short exact sequences.

\begin{ex} \label{psi-exact-structure}
 The following construction provides a source of examples of exact
category structures that is important for the purposes of our
discussion.

 Let $\sE$ and $\sG$ be exact categories, and $\Psi\:\sE\rarrow\sG$
be an additive functor.
 We will say that $\Psi$ \emph{preserves the kernels of admissible
epimorphisms} if for every short exact sequence $0\rarrow E'\rarrow
E\rarrow E''\rarrow0$ in $\sE$ the morphism $\Psi(E')\rarrow\Psi(E)$
is the kernel of the morphism $\Psi(E)\rarrow\Psi(E'')$ in~$\sG$.
 Similarly, we say that $\Psi$ \emph{preserves the cokernels of
admissible monomorphisms} if for every short exact sequence
$0\rarrow E'\rarrow E\rarrow E''\rarrow0$ in $\sE$ the morphism
$\Psi(E)\rarrow\Psi(E'')$ is the cokernel of the morphism
$\Psi(E')\rarrow\Psi(E)$ in~$\sG$.

 Finally, we will say that a short sequence $0\rarrow E'\rarrow E
\rarrow E''\rarrow0$ in $\sE$ is \emph{$\Psi$\+exact} if
$0\rarrow E'\rarrow E\rarrow E''\rarrow0$ is a short exact sequence
in $\sE$ \emph{and} $0\rarrow\Psi(E')\rarrow\Psi(E)\rarrow\Psi(E'')
\rarrow0$ is a short exact sequence in~$\sG$.
 Assume that the functor $\Psi$ \emph{either} preserves the kernels
of admissible epimorphisms, \emph{or} it preserves the cokernels of
admissible monomorphisms.
 Then we claim that the class of all $\Psi$\+exact sequences in
$\sE$ is a (new) exact category structure on~$\sE$.
 We will call it the \emph{$\Psi$\+exact category structure}.
 Accordingly, we will speak of \emph{$\Psi$\+admissible monomorphisms}
and \emph{$\Psi$\+admissible epimorphisms} in~$\sE$.

 Indeed, without any assumptions on $\Psi$ it is clear that the class
of all $\Psi$\+exact sequences satisfies~Ex0--Ex1.
 Assuming that $\Psi$ preserves the kernels of admissible epimorphisms,
we will check~Ex2--Ex3.

 Ex2(a): given a $\Psi$\+exact sequence $0\rarrow E'\rarrow E\rarrow E''
\rarrow0$ and a morphism $E'\rarrow F'$ in $\sE$, we can apply
axiom~Ex2(a) or Ex2$'$(a) in the exact category $\sE$ and obtain
a commutative diagram~\eqref{Ex2a-diagram} with $F=F'\sqcup_{E'}E$
and a short exact sequence $0\rarrow F'\rarrow F\rarrow E''\rarrow0$.
 Since the functor $\Psi$ preserves the kernels of admissible
epimorphisms, the morphism $\Psi(F')\rarrow\Psi(F)$ is the kernel of
the morphism $\Psi(F)\rarrow\Psi(E'')$ in~$\sG$.
 So the morphism $\Psi(F)\rarrow\Psi(E'')$ has a kernel, and
the composition $\Psi(E)\rarrow\Psi(F)\rarrow\Psi(E'')$ is
an admissible epimorphism in~$\sG$.
 According to property~OEx(b) in the exact category~$\sG$,
it follows that $\Psi(F)\rarrow\Psi(E'')$ is an admissible
epimorphism in~$\sG$.
 Thus $0\rarrow\Psi(F')\rarrow\Psi(F)\rarrow\Psi(E'')\rarrow0$
is a short exact sequence in~$\sG$.
 By the definition, this means that the short sequence $0\rarrow F'
\rarrow F\rarrow E''\rarrow0$ is $\Psi$\+exact in~$\sE$, as desired.

 Ex2(b): given a $\Psi$\+exact sequence $0\rarrow E'\rarrow E\rarrow E''
\rarrow0$ and a morphism $F''\rarrow E''$ in $\sE$, we apply
axiom~Ex2(b) or~Ex2$'$(b) in $\sE$ and obtain a commutative
diagram~\eqref{Ex2b-diagram} with $F=F''\sqcap_{E''}E$ and
a short exact sequence $0\rarrow E'\rarrow F\rarrow F''\rarrow0$.
 According to the discussion above, the short sequence $0\rarrow F
\rarrow E\oplus F''\rarrow E''\rarrow0$ is exact in~$\sE$.
 Since $\Psi$ preserves the kernels of admissible epimorphisms, it
follows that the morphism $\Psi(F)\rarrow\Psi(E)\oplus\Psi(F'')$
is the kernel of the morphism $\Psi(E)\oplus\Psi(F'')\rarrow\Psi(E'')$
in~$\sG$.
 In other words, this means that $\Psi(F)=\Psi(F'')\sqcap_{\Psi(E'')}
\Psi(E)$.
 Now axiom Ex2$'$(b) in the exact category $\sG$ tells that
the short sequence $0\rarrow\Psi(E')\rarrow\Psi(F)\rarrow\Psi(F'')
\rarrow0$ is exact in~$\sG$.
 Thus the short sequence $0\rarrow E'\rarrow F\rarrow F''\rarrow0$
is $\Psi$\+exact in $\sE$.

 Ex3: following the above discussion, it suffices to check~Ex3(b).
 Let $f$ and~$g$ be a pair of composable $\Psi$\+admissible
epimorphisms in~$\sE$.
 By axiom~Ex3(b) in the exact category $\sE$, the composition
$fg\:E\rarrow E''$ is an admissible epimorphism in $\sE$; so we
have a short exact sequence $0\rarrow E'\rarrow E\rarrow E''\rarrow0$
in~$\sE$.
 Since $\Psi$ preserves the kernels of admissible epimorphisms,
the morphism $\Psi(E')\rarrow\Psi(E)$ is a kernel of
the morphism $\Psi(E)\rarrow\Psi(E'')$.
 By the definitions, the morphisms $\Psi(f)$ and $\Psi(g)$ are
admissible epimorphisms in~$\sG$; so by axiom~Ex3(b) in the category
$\sG$, the morphism $\Psi(fg)\:\Psi(E)\rarrow\Psi(E'')$ is
an admissible epimorphism in $\sG$, too.
 Hence the short sequence $0\rarrow\Psi(E')\rarrow\Psi(E)\rarrow
\Psi(E'')\rarrow0$ is exact in~$\sG$, the short sequence $0\rarrow
E'\rarrow E\rarrow E''\rarrow0$ is $\Psi$\+exact in $\sE$,
and the morphism $fg\:E\rarrow E''$ is a $\Psi$\+admissible
epimorphism.
\end{ex}

\Section{Quasi-Abelian Categories}  \label{quasi-abelian-secn}

 Quasi-abelian categories~\cite{Schn,Rum0} form the second most
well-behaved class of additive categories after the abelian ones.
 We refer to~\cite{Rum1,Rum2} for a historical and terminological
discussion of quasi-abelian categories.
 In the language of general category theory, the quasi-abelian
categories can be described as the regular and coregular additive
categories (in the sense of~\cite{BGO}).

 Let $\sA$ be an additive category.
 Throughout this section, we will assume that the kernels and
cokernels of all morphisms exist in $\sA$ (such additive categories
$\sA$ are nowadays called \emph{preabelian}).
 We will say that a morphism~$k$ in $\sA$ is \emph{a kernel} if
$k$~is a kernel of some morphism in $\sA$ (in this case, $k$~is
a kernel of its cokernel).
 Similarly, a morphism~$c$ in $\sA$ is said to be \emph{a cokernel}
if it is a cokernel of some morphism in~$\sA$.
 Obviously, any kernel is a monomorphism and any cokernel is
an epimorphism, but the converse need not be true.

 Given a morphism $f\:A\rarrow B$ in $\sA$ with a kernel
$K=\ker(f)$ and a cokernel $C=\coker(f)$, one denotes
the cokernel of the morphism $K\rarrow A$ by $\coker(K\to A)=\coim(f)$
and the kernel of the morphism $B\rarrow C$ by $\ker(B\to C)=\im(f)$.
 Then the original morphism $f\:A\rarrow B$ decomposes naturally as
$A\rarrow\coim(f)\rarrow\im(f)\rarrow B$, with a uniquely defined
morphism $\coim(f)\rarrow\im(f)$ in between.

\begin{prop} \label{four-properties-kernel-side-prop}
 For any additive category\/ $\sA$ with kernels and cokernels,
the following four conditions are equivalent:
\begin{enumerate}
\item for any morphism~$f$ in $\sA$, the natural morphism\/
$\coim(f)\rarrow\im(f)$ is an epimorphism;
\item the composition of any two kernels is a kernel in\/~$\sA$;
\item if a composition~$fg$ is a kernel in\/ $\sA$, then $g$~is
a kernel;
\item for any pair of morphisms $A'\rarrow A$ and $A'\rarrow B'$ in\/
$\sA$ such that $A'\rarrow A$ is a kernel, there exists a commutative
square
$$
\xymatrix{
 A' \ar[r] \ar[d] & A \ar[d] \\
 B' \ar[r] & B
}
$$
where $B'\rarrow B$ is a monomorphism.
\end{enumerate}
\end{prop}

\begin{proof}
 See~\cite[Example~(7) and ``Proof of~(7)'' in Section~A.5]{Partin}.

 The equivalence (1)\,$\Longleftrightarrow$\,(4) can be also found
in~\cite[Proposition~1]{Rum0}, and the implications
(1)\,$\Longrightarrow$\,(2), \ (1)\,$\Longrightarrow$\,(3)
in~\cite[Proposition~2]{Rum0}.
 The equivalences
(1)\,$\Longleftrightarrow$\,(2)\,$\Longleftrightarrow$\,(3) can be
found in~\cite[Proposition~3.1]{KW}.
 See~\cite[Section~3]{KW} for a historical discussion with references.
\end{proof}

\begin{prop} \label{four-properties-cokernel-side-prop}
 For any additive category\/ $\sA$ with kernels and cokernels,
the following four conditions are equivalent:
\begin{enumerate}
\item for any morphism~$f$ in $\sA$, the natural morphism\/
$\coim(f)\rarrow\im(f)$ is a monomorphism;
\item the composition of any two cokernels is a cokernel in\/~$\sA$;
\item if a composition~$fg$ is a cokernel in\/ $\sA$, then $f$~is
a cokernel;
\item for any pair of morphisms $A\rarrow A''$ and $B''\rarrow A''$ in\/
$\sA$ such that $A\rarrow A''$ is a cokernel, there exists a commutative
square
$$
\xymatrix{
 A \ar[r] & A'' \\
 B \ar[r] \ar[u] & B'' \ar[u]
}
$$
where $B\rarrow B''$ is an epimorphism.
\end{enumerate}
\end{prop}

\begin{proof}
 Dual to Proposition~\ref{four-properties-kernel-side-prop}.
\end{proof}

 Let us endow the additive category $\sA$ with the class of all
short sequences $0\rarrow A'\overset i\rarrow A\overset p\rarrow
A''\rarrow0$ satisfying~Ex1 (i.~e., all such sequences in which
$i$~is the kernel of~$p$ and $p$~is the cokernel of~$i$).
 An additive category $\sA$ with kernels and cokernels is said to be
\emph{quasi-abelian} if this (most obvious) class of short exact
sequences is an exact category structure on~$\sE$.

\begin{thm} \label{quasi-abelian-theorem}
 An additive category\/ $\sA$ with kernels and cokernels is
quasi-abelian if and only if the following three conditions hold:
\begin{enumerate}
\renewcommand{\theenumi}{\alph{enumi}}
\item $\sA$ satisfies the equivalent conditions of
Proposition~\ref{four-properties-kernel-side-prop};
\item $\sA$ satisfies the equivalent conditions of
Proposition~\ref{four-properties-cokernel-side-prop};
\item condition Ex2$''$(c) holds for the class of all short
sequences satisfying Ex1 in\/~$\sA$.
\end{enumerate}
 Furthermore, $\sA$ is quasi-abelian if and only if any pushout of
a kernel is a kernel and any pullback of a cokernel is a cokernel
in\/~$\sA$.
\end{thm}

\begin{proof}
 Clearly, all pushouts and pullbacks (as well as generally all finite
colimits and limits) exist in an additive category with cokernels and
kernels.
 So the condition that any pushout of a kernel is a kernel is
a reformulation of Ex2$'$(a), and the condition that any pullback
of a cokernel is a cokernel is a reformulation of Ex2$'$(b)
(for the class of all short sequences satisfying Ex1 in~$\sA$).

 Furthermore, the condition that any pushout of a kernel is a kernel
implies the condition of
Proposition~\ref{four-properties-kernel-side-prop}(4), which is
equivalent to~\ref{four-properties-kernel-side-prop}(2), which is
a restatement of Ex3(a).
 Similarly, the condition that any pullback of a cokernel is a cokernel
implies~\ref{four-properties-cokernel-side-prop}(4), which is
equivalent to~\ref{four-properties-cokernel-side-prop}(2), which is
a restatement of Ex3(b).
 Thus if the class of all short sequences satisfying Ex1 in $\sA$
satisfies Ex2$'$, then it also satisfies~Ex3.
 This proves the second assertion of the theorem.

 Finally, the condition of
Proposition~\ref{four-properties-kernel-side-prop}(3) is a restatement
of condition~Ex2$''$(a), and the condition of
Proposition~\ref{four-properties-cokernel-side-prop}(3) is a restatement
of condition~Ex2$''$(b).
 Since Ex2$''$ is equivalent to Ex2 for a weakly idempotent-complete
additive category with a class of short exact sequences satisfying
Ex0--Ex1 (see the discussion in Section~\ref{exact-categories-secn}),
the first assertion of the theorem follows.
\end{proof}

 In the terminology of~\cite{Rum0,Rum1}, an additive category $\sA$
with kernels and cokernels is said to be \emph{right semi-abelian}
if it satisfies the equivalent conditions of
Proposition~\ref{four-properties-kernel-side-prop}, and $\sA$ is
\emph{left semi-abelian} if it satisfies the equivalent conditions
Proposition~\ref{four-properties-cokernel-side-prop}.
 In the language of general category theory, an additive category
with kernels and cokernels is left semi-abelian if and only if it has
a (regular epimorphism, monomorphism)-factorization, and
right semi-abelian if and only if it has an (epimorphism, regular
monomorphism)-factorization.

 An additive category is called \emph{semi-abelian} if it is both
left and right semi-abelian.
 In other words, an additive category $\sA$ with kernels and cokernels
is semiabelian if and only if, for every morphism~$f$ in $\sA$,
the natural morphism $\coim(f)\rarrow\im(f)$ is both an epimorphism
and a monomorphism.

 In the terminology of~\cite{Rum1}, an additive category $\sA$ with
kernels and cokernels is \emph{right quasi-abelian} if any pushout
of a kernel is a kernel in $\sA$, and $\sA$ is \emph{left
quasi-abelian} if any pullback of a cokernel is a cokernel.
 In the language of general category theory~\cite{BGO}, an additive
category with kernels and cokernels is left quasi-abelian if and only
if it is regular, and right quasi-abelian if and only if it is
coregular.

 It is clear from Propositions~\ref{four-properties-kernel-side-prop}(4)
and~\ref{four-properties-cokernel-side-prop}(4) that any right
quasi-abelian category is right semi-abelian, and any left
quasi-abelian category is left semi-abelian~\cite[Corollary~1]{Rum0}.
 By the second assertion of Theorem~\ref{quasi-abelian-theorem},
a category is quasi-abelian if and only if it is both left and right
quasi-abelian.

 The following result can be found in~\cite[Proposition~3]{Rum0}.

\begin{cor} \label{quasi-modulo-semi}
 An additive category is quasi-abelian if and only if it is
right quasi-abelian and left semi-abelian, or equivalently, if and
only if it is right semi-abelian and left quasi-abelian.
 In other words, a semi-abelian category is right quasi-abelian
if and only if it is left quasi-abelian.
\end{cor}

\begin{proof}
 Following the discussion in Section~\ref{exact-categories-secn}
(or more specifically, \cite[Proposition~A.2 in Section~A.4]{Partin}),
any one of the conditions Ex2$'$(a) or Ex2$'$(b) implies Ex2$''$(c).
 So the condition of Theorem~\ref{quasi-abelian-theorem}(c) holds
for any additive category that is either right or left
quasi-abelian, and the corollary follows from the first assertion of
Theorem~\ref{quasi-abelian-theorem}.
\end{proof}

 The question whether all semi-abelian categories are quasi-abelian
came to be known as the \emph{Raikov problem} or \emph{Raikov
conjecture}, with the reference to Raikov's papers~\cite{Rai0,Rai1}.
 The conjecture was explicitly formulated in the paper~\cite{KCh}.
 The condition of our Theorem~\ref{quasi-abelian-theorem}(c) explicitly
appears in the papers~\cite{Rai0,KCh,Rai1} as one of the axioms.

 Counterexamples to Raikov's conjecture were discovered in
the papers~\cite{BD,Rum1,Rum2,SW,Wen}.
 The example in~\cite{Rum1} has algebraic (representation-theoretic)
flavour, while the examples in~\cite{BD,Rum2,SW,Wen} come from
functional analysis (the theory of locally convex and bornological
spaces).
 A historical discussion can be found in~\cite{Wen}.
 A very simple algebraic counterexample was also given
in~\cite[Example~A.5 in Section~A.4]{Partin}.

\begin{ex}[{\cite[Example~A.5]{Partin}}]
 Let $\sA$ be the additive category whose objects are morphisms of
(e.~g., finite dimensional) $k$\+vector spaces $f\:V''\rarrow V'$
endowed with the following additional datum: a vector subspace
$V\subset\im(f)$ is chosen in the vector space~$\im(f)$.
 We will denote the objects of $\sA$ by the symbols
$(V''\underset V{\overset f\rarrow} V')$.
 Morphisms in $\sA$ are defined in the obvious way.

 Then the forgetful functor from $\sA$ to the abelian category of
morphisms of vector spaces $f\:V''\rarrow V'$ (forgetting
the additional datum) is faithful and preserves kernels
and cokernels.
 Hence it follows easily that the category $\sA$ is semi-abelian,
and moreover, any pushout of a monomorphism in $\sA$ is a monomorphism
and any pullback of an epimorphism is an epimorphism (so the category
$\sA$ is \emph{integral} in the sense of~\cite{Rum0,HSW}).
 To see that $\sA$ is not quasi-abelian, it suffices to consider
the diagram
$$
\xymatrix{
 {(0\underset0{\overset0\rarrow}k)} \ar[r] \ar[rd]
 & {(k\underset0{\overset{\mathrm{id}}\rarrow}k)} \ar[r] \ar[d]
 & {(k\underset0{\overset0\rarrow}0)} \\
 & {(k\underset k{\overset{\mathrm{id}}\rarrow}k)} \ar[ru]
}
$$
showing that $\sA$ does not satisfy the condition of
Theorem~\ref{quasi-abelian-theorem}(c).
\end{ex}

 A simple algebraic example of an additive category which
is right quasi-abelian but not left semi-abelian can be found
in~\cite[Example~4.2]{LPRV} (while~\cite[Example~4.1]{LPRV}
is a left quasi-abelian but not right semi-abelian additive
category).
 In fact, both of these are examples of \emph{balanced} additive
categories with kernels and cokernels which are not abelian
(where ``balanced'' means that any morphism which is
simultaneously a monomorphism and an epimorphism is
an isomorphism).

 For an example of a left quasi-abelian but not right semi-abelian
additive category arising in functional analysis,
see~\cite[Example~4.2]{SW} and~\cite[Example~4.1]{KW}.

 The category of complete, separated topological vector spaces with
linear topology is right quasi-abelian but not left
semi-abelian, as we will see below in
Section~\ref{maximal-exact-VSLTs-secn}.
 For an analytic version, see~\cite[Propositions~4.1.6(b)
and~4.1.14]{Pros2} or~\cite[Example~4.2]{KW}.

\Section{Maximal Exact Structures} \label{maximal-exact-struct-secn}

 Let $\sA$ be a fixed additive category.
 Then exact category structures on $\sA$ are ordered by inclusion
(of their classes of short exact sequences).
 Clearly, there is a \emph{minimal} or \emph{split exact category
structure}, in which only the split short sequences $0\rarrow A'
\rarrow A'\oplus A''\rarrow A''\rarrow0$ are considered to be exact.

 It was shown in the paper~\cite{SW} that any additive category with
kernels and cokernels admits a \emph{maximal} exact category structure.
 This result was generalized to weakly idempotent-complete additive
categories in the paper~\cite{Cr}.
 A maximal exact category structure on an arbitrary additive category
was constructed in~\cite{Rum3}.

 In this section, we mostly follow~\cite{Cr}.
 So we assume $\sA$ to be a weakly idempotent-complete additive
category (see Section~\ref{exact-categories-secn}).
 As in Section~\ref{quasi-abelian-secn}, a morphism in $\sA$ is said
to be \emph{a} (\emph{co})\emph{kernel} if it is a (co)kernel of some
morphism in~$\sA$.
 If a morphism $i$ is a kernel and has a cokernel, then $i$~is
the kernel of its cokernel.
 Dually, if $p$~is a cokernel and has a kernel, then $p$~is
the cokernel of its kernel.

 A morphism $A'\rarrow A$ in $\sA$ is said to be
a \emph{semi-stable kernel}~\cite{Cr} if, for every morphism
$A'\rarrow B'$ in $\sA$, there exists a pushout object
$B=B'\sqcup_{A'}A$, and the natural morphism $B'\rarrow B$ is
a kernel in~$\sA$.
 Taking $B'=0$, one can see that any semi-stable kernel has
a cokernel.

 Dually, a morphism $A\rarrow A''$ in $\sA$ is said to be
a \emph{semi-stable cokernel} if, for every morphism $C''\rarrow A''$
in $\sA$, there exists a pullback object $C=C''\sqcap_{A''}A$, and
the natural morphism $C\rarrow C''$ is a cokernel in~$\sA$.
 Taking $C''=0$, one can see that any semi-stable cokernel has
a kernel.

\begin{prop} \label{semi-stable-cokernels-prop}
 In any weakly idempotent-complete additive category\/ $\sA$,
the class of all semi-stable cokernels has the following
properties: \par
\textup{(a)} any pullback of a semi-stable cokernel is a semi-stable
cokernel; \par
\textup{(b)} the direct sum of two semi-stable cokernels is
a semi-stable cokernel; \par
\textup{(c)} the composition of two semi-stable cokernels is
a semi-stable cokernel; \par
\textup{(d)} if the composition~$fg$ is a semi-stable cokernel,
then $f$~is a semi-stable cokernel.
\end{prop}

\begin{proof}
 Part~(a) is easy (see~\cite[Lemma~2.2]{Cr}).
 Part~(b) is~\cite[Lemma~3.2]{Cr}.
 Part~(c) is~\cite[Proposition~3.1]{Cr}.
 Part~(d) is~\cite[Proposition~3.4]{Cr}.
\end{proof}

\begin{prop} \label{semi-stable-kernels-prop}
 In any weakly idempotent-complete additive category\/ $\sA$,
the class of all semi-stable kernels has the following
properties: \par
\textup{(a)} any pushout of a semi-stable kernel is a semi-stable
kernel; \par
\textup{(b)} the direct sum of two semi-stable kernels is
a semi-stable kernel; \par
\textup{(c)} the composition of two semi-stable kernels is
a semi-stable kernel; \par
\textup{(d)} if the composition~$fg$ is a semi-stable kernel,
then $g$~is a semi-stable kernel.
\end{prop}

\begin{proof}
 Dual to Proposition~\ref{semi-stable-cokernels-prop}.
\end{proof}

 A short sequence $0\rarrow A'\overset i\rarrow A\overset p\rarrow
A''\rarrow0$ in $\sA$ is said to be \emph{stable exact} if
the morphism~$i$ is a kernel of the morphism~$p$, the morphism~$p$
is a cokernel of the morphism~$i$, the morphism~$i$ is a semi-stable
kernel, and the morphism~$p$ is a semi-stable cokernel.
 In this case, $i$~is said to be a \emph{stable kernel} and $p$~is
said to be a \emph{stable cokernel}.
 In other words, a stable kernel is defined as a semi-stable kernel
whose cokernel is semi-stable; a stable cokernel is defined as
a semi-stable cokernel whose kernel is semi-stable.

\begin{thm} \label{maximal-exact-theorem}
 For any weakly idempotent-complete additive category\/ $\sA$,
the class of all stable short exact sequences in\/ $\sA$
is an exact category structure on\/~$\sA$.
 This is the maximal exact category structure on\/ $\sA$
(i.~e., in any exact category structure on\/ $\sA$, all
short exact sequences are stable).
\end{thm}

\begin{proof}
 The second assertion follows immediately from axioms Ex1 and~Ex2$'$.
 Both the assertions are the result of~\cite[Theorem~3.5]{Cr}.
\end{proof}

\begin{rem}
 Theorem~\ref{maximal-exact-theorem} is surprising in the following
aspect, which was overlooked in~\cite[Example~(8) in
Section~A.5]{Partin}.
 Let $0\rarrow A'\overset i\rarrow A\overset p\rarrow A''\rarrow0$
be a stable short exact sequence in~$\sA$.
 So, for any morphism $A'\rarrow B'$, the pushout object
$B=B'\sqcup_{A'}A$ exists, and the morphism $j\:B'\rarrow B$ is
a (semi-stable) kernel; and dually, for any morphism $C''\rarrow B''$,
the pullback object $C=C''\sqcap_{A''}A$ exists, and the morphism
$q\:C\rarrow C''$ is a (semi-stable) cokernel.

 Why does the morphism~$j$ have to be a stable kernel, and why
does the morphism~$q$ have to be a stable cokernel?
 In other words, why does a pushout object $D=B'\sqcup_{A'}C$ exist,
or why does a pullback object $D=C''\sqcap_{A''}B$ exist?
 (One can see that it must be the same object, appearing in
the same stable short exact sequence $0\rarrow B'\rarrow D\rarrow C''
\rarrow0$, as the pushouts and the pullbacks of short exact sequences
in an exact category commute with each other.)

 Here is why.
 Consider the diagonal morphism $C''\oplus B\rarrow A''$, and
let $E=(C''\oplus B)\sqcap_{A''}A$ be the related pullback (which exists
since the morphism~$p$ is a semi-stable cokernel).
 The morphism $p$~factorizes as $A\rarrow C''\oplus B\rarrow A''$,
and consequently there is a morphism $A\rarrow E$ such that
the composition $A\rarrow E\rarrow A$ is the identity morphism.
 Since the category $\sA$ is weakly idempotent-complete by assumption,
it follows that $E=A\oplus D$, where $D$ is the kernel of $E\rarrow A$
or the cokernel of $A\rarrow E$.
 This object $D$ is the kernel of the morphism $C''\oplus B\rarrow A''$,
hence it is the desired pullback $D=C''\sqcap_{A''}B$.
\end{rem}

\begin{ex}
 Let $\sA$ be a quasi-abelian additive category
(see Section~\ref{quasi-abelian-secn}).
 Then all the kernels and cokernels in $\sA$ are semi-stable, and
consequently all of them are stable.
 The quasi-abelian exact category structure on $\sA$ (that is,
the class of all short sequences satisfying~Ex1) is its maximal
exact category structure.
\end{ex}

\begin{ex} \label{right-quasi-abelian-semi-stable-example}
 Let $\sA$ be a right quasi-abelian additive category.
 By the definition, it means that all the morphisms in $\sA$ have
kernels and cokernels, and all kernels in $\sA$ are semi-stable.
 It follows that all the semi-stable cokernels are stable.

 If $\sA$ is right quasi-abelian but not left quasi-abelian
(equivalently, right quasi-abelian but not left semi-abelian), then
there exist cokernels in $\sA$ that are not semi-stable cokernels.
 The kernels of such cokernels are semi-stable but not stable
kernels.

 The maximal exact category structure on $\sA$ consists of all
the short sequences $0\rarrow A'\rarrow A\rarrow A''\rarrow 0$
satisfying Ex1 such that the morphism $A\rarrow A''$ is
a semi-stable cokernel, or equivalently, the morphism $A'\rarrow A$
is a stable kernel.
\end{ex}

\Section{Categories of Incomplete VSLTs are Quasi-Abelian}
\label{incomplete-VSLTs-secn}

 It is claimed in~\cite[page~1, Section~1.1]{Beil} that the category
$\Top^\scc_k$ is quasi-abelian, and in particular it has an exact
category structure in which the admissible monomorphisms are
the closed embeddings and the admissible epimorphisms are
the open surjections.
 These assertions are not true.
 In the next Section~\ref{maximal-exact-VSLTs-secn} we will explain why.

 Surprisingly, the category $\Top_k$ of arbitrary (not necessarily
complete or separated) topological vector spaces with linear topology
has better exactness properties; in fact, it is quasi-abelian.
 So is the category $\Top^\s_k$ of separated, but not necessarily
complete topological vector spaces (with linear topology).

\begin{thm} \label{nonseparated-top-groups-spaces-theorem}
\textup{(a)} The category\/ $\Top_\boZ$ of (not necessarily complete
or separated) topological abelian groups with linear topology is
quasi-abelian. \par
 A continuous homomorphism $i\:K\rarrow A$ is a kernel (i.~e.,
an admissible monomorphism in the quasi-abelian exact structure)
in\/ $\Top_\boZ$ if and only if $i$~is injective and the topology
of $K$ is induced from the topology of $A$ (on $K$ viewed as a subgroup
in $A$ via~$i$). \par
 A continuous homomorphism $p\:A\rarrow C$ is a cokernel (i.~e.,
an admissible epimorphism in the quasi-abelian exact structure)
in\/ $\Top_\boZ$ if and only if $p$~is surjective and the topology
of $C$ is the quotient topology of the topology of $A$ (via~$p$);
in other words, this means that $p$~is a surjective open map. \par
\textup{(b)} The category\/ $\Top_k$ of (not necessarily complete
or separated) topological vector spaces with linear topology is
quasi-abelian. \par
 A continuous linear map $i\:K\rarrow V$ is a kernel (i.~e.,
an admissible monomorphism in the quasi-abelian exact structure)
in\/ $\Top_k$ if and only if $i$~is injective and the topology
of $K$ is induced from the topology of $V$ (on $K$ viewed as a subspace
in $V$ via~$i$). \par
 A continuous linear map $p\:V\rarrow C$ is a cokernel (i.~e.,
an admissible epimorphism in the quasi-abelian exact structure)
in\/ $\Top_k$ if and only if $p$~is surjective and the topology
of $C$ is the quotient topology of the topology of $V$ (via~$p$);
in other words, this means that $p$~is a surjective open map.
\end{thm}
 
\begin{proof}
 This result is well-known; see~\cite[Proposition~2.6]{Pros0}
or~\cite[Section~2.2]{Rum0}.
 Let us explain part~(a); part~(b) is similar.

 Following the discussion in Section~\ref{top-abelian-secn},
the forgetful functor $\Top_\boZ\rarrow\Ab$ preserves kernels and
cokernels (as well as all limits and colimits).
 As this functor is also faithful, it follows immediately that,
for any morphism~$f$ in $\Top_\boZ$, the induced morphism
$\coim(f)\rarrow\im(f)$ is an epimorphism and a monomorphism
(in fact, $\coim(f)\rarrow\im(f)$ is a bijective map).
 So the category $\Top_\boZ$ is semi-abelian.
 Furthermore, the monomorphisms in $\Top_\boZ$ are the injective
continuous homomorphisms, and the epimorphisms are the surjective
continuous homomorphisms.

 Clearly, if $i\:K\rarrow A$ is a kernel in $\Top_\boZ$, then
the topology of $K$ is induced from the topology of $A$ via~$i$.
 Conversely, if $i$ is an injective map and the topology of $K$ is
induced from the topology of~$A$, then $i$~is a kernel of
the continuous homomorphism $A\rarrow A/i(K)$ (where $A/i(K)$ is
endowed with the quotient topology).
 Similarly, if $p\:A\rarrow C$ is a cokernel in $\Top_\boZ$,
then the topology of $C$ is the quotient topology of the topology
of~$A$.
 Conversely, if $p$~is a surjective map and the topology of $C$ is
the quotient topology of the topology of $A$, then $p$~is a cokernel
of the continuous homomorphism $\ker(p)\rarrow A$ (where $\ker(p)
\subset A$ is endowed with the induced topology).

 According to Corollary~\ref{quasi-modulo-semi}, in order to show
that a semi-abelian category is quasi-abelian, it suffices to
check that it is \emph{either} right \emph{or} left quasi-abelian.
 For the sake of clarity and completeness, let us check both
the properties.

 Let $i\:K\rarrow A$ be an injective continuous homomorphism such
that the topology of $K$ is induced from the topology of $K$ via~$i$,
and let $f\:K\rarrow L$ be an arbitrary continuous homomorphism.
 Then the pushout $B=L\sqcup_KA$ is the abelian group $B=
\coker((-f,i))=(L\oplus A)/K$ endowed with the quotient topology.
 The natural map $j\:L\rarrow B$ is injective since the map~$i$ is.
 We need to check that the topology of $L$ is induced from
the topology of $B$ via~$j$.

 Let $E\subset L$ be an open subgroup.
 Then the preimage $f^{-1}(L)\subset K$ is an open subgroup, too.
 As the topology of $K$ is induced from $A$, this means that
there exists an open subgroup $U\subset A$ such that $f^{-1}(E)
=i^{-1}(U)$.
 Now the image of the open subgroup $E\oplus U\subset L\oplus A$
under the surjective map $L\oplus A\rarrow B$ is an open subgroup
$W\subset B$ (by the definition of the quotient topology).
 The preimage $j^{-1}(W)$ is the subgroup of all elements $l\in L$
for which there exist $e\in E$, \,$u\in U$, and $k\in K$ such that
$l=e+f(k)$ in $L$ and $u-i(k)=0$ in~$A$.
 We have $k\in i^{-1}(U)$, hence $f(k)\in E$; so $l\in E$, and
we have shown that $j^{-1}(W)=E$.

 Let $p\:A\rarrow C$ be an open continuous surjective homomorphism,
and let $f\:D\rarrow C$ be an arbitrary continuous homomorphism.
 Then the pullback $B=D\sqcap_CA$ is the subgroup $B=\ker((-f,p))
\subset D\oplus A$ endowed with the induced topology.
 The projection map $q\:B\rarrow D$ is surjective since the map~$p$
is.
 We need to check that $q$~is an open map.
 A base of open subgroups in $B$ is formed by the subgroups
$W=B\cap(E\oplus U)$, where $E\subset D$ and $U\subset A$ are
open subgroups.
 Now the subgroup $q(W)\subset D$ consists of all elements
$d\in D$ such that $d\in E$ and there exists $u\in U$ for which
$f(d)=p(u)$.
 So $q(W)=E\cap f^{-1}(p(U))$, which is an open subgroup in $D$
since the map~$p$ is open and the map~$f$ is continuous.
\end{proof}

\begin{thm} \label{incomplete-top-groups-spaces-theorem}
\textup{(a)} The category\/ $\Top^\s_\boZ$ of separated (but not
necessarily complete) topological abelian groups with linear topology
is quasi-abelian. \par
 A continuous homomorphism $i\:K\rarrow A$ is a kernel (i.~e.,
an admissible monomorphism in the quasi-abelian exact structure)
in\/ $\Top^\s_\boZ$ if and only if $i$~is injective, the subgroup
$i(K)$ is closed in $A$, and the topology of $K$ is induced from
the topology of $A$ (via~$i$); in other words, this means that
$i$~is an injective closed map. \par
 A continuous homomorphism $p\:A\rarrow C$ is a cokernel (i.~e.,
an admissible epimorphism in the quasi-abelian exact structure)
in\/ $\Top^\s_\boZ$ if and only if $p$~is surjective and the topology
of $C$ is the quotient topology of the topology of $A$ (via~$p$);
in other words, this means that $p$~is a surjective open map. \par
\textup{(b)} The category\/ $\Top^\s_k$ of separated (but not
necessarily complete) topological vector spaces with linear topology
is quasi-abelian. \par
 A continuous linear map $i\:K\rarrow V$ is a kernel (i.~e.,
an admissible monomorphism in the quasi-abelian exact structure)
in\/ $\Top^\s_k$ if and only if $i$~is injective, the subspace
$i(K)$ is closed in $V$, and the topology of $K$ is induced from
the topology of $V$ (via~$i$); in other words, this means that
$i$~is an injective closed map. \par
 A continuous linear map $p\:V\rarrow C$ is a cokernel (i.~e.,
an admissible epimorphism in the quasi-abelian exact structure)
in\/ $\Top^\s_k$ if and only if $p$~is surjective and the topology
of $C$ is the quotient topology of the topology of $V$ (via~$p$);
in other words, this means that $p$~is a surjective open map.
\end{thm}

\begin{proof}
 This is also well-known; see~\cite[Section~2.2]{Rum0},
cf.~\cite[Proposition~3.1.8]{Pros2}.
 We will explain part~(a); part~(b) is similar.

 According to the discussion in Section~\ref{top-abelian-secn},
the inclusion functor $\Top^\s_\boZ\rarrow\Top_\boZ$ preserves
kernels (as well as all limits).
 The cokernel of a morphism $f\:A\rarrow B$ in $\Top^\s_\boZ$ is
computed as the maximal separated quotient group $C/\overline{\{0\}}_C$
of the cokernel $C=B/f(A)$ of the morphism~$f$ in the category
$\Top_\boZ$ (with the quotient topology on $C/\overline{\{0\}}_C$).
 Equivalently, the cokernel of~$f$ in $\Top^\s_\boZ$ is
the quotient group $B/\overline{f(A)}_B$, endowed with the quotient
topology.

 Clearly, if $i\:K\rarrow A$ is the kernel of a morphism $f\:A
\rarrow B$ in $\Top^\s_\boZ$, then the topology of $K$ is induced
from the topology of $A$ via~$i$.
 Furthermore, $K=i^{-1}(0)$ is a closed subgroup in $A$,
as the zero subgroup is closed in a separated topological group~$B$.
 Conversely, if $i$~is an injective map, $i(K)\subset A$ is a closed
subgroup, and the topology of $K$ is induced from the topology of~$A$,
then $i$~is a kernel of the continuous homomorphism $A\rarrow A/i(K)$,
where the group $A/i(K)$ is separated in the quotient topology since
the subgroup $i(K)$ is closed in~$A$.
 If $p\:B\rarrow C$ is the cokernel of a morphism $f\:A\rarrow B$
in $\Top^\s_\boZ$, then the map~$p$ is surjective and the topology
of $C$ is the quotient topology of the topology of~$B$.
 Conversely, if $p\:B\rarrow C$ is a surjective map and the topology
of $C$ is the quotient topology of the topology of $B$, then~$p$
is a cokernel of the continuous homomorphism $\ker(p)\rarrow B$,
where the subgroup $\ker(p)\subset B$ is separated in the induced
topology since the group $B$ is separated.

 So we have obtained the desired descriptions of the classes of all
kernels and cokernels in $\Top^\s_\boZ$, and it is clear from
these descriptions that the composition of two kernels is a kernel
and the composition of two cokernels is a cokernel.
 Hence the category $\Top^\s_\boZ$ is semi-abelian.
 Besides, the monomorphisms in $\Top^\s_\boZ$ are the injective
continuous homomorphisms, and the epimorphisms are continuous
homomorphisms $g\:C\rarrow D$ with a dense image,
$\overline{f(C)_D}=D$.
 For an arbitrary morphism $f\:A\rarrow B$, the induced morphism
$\coim(f)\rarrow\im(f)$ is the natural injective map $f(A)\rarrow
\overline{f(A)}_B$, where $f(A)$ is endowed with the quotient topology
of the topology of $A$, while $\overline{f(A)}_B$ carries the induced
topology of a subgroup in~$B$.
 As $f(A)$ is dense in $\overline{f(A)}_B$, the morphism
$\coim(f)\rarrow\im(f)$ is both a monomorphism and an epimorphism.

 In order to show that a semi-abelian category is quasi-abelian,
it suffices to check that it is either right or left quasi-abelian.
 We will check both the properties.

 Let $i\:K\rarrow A$ be a closed continuous injective homomorphism
of separated topological groups, and let $f\:K\rarrow L$ be an arbitrary
continuous homomorphism of separated topological groups.
 Let $B=(L\oplus A)/K$ denote the pushout of the morphisms~$i$ and~$f$
in the category $\Top_\boZ$.
 We claim that $B$ is separated, so $B$ is also the pushout
$B=L\sqcup_KA$ in the category $\Top^\s_\boZ$.
 Then we know from the proof of
Theorem~\ref{nonseparated-top-groups-spaces-theorem} that
the topology of $L$ is induced from the topology of $B$ via
the natural map $j\:L\rarrow B$, and it remains to check that
$j(L)$ is a closed subgroup in~$B$.

 Indeed, a subgroup $C$ in a topological abelian group $D$ is closed
if and only if the quotient group $D/C$ is separated in the quotient
topology.
 Now the quotient group $B/j(L)$ is naturally isomorphic, as
a topological group, to the quotient group $A/i(K)$.
 Since the subgroup $i(K)$ is closed in $A$, it follows that
the subgroup $j(L)$ is closed in~$B$.
 The zero subgroup is closed in $L$, hence it is also closed in $B$;
so $B$ is separated, and we are finished.

 The pullbacks in the category $\Top^\s_\boZ$ agree with those
in $\Top_\boZ$, because the kernels agree.
 Let $p\:A\rarrow C$ be an open continuous surjective homomorphism
of separated topological abelian groups, and let $f\:D\rarrow C$ be
an arbitrary continuous homomorphism of separated topological
abelian groups.
 Then the subgroup $B=\ker((-f,p))\subset D\oplus A$ endowed with
the induced topology is the pullback $B=D\sqcap_CA$ in both
the categories $\Top_\boZ$ and $\Top^\s_\boZ$.
 It was explained already in the proof of
Theorem~\ref{nonseparated-top-groups-spaces-theorem} that
$q\:B\rarrow D$ is a surjective open map.
\end{proof}

\begin{cor}
\textup{(a)} The full subcategory\/ $\Top^\s_\boZ\subset\Top_\boZ$
is closed under extensions (in the quasi-abelian exact structure
of\/ $\Top_\boZ$) and subobjects.
 The inherited exact category structure on\/ $\Top^\s_\boZ$ from
the quasi-abelian exact structure on\/ $\Top_\boZ$ coincides with
the quasi-abelian exact structure on\/ $\Top^\s_\boZ$. \par
\textup{(b)} The full subcategory\/ $\Top^\s_k\subset\Top_k$ is
closed under extensions (in the quasi-abelian exact structure of\/
$\Top_k$) and subobjects.
 The inherited exact category structure on\/ $\Top^\s_k$ from
the quasi-abelian exact structure on\/ $\Top_k$ coincides with
the quasi-abelian exact structure on\/ $\Top^\s_k$.
\end{cor}

\begin{proof}
 We will explain part~(a).
 The monomorphisms in $\Top_\boZ$ are the injective continuous
homomorphisms, and if $f\:A'\rarrow A''$ is an injective continuous
abelian group map and the topological abelian group $A''$ is separated,
then the topological abelian group $A'$ is separated, too (since
one always has $f(\overline{\{0\}}_{A'})\subset\overline{\{0\}}_{A''}$
for a continuous map~$f$).
 Now let $0\rarrow K\overset i\rarrow A\overset p\rarrow C\rarrow0$
be a short exact sequence in (the quasi-abelian exact structure on)
$\Top_\boZ$ with $A$, $C\in\Top^\s_\boZ$.
 Then $0$ is a closed subgroup in $C$, hence $p^{-1}(0)=i(K)$ is
a closed subgroup in~$A$.
 Since $0$~is a closed subgroup in $K$, the topology of $K$ is induced
from the topology of $A$ via~$i$, and $i(K)$ is a closed subgroup
in $A$, it follows that $0$~is a closed subgroup in~$A$,
so $A\in\Top^\s_\boZ$.

 Following Example~\ref{inherited-exact-structure}, the inherited
exact category structure on $\Top^\s_\boZ$ exists.
 If $0\rarrow K\overset i\rarrow A\overset p\rarrow C\rarrow0$ is
a short sequence in $\Top^\s_\boZ$ that is exact in $\Top_\boZ$,
then $i$~is the kernel of~$p$ and $p$~is the cokernel of~$i$ in
$\Top_\boZ$, hence also in $\Top^\s_\boZ$.
 Conversely, let $0\rarrow K\overset i\rarrow A\overset p\rarrow C
\rarrow0$ be a short exact sequence in the quasi-abelian exact
structure on $\Top^\s_\boZ$.
 Then $i$~is the kernel of~$p$ in $\Top^\s_\boZ$, hence
$i$~is also the kernel of~$p$ in $\Top_\boZ$, as the kernels of
morphisms in $\Top_\boZ$ and $\Top^\s_\boZ$ agree.
 Moreover, the subgroup $p^{-1}(0)=i(K)$ is closed in $A$ in
this case, since $C\in\Top^\s_\boZ$.
 Therefore, the quotient group $A/i(K)$ is separated in the quotient
topology, and the cokernels of the morphism~$i$ in $\Top_\boZ$ and
$\Top^\s_\boZ$ agree.
 So $p$~is also the cokernel of~$i$ in $\Top_\boZ$.

 For an alternative proof, compare
Propositions~\ref{nonseparated-top-vector-spaces-as-suppl-proobjects}(a)
and~\ref{incomplete-top-vector-spaces-as-suppl-proobjects}(a) below
with Lemma~\ref{subcategories-of-supplemented-proobjects}(c) and
Corollary~\ref{sup-epimorphic-quasi-abelian-cor}(b).
\end{proof}

 Denote by $\Top^\omega_\boZ\subset\Top_\boZ$ and $\Top^{\omega,\s}
_\boZ\subset\Top^\s_\boZ$ the full subcategories formed by all
the (unseparated or separated, respectively) topological abelian
groups with a \emph{countable base of neighborhoods of zero}
(consisting of open subgroups).
 Similarly, let $\Top^\omega_k\subset\Top_k$ and $\Top^{\omega,\s}_k
\subset\Top^\s_k$ denote the full subcategories of all topological
vector spaces with a countable base of neighborhoods of zero
(consisting of open subspaces).

 Clearly, the full subcategories $\Top^\omega_\boZ\subset\Top_\boZ$
and $\Top^{\omega,\s}_\boZ\subset\Top^\s_\boZ$ are closed under
kernels, cokernels, and countable products.
 It follows that the additive categories $\Top^\omega_\boZ$ and
$\Top^{\omega,\s}_\boZ$ are quasi-abelian.
 Furthermore, the full subcategories $\Top^\omega_\boZ\subset\Top_\boZ$
and $\Top^{\omega,\s}_\boZ\subset\Top^\s_\boZ$ inherit exact category
structures from their ambient exact categories (in the sense of
Example~\ref{inherited-exact-structure}), and the inherited exact
structures on $\Top^\omega_\boZ$ and $\Top^{\omega,\s}_\boZ$ coincide
with these categories' own quasi-abelian exact structures.
 The full subcategories $\Top^\omega_k\subset\Top_k$ and
$\Top^{\omega,\s}_k\subset\Top^\s_k$ have similar properties.

 However, the full subcategories of topological vector spaces/abelian
groups with a countable base of neighborhoods of zero \emph{are
not closed under countable coproducts}, as the next lemma shows.

\begin{lem} \label{countable-coproducts-not-preserve-countable-base}
 Let $A_0$, $A_1$, $A_2$,~\dots\ be a sequence of topological abelian
groups such that the separated topological abelian group
$A_n/\overline{\{0\}}_{A_n}$ is nondiscrete for every $n\ge0$.
 Then the topological abelian group\/ $\bigoplus_{n=0}^\infty A_n$
does \emph{not} have a countable base of neighborhoods of zero in
the coproduct topology.
\end{lem}

\begin{proof}
 For every $n\ge0$, the natural surjective map $\bigoplus_{i=0}^\infty
A_i\rarrow A_n$ makes $A_n$ a quotient group of
$A=\bigoplus_{i=0}^\infty A_i$, endowed with the quotient topology.
 Hence if $A$ has a countable base of neighborhoods of zero, then so
do $A_n$ for all $n\ge0$.
 This allows us to assume that the topological group $A_n$ has
a countable base of neighborhoods of zero for every $n\ge0$.
 Now if the separated quotient group $A_n/\overline{\{0\}}_{A_n}$ is
not discrete, then the topological group $A_n$ does not have
a finite base of neighborhoods of zero.
 Thus, for every $n\ge0$, we can choose a strictly decreasing
sequence of open subgroups $A_n=U_n^0\varsupsetneq U_n^1\varsupsetneq
U_n^2\varsupsetneq\dotsb$ indexed by the nonnegative integers $j\ge0$
such that the subgroups $(U_n^j)_{j\ge0}$ form a base of neighborhoods
of zero in~$A_n$.

 Then a base of neighborhoods of zero in the topological group $A=
\bigoplus_{n=0}^\infty A_n$ is formed by the subgroups
$U^\psi=\bigoplus_{n=0}^\infty U_n^{\psi(n)}\subset A$, where
$\psi$~ranges over all the maps $\psi\:\omega\rarrow\omega$ from
the set~$\omega$ of all nonnegative integers to itself.
 Suppose, for the sake of contradiction, that the topological group
$A$ has a countable base of neighborhoods of zero.
 Then the set of all the subgroups $U^\psi\subset A$ has a countable
subset forming a base of neighborhoods of zero in~$A$.
 This means that there is a countable subset $\cF\subset \omega^\omega$
in the set of all functions $\omega\rarrow\omega$ such that for
every $\psi\:\omega\rarrow\omega$ there exists $\phi\in\cF$ for which
$U^\phi\subset U^\psi$.
 In view of our choice of the open subgroups $U_n^j\subset A_n$ and
the construction of the open subgroups $U^\psi\subset A$, the inclusion
$U^\phi\subset U^\psi$ means that $\phi(n)\ge\psi(n)$ for all $n\ge0$.
 A simple application of Cantor's diagonal argument shows that such
a countable set $\cF$ of functions $\omega\rarrow\omega$ does not exist.
\end{proof}

\Section{The Maximal Exact Category Structure on VSLTs}
\label{maximal-exact-VSLTs-secn}

 The main result of this section is that the categories
$\Top^\scc_\boZ$ and $\Top^\scc_k$ of complete, separated topological
abelian groups/vector spaces with linear topology are right, but not
left quasi-abelian.
 We also describe the classes of (co)kernels, semi-stable (co)kernels,
and stable (co)kernels in these additive categories, thus obtaining,
in particular, a description of their maximal exact category structures.

\begin{prop} \label{top-groups-spaces-kernels-cokernels-prop}
\textup{(a)} A morphism $i\:\fK\rarrow\fA$ in\/ $\Top^\scc_\boZ$ is
a monomorphism if and only if it is an injective map.
 A morphism $p\:\fA\rarrow\fC$ in\/ $\Top^\scc_\boZ$ is an epimorphism
if and only if the subgroup $p(\fA)$ is dense in\/~$\fC$. \par
 A morphism $i\:\fK\rarrow\fA$ in\/ $\Top^\scc_\boZ$ is a kernel if
and only if $i$~is injective, the subgroup $i(\fK)$ is closed in\/
$\fA$, and the topology of\/ $\fK$ is induced from the topology of\/
$\fA$ (via~$i$).
 In other words, this means that $i$~is an injective closed map. \par
 A morphism $p\:\fA\rarrow\fC$ in\/ $\Top^\scc_\boZ$ is a cokernel if
and only if $p$~is isomorphic to the composition of natural continuous
homomorphisms\/ $\fA\rarrow\fA/\fK\rarrow(\fA/\fK)\sphat\,$, where\/
$\fK\subset\fA$ is a closed subgroup, $\fA\rarrow\fA/\fK$ is
the quotient map, $\fA/\fK$ is endowed with the quotient topology,
$(\fA/\fK)\sphat$ is the completion of the (separated) topological
abelian group\/ $\fA/\fK$, endowed with the completion topology, and\/
$\fA/\fK\rarrow(\fA/\fK)\sphat\,$ is the completion map.
 In other words, this means that $p$~is the composition of
a surjective continuous open homomorphism (onto an incomplete
topological abelian group) followed by the completion map. \par
\textup{(b)} A morphism $i\:\fK\rarrow\fV$ in\/ $\Top^\scc_k$ is
a monomorphism if and only if it is an injective map.
 A morphism $p\:\fV\rarrow\fC$ in\/ $\Top^\scc_k$ is an epimorphism
if and only if the subspace $p(\fV)$ is dense in\/~$\fC$. \par
 A morphism $i\:\fK\rarrow\fV$ in\/ $\Top^\scc_k$ is a kernel if
and only if $i$~is injective, the subspace $i(\fK)$ is closed in\/
$\fV$, and the topology of\/ $\fK$ is induced from the topology of\/
$\fV$ (via~$i$).
 In other words, this means that $i$~is an injective closed map. \par
 A morphism $p\:\fV\rarrow\fC$ in\/ $\Top^\scc_k$ is a cokernel if
and only if $p$~is isomorphic to the composition of natural continuous
linear maps\/ $\fV\rarrow\fV/\fK\rarrow(\fV/\fK)\sphat\,$, where\/
$\fK\subset\fV$ is a closed subspace, $\fV\rarrow\fV/\fK$ is
the quotient map, $\fV/\fK$ is endowed with the quotient topology,
$(\fV/\fK)\sphat$ is the completion of the (separated) topological
vector space\/ $\fV/\fK$, endowed with the completion topology, and\/
$\fV/\fK\rarrow(\fV/\fK)\sphat\,$ is the completion map.
 In other words, this means that $p$~is the composition of
a surjective continuous open linear map (onto an incomplete
topological vector space) followed by the completion map.
\end{prop}

\begin{proof}
 We will explain part~(a); part~(b) is similar.
 According to the discussion in Section~\ref{top-abelian-secn},
the inclusion functors $\Top^\scc_\boZ\rarrow\Top^\s_\boZ\rarrow
\Top_\boZ$ preserve kernels (as well as all limits).
 The cokernel of a morphism $f\:\fA\rarrow\fB$ in $\Top^\scc_\boZ$ is
computed as the completion $C\sphat\,=(\fB/f(\fA))\sphat\,$ of
the cokernel $C=\fB/f(\fA)$ of the morphism~$f$ in the category
$\Top_\boZ$ (with the completion topology on $C\sphat\,$).

 According to Theorem~\ref{incomplete-top-groups-spaces-theorem}(a),
if $i\:\fK\rarrow\fA$ is the kernel of a morphism $f\:\fA\rarrow\fB$
in $\Top^\scc_\boZ$ (hence also in $\Top^\s_\boZ$), then $i$~is
an injective closed map.
 Conversely, if $i$~is an injective closed map, then $i$~is
the kernel of the morphism $\fA\rarrow(\fA/i(\fK))\sphat\,$ in
$\Top^\scc_\boZ$.
 Notice that the topological group $\fA/i(\fK)$ is separated in
the quotient topology, since the subgroup $i(\fK)$ is closed in $\fA$;
hence the completion map $\fA/i(\fK)\rarrow(\fA/i(\fK))\sphat\,$
is injective and the kernel of the composition
$\fA\rarrow\fA/i(\fK)\rarrow(\fA/i(\fK))\sphat\,$ is the subgroup
$i(\fK)\subset\fA$.

 It is clear from the above description of the cokernels of morphisms
in $\Top^\scc_\boZ$ that if $p\:\fB\rarrow\fC$ is the cokernel of
a morphism $f\:\fA\rarrow\fB$, then $p$~is the composition of
a surjective continuous open homomorphism $\fB\rarrow\fB/f(\fA)$
and the completion map $\fB/f(\fA)\rarrow(\fB/f(\fA))\sphat\,$.
 Denote by $\fK=\overline{f(\fA)}_\fB\subset\fB$ the closure of
the subgroup $f(\fA)\subset\fB$; then~$p$ is also the composition
of a surjective continuous open homomorphism $\fB\rarrow\fB/\fK$
onto a separated topological abelian group $\fB/\fK$ and
an injective completion map $(\fB/\fK)\rarrow(\fB/\fK)\sphat\,$.
 Conversely, for any closed subgroup $\fK\subset\fB$, the morphism
$\fB\rarrow(\fB/\fK)\sphat\,$ is the cokernel of the morphism
$\fK\rarrow\fB$ in $\Top^\scc_\boZ$ (where $\fK$ is endowed with
the induced topology of a subgroup in~$\fB$).
\end{proof}

\begin{rem} \label{complete-implies-closed}
 For any separated topological abelian group $B$ with the completion
$\fB=B\sphat\,$, the topology of $B$ is induced from the topology
of $\fB$ via the injective completion map $B\rarrow\fB$.
 Furthermore, let $\fL\subset B$ be a subgroup; suppose that $\fL$
is complete in the induced topology.
 Then $\fL\subset B\subset\fB$ can be also considered as a subgroup
in $\fB$; as the topology of $\fL$ is induced from the topology of
$B$ and the latter is induced from the topology of $\fB$, it follows
that the topology of $\fL$ is induced from the topology of~$\fB$.
 According to Lemma~\ref{closure-completion-lemma}, the closure of
$\fL$ in $\fB$ is the completion of~$\fL$.
 Since $\fL$ is complete by assumption, it follows that $\fL$ is
closed in~$\fB$.
 Consequently, $\fL$ is also closed in~$B$.
 We have shown that \emph{a complete, separated topological abelian
group is closed in any separated topological abelian group where it
is embedded with the induced topology}.

 So the condition that $i(\fK)$ is closed in $\fA$ can be dropped
from the characterization of kernels in $\Top^\scc_\boZ$ given in
Proposition~\ref{top-groups-spaces-kernels-cokernels-prop}(a):
any morphism in $\Top^\scc_\boZ$ that is a kernel in $\Top_\boZ$ is
also a kernel in  $\Top^\scc_\boZ$.
 The inclusion of the full subcategory $\Top^\s_\boZ\rarrow\Top_\boZ$
does \emph{not} have this property (cf.\
Theorems~\ref{nonseparated-top-groups-spaces-theorem}(a)
and~\ref{incomplete-top-groups-spaces-theorem}(a)).
 Similarly, the condition that $i(\fK)$ is closed in $\fV$ can be
dropped from the characterization of kernels in $\Top^\scc_k$ given
in Proposition~\ref{top-groups-spaces-kernels-cokernels-prop}(b),
and any morphism in $\Top^\scc_k$ that is a kernel in $\Top_k$ is
also a kernel in $\Top^\scc_k$.
 The inclusion of the full subcategory $\Top^\s_k\rarrow\Top_k$
does \emph{not} have this property (cf.\
Theorem~\ref{nonseparated-top-groups-spaces-theorem}(b)
and~\ref{incomplete-top-groups-spaces-theorem}(b)).
\end{rem}

 We refer to Section~\ref{maximal-exact-struct-secn} for the definitions
of semi-stable and stable kernels and cokernels in additive categories
generally, and to Example~\ref{right-quasi-abelian-semi-stable-example}
for a discussion of these classes of morphisms in right quasi-abelian
categories specifically.

\begin{prop} \label{right-quasi-abelian-prop}
 The additive categories\/ $\Top^\scc_\boZ$ and\/ $\Top^\scc_k$ are
right quasi-abelian.
 In other words, all kernels in\/ $\Top^\scc_\boZ$ and\/ $\Top^\scc_k$
are semi-stable kernels, and consequently all semi-stable cokernels in
these categories are stable cokernels. 
\end{prop}

\begin{proof}
 Let us explain the assertion for topological abelian groups;
the case of topological vector spaces is similar.
 Let $i\:\fK\rarrow\fA$ be a kernel of some morphism
in $\Top^\scc_\boZ$; by
Proposition~\ref{top-groups-spaces-kernels-cokernels-prop},
this means that $i$~is a closed continuous injective homomorphism
of complete, separated topological abelian groups.
 Let $f\:\fK\rarrow\fL$ be an arbitrary morphism in $\Top^\scc_\boZ$.
 Then the pushout $B$ of the morphisms~$i$ and~$p$ in
the category $\Top_\boZ$ is the quotient group $B=(\fL\oplus\fA)/\fK$,
endowed with the quotient topology.
 Following the proof of
Theorem~\ref{incomplete-top-groups-spaces-theorem}, \,$B$ is
a separated topological abelian group.
 The pushout $\fB=\fL\sqcup_\fK\fA$ in the category $\Top^\scc_\boZ$
is the completion of the topological group $B$, i.~e., $\fB=B\sphat\,$.
 We know from Theorem~\ref{nonseparated-top-groups-spaces-theorem}
that the topology of $\fL$ is induced from the topology of $B$ via
the natural injective map $\fL\rarrow B$.
 Hence the topology of $\fL$ is also induced from the topology of $\fB$
via the natural morphism $\fL\rarrow\fB$, which is the composition of
injective maps $\fL\rarrow B\rarrow\fB$.
 According to Remark~\ref{complete-implies-closed}, it follows that
the morphism $\fL\rarrow\fB$ is a kernel in $\Top^\scc_\boZ$.

 For an alternative proof, compare
Theorem~\ref{complete-top-vector-spaces-as-proobjects} below with
Proposition~\ref{limit-epimorphic-inside-supplemented-prop}.
\end{proof}

\begin{prop} \label{vslt-semi-stable-cokernels-prop}
\textup{(a)} A morphism $p\:\fA\rarrow\fC$ in\/ $\Top^\scc_\boZ$ is
a semi-stable cokernel if and only if~$p$ is a surjective open map. \par
\textup{(b)} A morphism $p\:\fV\rarrow\fW$ in\/ $\Top^\scc_k$ is
a semi-stable cokernel if and only if~$p$ is a surjective open map.
\end{prop}

\begin{proof}
 Let us explain part~(a); part~(b) is similar.
 By Proposition~\ref{top-groups-spaces-kernels-cokernels-prop},
a morphism $p\:\fA\rarrow\fC$ is a cokernel in $\Top^\scc_\boZ$ if
and only if it is the composition of a surjective open map
$\fA\rarrow C$ onto a separated topological abelian group $C$
and the completion map $C\rarrow C\sphat\,=\fC$.
 So we have to show that a cokernel~$p$ is a semi-stable cokernel
in $\Top^\scc_\boZ$ if and only if $p$~is a surjective map.

 Suppose that $p$~is not surjective, and choose an element $x\in\fC$
not belonging to the image of~$p$.
 Consider the abelian group $\boZ$ with the discrete topology,
and the group homomorphism $f\:\boZ\rarrow\fC$ taking $1$ to~$x$.
 Recall that the pullbacks in $\Top^\scc_\boZ$ agree with those
in $\Top_\boZ$ and are preserved by the forgetful functor
$\Top^\scc_\boZ\rarrow\Ab$.
 The map $q\:\boZ\sqcap_\fC\fA\rarrow\boZ$ is not surjective, as
the element $1\in\boZ$ does not belong to its image.
 As the topology on $\boZ$ is discrete, the image of~$q$ cannot be
dense in $\boZ$; so $q$~is not even an epimorphism in $\Top^\scc_\boZ$.
 Thus $p$~is not a semi-stable cokernel.

 Suppose that $p$~is surjective; then $p$~is an open map, hence
it is a cokernel in $\Top_\boZ$ and $\Top^\s_\boZ$.
 Let $f\:\fD\rarrow\fC$ be an arbitrary morphism in $\Top^\scc_\boZ$.
 Once again we recall that the pullbacks in $\Top^\scc_\boZ$ agree with
those in $\Top_\boZ$ and $\Top^\s_\boZ$.
 By Theorem~\ref{incomplete-top-groups-spaces-theorem}, the morphism
$q\:\fD\sqcap_\fC\fA\rarrow\fA$ is a cokernel in $\Top^\s_\boZ$.
 Following the descriptions of the classes of cokernels in
the categories $\Top^\s_\boZ$ and $\Top^\scc_\boZ$ given in
Theorem~\ref{incomplete-top-groups-spaces-theorem} and
Proposition~\ref{top-groups-spaces-kernels-cokernels-prop},
any morphism in $\Top^\scc_\boZ$ that is a cokernel in
$\Top^\s_\boZ$ is also a cokernel in $\Top^\scc_\boZ$.
 (It is also true that a morphism in $\Top^\s_\boZ$ is a cokernel
if and only if it is a cokernel in $\Top_\boZ$;
see Theorem~\ref{nonseparated-top-groups-spaces-theorem}.)
 Thus $	q$~is a cokernel in $\Top^\scc_\boZ$, and $p$~is
a semi-stable cokernel, as desired.
\end{proof}

\begin{cor} \label{vslt-stable-kernels-cor}
\textup{(a)} A morphism $i\:\fK\rarrow\fA$ in\/ $\Top^\scc_\boZ$ is
a stable kernel if and only if $i$~is an injective closed map
and the quotient group\/ $\fA/i(\fK)$ is complete in the quotient
topology. \par
\textup{(b)} A morphism $i\:\fK\rarrow\fV$ in\/ $\Top^\scc_k$ is
a stable kernel if and only if $i$~is an injective closed map
and the quotient space\/ $\fV/i(\fK)$ is complete in the quotient
topology.
\end{cor}

\begin{proof}
 Part~(a) follows from Propositions~\ref{right-quasi-abelian-prop}
and~\ref{vslt-semi-stable-cokernels-prop}(a) together with
the description of the cokernels of morphisms in the category
$\Top^\scc_\boZ$ (see Section~\ref{top-abelian-secn} and/or
the first paragraph of the proof
of Proposition~\ref{top-groups-spaces-kernels-cokernels-prop}).
 Part~(b) is similar.
\end{proof}

\begin{cor} \label{not-left-quasi-abelian-cor}
 The categories\/ $\Top^\scc_\boZ$ and\/ $\Top^\scc_k$ are not
left quasi-abelian.
 In fact, they are not even left semi-abelian.
\end{cor}

\begin{proof}
 Let us discuss the case of topological vector spaces.
 Choose an incomplete separated topological $k$\+vector space $C$
(with linear topology).
 For example, one can consider the complete, separated topological
vector space $\fC=k^\omega=\prod_{n=0}^\infty k$ (with the product
topology of discrete one-dimensional vector spaces~$k$), and
the dense vector subspace $C=k^{(\omega)}=\bigoplus_{n=0}^\infty k
\subset\fC$ with the topology on $C$ induced from~$\fC$.
 Then $\fC$ and consequently $C$ are even topological vector spaces
with countable bases of open subspaces.
 Furthermore, $\fC$ is a linearly compact (profinite-dimensional)
topological vector space, while $C$ has countable dimension.

 Following Theorem~\ref{top-vector-main-theorem}, there exists
an open, continuous surjective linear map $\Sigma\:\fA_\omega(C)
\rarrow C$ onto $C$ from the complete, separated topological vector
space $\fA_\omega(C)=C^{(\omega)}=\bigoplus_{m=0}^\infty C$ with
the modified coproduct topology.
 Put $\fV=\fA_\omega(C)$, and let $\fK\subset\fV$ be the kernel
of~$\Sigma$, endowed with the induced topology as a closed
subspace in~$\fV$.
 Then $i\:\fK\rarrow\fV$ is a morphism in $\Top^\scc_k$, the map
$\Sigma\:\fV\rarrow C$ is the cokernel of~$i$ in $\Top_k$ and in
$\Top^\s_k$, and the composition $\fV\overset\Sigma\rarrow C
\overset{\lambda_C}\rarrow C\sphat\,=\fC$ is the cokernel of~$i$
in $\Top^\scc_k$.
 Denote this composition by $p\:\fV\rarrow\fC$.

 Now $p$~is a cokernel in $\Top^\scc_k$ which is not a surjective map.
 Choose a vector $x\in\fC\setminus C$; so $x$~does not belong to
the image of~$p$.
 Consider the one-dimensional vector space~$k$ with the discrete
topology, and the linear map $f\:k\rarrow\fC$ taking $1$ to~$x$.
 Then the pullback $k\sqcap_\fC\fV$ of the morphisms $p$ and~$f$
is $k\sqcap_\fC\fV=\fK$, and the natural morphism
$q\:k\sqcap_\fC\fV\rarrow k$ is the zero map.
 So the morphism $q\:k\sqcap_\fC\fV\rarrow k$ is certainly not
a cokernel (and not an epimorphism) in $\Top^\scc_k$.
\end{proof}

\begin{exs} \label{not-left-semi-abelian-examples}
 The proof of
Corollary~\ref{not-left-quasi-abelian-cor} presents an example
showing that the property of
Proposition~\ref{four-properties-cokernel-side-prop}(4) does not hold
in the category $\Top^\scc_k$.
 Consequently, there should also exist counterexamples to
the properties of
Proposition~\ref{four-properties-cokernel-side-prop}(1\+-3)
in $\Top^\scc_k$.
 Let us suggest such counterexamples here.

 We keep the notation of the proof of
Corollary~\ref{not-left-quasi-abelian-cor}.
 The induced topology on the vector subspace $kx\subset\fC$
is separated (since $\fC$ is separated); so it must be discrete.
 By Corollary~\ref{countable-base-complete-subspace-splits}, it
follows that $kx$ is a direct summand in $\fC$, so the quotient
space $\fC/kx$ is separated and complete.
 Therefore, we have $\coker(f)=\fC/kx$ in $\Top^\scc_k$.
 Denote by~$g$ the split epimorphism $g\:\fC\rarrow\fC/kx$.
 Then both the morphisms $p$ and~$g$ are cokernels in $\Top^\scc_k$,
but the composition $gp\:\fV\rarrow\fC/kx$ is not a cokernel.
 Indeed, the kernel of~$gp$ is the morphism $i\:\fK\rarrow\fV$,
and the cokernel of~$i$ is $p$ rather than~$gp$.

 Alternatively, endow the direct sum $kx\oplus\fV$ with
the (co)product topology of the discrete topology on $kx$ and
the above topology on~$\fV$.
 Then the morphism $\mathrm{id}_{kx}\oplus p\:kx\oplus\fV\rarrow
kx\oplus\fC$ is a cokernel (of the morphism $(0,i)\:\fK\rarrow kx
\oplus\fV$), and the morphism $(f,\mathrm{id}_{\fC})\:
kx\oplus\fC\rarrow\fC$ is a cokernel (in fact, even a direct summand
projection), but the composition $(f,\mathrm{id}_{\fC})\circ
(\mathrm{id}_{kx}\oplus p)=(f,p)\:kx\oplus\fV\rarrow\fC$
is not a cokernel.
 Indeed, the kernel of $(f,p)$ is the morphism $(0,i)\:\fK\rarrow
kx\oplus\fC$, and the cokernel of $(0,i)$ is $kx\oplus\fV\rarrow
kx\oplus\fC$ rather than $kx\oplus\fV\rarrow\fC$.
 These are counterexamples to the property of
Proposition~\ref{four-properties-cokernel-side-prop}(2)
in $\Top^\scc_k$.

 Furthermore, the composition $\fV\rarrow kx\oplus\fV\rarrow\fC$
of the morphisms $(0,\mathrm{id}_\fV)\:\fV\rarrow kx\oplus\fV$
and $(f,p)\:kx\oplus\fV\rarrow\fC$ is the morphism $p\:\fV\rarrow\fC$,
which is the cokernel of $i\:\fK\rarrow\fC$.
 But the morphism $kx\oplus\fV\rarrow\fC$ is not a cokernel.
 This is a counterexample to the property of
Proposition~\ref{four-properties-cokernel-side-prop}(3)
in $\Top^\scc_k$. {\hbadness=1125\par}

 Finally, the coimage of the morphism $gp\:\fV\rarrow\fC/k$ in
the category $\Top^\scc_k$ is the object $\coim(gp)=\coker(i)=
\fC$, while the image is $\im(gp)=\ker(\fC/k\to0)=\fC/k$.
 The natural morphism $\coim(gp)\rarrow\im(gp)$ is not
a monomorphism; in fact, it is the split epimorphism
$g\:\fC\rarrow\fC/kx$.

 Alternatively, the coimage of the morphism $(f,p)\:kx\oplus\fV
\rarrow\fC$ is $\coim((f,p))=\coker((0,i)\:\fK\to kx\oplus\fV)=
kx\oplus\fC$, while the image of $(f,p)$ is $\im((f,p))=
\ker(\fC\to0)=\fC$.
 The natural morphism $\coim((f,p))\rarrow\im((f,p))$ is not
a monomorphism; in fact, it is the split epimorphism
$(f,\mathrm{id}_{\fC})\:kx\oplus\fC\rarrow\fC$.
 These are counterexamples to the property of
Proposition~\ref{four-properties-cokernel-side-prop}(1)
in $\Top^\scc_k$.
\end{exs}

\begin{cor}
\textup{(a)} The full subcategory\/ $\Top^\scc_\boZ\subset\Top^\s_\boZ$
is closed under extensions (in the quasi-abelian exact structure
of\/ $\Top^\s_\boZ$) and kernels.
 The inherited exact category structure on\/ $\Top^\scc_\boZ$ from
the quasi-abelian exact structure on\/ $\Top^\s_\boZ$ coincides with
the maximal exact structure on\/ $\Top^\scc_\boZ$. \par
\textup{(b)} The full subcategory\/ $\Top^\scc_k\subset\Top^\s_k$ is
closed under extensions (in the quasi-abelian exact structure of\/
$\Top^\s_k$) and kernels.
 The inherited exact category structure on\/ $\Top^\scc_k$ from
the quasi-abelian exact structure on\/ $\Top^\s_k$ coincides with
the maximal exact structure on\/ $\Top^\scc_k$.
\end{cor}

\begin{proof}
 Let us explain part~(a).
 Let $0\rarrow\fK\overset i\rarrow A\overset p\rarrow\fC\rarrow0$
be a short exact sequence in (the quasi-abelian exact structure on)
$\Top^\s_\boZ$ with $\fK$, $\fC\in\Top^\scc_\boZ$.
 Then, for every open subgroup $U\subset A$, we have a short exact
sequence of (discrete quotient) groups $0\rarrow\fK/(i^{-1}(U)\cap\fK)
\rarrow A/U\rarrow\fC/(p(U))\rarrow0$.
 Consider the commutative diagram of a morphism of short sequences
of abelian groups
$$
\xymatrix{
 0\ar[r] & \fK \ar[r]\ar@{=}[d] & A \ar[r]\ar[d]
 & \fC \ar[r]\ar@{=}[d] & 0 \\
 0\ar[r] & \varprojlim_{U\subset A} \fK/(i^{-1}(U)\cap\fK) \ar[r]
 & \varprojlim_{U\subset A} A/U \ar[r]
 & \varprojlim_{U\subset A} \fC/(p(U))
}
$$
 The lower line is a left exact sequence, since the projective limit
functor is left exact.
 The map $\fK\rarrow\varprojlim_{U\subset A}\fK/(i^{-1}(U)\cap\fK)$
is an isomorphism, since the topological abelian group $\fK$ is
complete in its topology induced from the topology of $A$ via~$i$.
 The map $\fC\rarrow\varprojlim_{U\subset A}\fC/(p(U))$ is
an isomorphism, since the topological abelian group $\fC$ is complete
in its quotient topology.
 It follows that the map $\varprojlim_{U\subset A}A/U\rarrow
\varprojlim_{U\subset A}\fC/(p(U))$ is surjective and
the map $A\rarrow\varprojlim_{U\subset A}A/U$ is an isomorphism,
so the topological abelian group $A$ is complete
(cf.\ Lemma~\ref{abelian-reflective-closed-under-extensions}).

 Following Example~\ref{inherited-exact-structure}, the inherited exact
structure on $\Top^\scc_\boZ$ exists.
 Let $0\rarrow\fK\overset i\rarrow\fA\overset p\rarrow\fC\rarrow0$
be a short exact sequence in the maximal exact structure on
$\Top^\scc_\boZ$; we have to show that this short sequence is also
exact in $\Top^\s_\boZ$.
 But this is clear from
Proposition~\ref{vslt-semi-stable-cokernels-prop}(a)
or Corollary~\ref{vslt-stable-kernels-cor}(a) compared with
the descriptions of the kernels and cokernels of morphisms in
the categories $\Top^\scc_\boZ$ and $\Top^\s_\boZ$ known from
the proofs of Theorem~\ref{incomplete-top-groups-spaces-theorem}(a)
and Proposition~\ref{top-groups-spaces-kernels-cokernels-prop}(a)
(actually, even from Section~\ref{top-abelian-secn}).

 For an alternative proof, compare
Theorem~\ref{complete-top-vector-spaces-as-proobjects}(a) and
Proposition~\ref{incomplete-top-vector-spaces-as-suppl-proobjects}(a)
below with
Corollary~\ref{sup-epimorphic-quasi-abelian-cor}
and Proposition~\ref{limit-epimorphic-inside-supplemented-prop}.
\end{proof}

 Denote by $\Top^{\omega,\scc}_\boZ\subset\Top^\scc_\boZ$ the full
subcategory formed by all the (complete, separated) topological
abelian groups with a countable base of neighborhoods of zero
(consisting of open subgroups).
 Similarly, let $\Top^{\omega,\scc}_k\subset\Top^\scc_k$ denote
the full subcategory of all topological vector spaces with
a countable base of neighborhoods of zero.
 Clearly, the full subcategories 
$\Top^{\omega,\scc}_\boZ\subset\Top^\scc_\boZ$ and
$\Top^{\omega,\scc}_k\subset\Top^\scc_k$ are closed under
kernels, cokernels, and countable products.

 Furthermore, by Proposition~\ref{countable-base-kernel},
the cokernel of any morphism in $\Top^{\omega,\scc}_\boZ$
(as well as the cokernel of any injective closed morphism
$i\:\fK\rarrow\fA$ in $\Top^\scc_\boZ$ with
$\fK\in\Top^{\omega,\scc}_\boZ$) is a surjective open map.
 Consequently, the category $\Top^{\omega,\scc}_\boZ$ is
quasi-abelian.
 The kernels (i.~e., the admissible monomorphisms in
the quasi-abelian exact structure) in $\Top^{\omega,\scc}_\boZ$
are the injective closed maps, and the cokernels (i.~e.,
the admissible epimorphisms in the quasi-abelian exact structure)
are the surjective open maps.
 The full subcategory $\Top^{\omega,\scc}_\boZ\subset\Top^\scc_\boZ$
inherits the maximal exact category structure of the additive category
$\Top^\scc_\boZ$, and the inherited exact category structure on
$\Top^{\omega,\scc}_\boZ$ is the quasi-abelian exact structure.

 The full subcategory $\Top^{\omega,\scc}_k\subset\Top^\scc_k$
has similar properties.
 Moreover, by Corollary~\ref{countable-base-complete-subspace-splits},
every short exact sequence in the quasi-abelian category
$\Top^{\omega,\scc}_k$ splits.
 In other words, the quasi-abelian (hence maximal) exact structure
on $\Top^{\omega,\scc}_k$ coincides with the split (minimal) exact
category structure.
 In the ambient categories $\Top^\scc_k\subset\Top^\s_k\subset\Top_k$,
the complete, separated topological vector spaces with a countable
base of neighborhoods of zero have a rather strong injectivity
property described in Proposition~\ref{countable-injectivity}.
 In particular, all the objects of $\Top^{\omega,\scc}_k$ are injective
with respect to the maximal exact category structure on $\Top^\scc_k$ 
(and even with respect to the quasi-abelian exact structures on
$\Top^\s_k$ and $\Top_k$).

 However, the full subcategory $\Top^{\omega,\scc}_k$ is \emph{not}
closed under countable coproducts in $\Top^\scc_k$;
see Lemma~\ref{countable-coproducts-not-preserve-countable-base}.
 Similarly, the full subcategory $\Top^{\omega,\scc}_\boZ$ is \emph{not}
closed under countable coproducts in $\Top^\scc_\boZ$.
 We recall that coproducts in the category $\Top^\scc_k$ agree with
those in $\Top^\s_k$ and in $\Top_k$ (and similarly for topological
abelian groups); see Lemma~\ref{coproduct-topology-lemma}.

 In order to formulate the conclusion, let us add some bits of
terminology.
 Given a complete, separated topological abelian group (or vector space)
$\fA$ with linear topology and a closed injective morphism of
topological abelian groups/vector spaces $i\:\fK\rarrow\fA$, we will
say that the map~$i$ is \emph{stably closed} if $i$~is a stable kernel
in $\Top^\scc_\boZ$ or in $\Top^\scc_k$ (see
Corollary~\ref{vslt-stable-kernels-cor} for the description).
 In this case, the closed subgroup/subspace $i(\fK)\subset\fA$ will
be also called \emph{stably closed}.

\begin{conc} \label{maximal-exact-conclusion}
 The category $\Top^\scc_k$ of complete, separated topological vector
spaces with linear topology is \emph{not} quasi-abelian
(see Proposition~\ref{right-quasi-abelian-prop}
and Corollary~\ref{not-left-quasi-abelian-cor}).
 Moreover, contrary to~\cite[page~1, Section~1.1]{Beil}, there does
\emph{not} exist an exact category structure on $\Top^\scc_k$ in
which all the closed embeddings would be admissible monomorphisms.
 In the maximal exact category structure on $\Top^\scc_k$,
the admissible epimorphisms are the open surjections, and
the admissible monomorphisms are the \emph{stably} closed embeddings.

 The problem does not arise in the categories of incomplete topological
vector spaces $\Top_k$ and $\Top^\s_k$
(see Theorems~\ref{nonseparated-top-groups-spaces-theorem}
and~\ref{incomplete-top-groups-spaces-theorem}), and it also does not
arise in the category of complete, separated topological vector spaces
with a countable base of neighborhoods of zero.
 However, countable coproducts in any one of the categories
$\Top^\scc_k\subset\Top^\s_k\subset\Top_k$ do \emph{not} preserve
the classes of topological vector spaces with a countable base of
open subspaces (by
Lemmas~\ref{countable-coproducts-not-preserve-countable-base}
and~\ref{coproduct-topology-lemma}).
\end{conc}

\Section{Pro-Vector Spaces} \label{pro-vector-spaces-secn}

 Complete, separated topological vector spaces with linear topology
form a full subcategory in the abelian category of pro-vector spaces;
and incomplete topological vector spaces can be interpreted as
pro-vector spaces with some additional datum.
 In this section we explain that the full subcategory $\Top^\scc_k$
does \emph{not} inherit an exact category structure from the abelian
exact category structure of $\Pro(\Vect_k)$.
 In particular, $\Top^\scc_k$ is \emph{not} closed under extensions
in $\Pro(\Vect_k)$.

 We refer to~\cite[Chapter~6]{KS} for a general discussion of
ind-objects; the pro-objects are dual.
 Given a category $\sC$, the category $\Pro(\sC)$ of pro-objects in
$\sC$ is defined as the opposite category to the full subcategory
in the category of covariant functors $\sC\rarrow\Sets$ formed by
the directed inductive limits of corepresentable functors
$\Hom_\sC(C,{-})$, \,$C\in\sC$.
 For an additive category $\sC$, one can use (if one wishes, additive)
functors $\sC\rarrow\Ab$ in lieu of the functors $\sC\rarrow\Sets$.

 Explicitly, this means that the objects of $\Pro(\sC)$ are
the projective systems $P\:\Gamma\rarrow\sC$ indexed by directed
posets~$\Gamma$.
 This means that for every $\gamma\in\Gamma$ there is an object
$P_\gamma\in\sC$ and for every $\gamma<\delta\in\Gamma$ there is
a morphism $P_\delta\rarrow P_\gamma$ such that for every
$\gamma<\delta<\epsilon\in\Gamma$ the triangle diagram
$P_\epsilon\rarrow P_\delta\rarrow P_\gamma$ is commutative.
 The object of $\Pro(\sC)$ corresponding to a projective system
$(P_\gamma)_{\gamma\in\Gamma}$ is denoted by
$$
 \plim_{\gamma\in\Gamma} P_\gamma\in\Pro(\sC). 
$$
 The (opposite object to the) object $\plim_{\gamma\in\Gamma}P_\gamma$
corresponds to the functor $\sC\rarrow\Sets$ (or, in the additive case,
possibly $\sC\rarrow\Ab$) defined by the rule
$$
 X\longmapsto\varinjlim\nolimits_{\gamma\in\Gamma}
 \Hom_\sC(P_\gamma,X), \qquad X\in\sC.
$$
 The set/group of morphisms $\plim_{\gamma\in\Gamma}P_\gamma
\rarrow\plim_{\delta\in\Delta}Q_\delta$ in $\Pro(\sC)$ is, by
the definition, the set/group of morphisms in the opposite
direction between the corresponding functors $\sC\rarrow\Sets$
(or $\sC\rarrow\Ab$).
 This means that
$$
 \Hom_{\Pro(\sC)}(\plim_{\gamma\in\Gamma}P_\gamma,\,
 \plim_{\delta\in\Delta}Q_\delta)=
 \varprojlim\nolimits_{\delta\in\Delta}
 \varinjlim\nolimits_{\gamma\in\Gamma}\Hom_\sC(P_\gamma,Q_\delta)
$$
for any two projective systems $(P_\gamma)_{\gamma\in\Gamma}$
and $(Q_\delta)_{\delta\in\Delta}$ in $\sC$ indexed by directed
posets $\Gamma$ and~$\Delta$.

 To any object of $\sC$ one can assign the object of $\Pro(\sC)$
represented by the projective system $R_C\:\{*\}\rarrow\sC$ indexed
by the sigleton $\Gamma=\{*\}$ with $R_C({*})=C$.
 This defines a fully faithful functor $\sC\rarrow\Pro(\sC)$ such
that an arbitrary object $\plim_{\gamma\in\Gamma} P_\gamma\in
\Pro(\sC)$ is the projective limit of the objects $P_\gamma\in
\sC\subset\Pro(\sC)$ in the category $\Pro(\sC)$.
 All directed projective limits exist in the category $\Pro(\sC)$.

 The formula for the set/group of morphisms in $\Pro(\sC)$ from
an arbitrary object $\plim_{\gamma\in\Gamma}P_\gamma$ to
an object $C\in\sC\subset\Pro(\sC)$ is worth writing down explicitly:
$$
 \Hom_{\Pro(\sC)}(\plim_{\gamma\in\Gamma}P_\gamma,\,C)=
 \varinjlim\nolimits_{\gamma\in\Gamma}\Hom_\sC(P_\gamma,C).
$$
 So any given morphism $\plim_{\gamma\in\Gamma}P_\gamma\rarrow C$
in $\Pro(\sC)$ factorizes through the canonical projection
$\plim_{\gamma\in\Gamma}P_\gamma\rarrow P_\delta$ for some
$\delta\in\Gamma$.

 Notice that directed projective limits in the category $\sC$, which
may or may not exist, are in any case \emph{almost never} preserved
by the embedding $\sC\rarrow\Pro(\sC)$.
 In fact, if $(P_\gamma\in\sC)_{\gamma\in\Gamma}$ is a projective
system indexed by a directed poset $\Gamma$ and the projective limit
$\varprojlim_{\gamma\in\Gamma}P_\gamma$ exists in $\sC$, then
the inclusion functor $\sC\rarrow\Pro(\sC)$ preserves this projective
limit if and only if the object $\plim_{\gamma\in\Gamma}P_\gamma
\in\Pro(\sC)$ belongs to the full subcategory $\sC\subset\Pro(\sC)$.
 Then one has $\plim_{\gamma\in\Gamma}P_\gamma=\varprojlim_{\gamma
\in\Gamma}P_\gamma$.
 If this is the case, then there exists an index $\delta\in\Gamma$
such that the projection $\plim_{\gamma\in\Gamma}P_\gamma\rarrow
P_\delta$ is a split monomorphism in $\Pro(\sC)$, and it follows
that the projection $\varprojlim_{\gamma\in\Gamma}P_\gamma\rarrow
P_\delta$ is a split monomorphism in~$\sC$.
 So this is a kind of degenerate situation.

 If the category $\sC$ is additive, then so is $\Pro(\sC)$.
 The following description of zero objects in $\Pro(\sC)$ is helpful.
 Let $\sC$ be an additive category and $(P_\gamma)_{\gamma\in\Gamma}$
be a directed projective system in~$\sC$.
 One can see from the above description of morphisms in $\Pro(\sC)$
that the object $\plim_{\gamma\in\Gamma} P_\gamma$ vanishes
in $\Pro(\sC)$ if and only if the projective system
$(P_\gamma)_{\gamma\in\Gamma}$ is \emph{pro-zero}, in the sense
that for every $\gamma\in\Gamma$ there exists $\delta\in\Delta$,
\,$\delta\ge\gamma$ such that the transition morphism $P_\delta
\rarrow P_\gamma$ is zero.

 If the category $\sC$ is abelian, then so is $\Pro(\sC)$.
 Let us explain this assertion in some more detail.
 Given a morphism $f:\plim_{\gamma\in\Gamma}P_\gamma\rarrow
\plim_{\delta\in\Delta}Q_\delta$ in $\Pro(\sC)$, consider the poset
$\Xi$ formed by all the triples
$\xi'=(\gamma',\delta', g_{\gamma'\delta'})$
such that $\gamma'\in\Gamma$, \,$\delta'\in\Delta$, \
$g_{\gamma'\delta'}$ is a morphism $g_{\gamma'\delta'}\:
P_{\gamma'}\rarrow Q_{\delta'}$ in $\sC$, and the square diagram formed
by~$f$, $g_{\gamma'\delta'}$, and the canonical projections
$\plim_{\gamma\in\Gamma}P_\gamma\rarrow P_{\gamma'}$, \
$\plim_{\delta\in\Delta}Q_\delta\rarrow Q_{\delta'}$ is commutative
in $\Pro(\sC)$.
 By the definition, $(\gamma',\delta',g_{\gamma'\delta'})\le
(\gamma'',\delta'',g_{\gamma''\delta''})$ in $\Xi$ if
$\gamma'\le\gamma''$ in $\Gamma$, \ $\delta'\le\delta''$ in $\Delta$,
and the relevant square diagram is commutative in~$\sC$.
 Now the morphism~$f$ is naturally isomorphic to the morphism
$\plim_{\xi'\in\Xi}g_{\gamma'\delta'}\:
\plim_{\xi'\in\Xi}P_{\gamma'}\rarrow \plim_{\xi'\in\Xi}Q_{\delta'}$.
 This construction shows that any morphism in $\Pro(\sC)$ can be
represented by a $\Xi$\+indexed projective system of morphisms
in~$\sC$ for a suitable directed poset~$\Xi$.

 Now let $(f_\xi\:P_\xi\to Q_\xi)_{\xi\in\Xi}$ be a directed
projective system of morphisms in~$\sC$.
 Then the kernel and cokernel of the morphism $\plim_{\xi\in\Xi}f_\xi\:
\plim_{\xi\in\Xi}P_\xi\rarrow\plim_{\xi\in\Xi}Q_\xi$ can be simply
computed termwise as
$$
 \ker(\plim_{\xi\in\Xi}f_\xi)=\plim_{\xi\in\Xi}\ker(f_\xi)
 \quad\text{and}\quad
 \coker(\plim_{\xi\in\Xi}f_\xi)=\plim_{\xi\in\Xi}\coker(f_\xi).
$$
 Hence, in particular, the embedding functor $\sC\rarrow\Pro(\sC)$
is exact.

 Let $\sC$ be a complete abelian category (i.~e., an abelian category
with projective limits, or equivalently, with infinite products).
 Then there is a left exact functor
$$
 \varprojlim\:\Pro(\sC)\lrarrow\sC
$$
taking a pro-object $\plim_{\gamma\in\Gamma}P_\gamma\in\Pro(\sC)$
to the projective limit of the directed projective system
$(P_\gamma)_{\gamma\in\Gamma}$ in the category~$\sC$,
$$
 \varprojlim(\plim_{\gamma\in\Gamma}P_\gamma)=
 \varprojlim\nolimits_{\gamma\in\Gamma}P_\gamma.
$$
 The functor $\varprojlim$ is right adjoint to the fully faithful
exact embedding functor $\sC\rarrow\Pro(\sC)$.
 In particular, there is a natural transformation of endofunctors
on the category $\Pro(\sC)$ assigning to every pro-object
$P\in\Pro(\sC)$ the adjunction morphism
$$
 \varprojlim\nolimits_{\gamma\in\Gamma}P_\gamma
 \lrarrow\plim_{\gamma\in\Gamma}P_\gamma
$$
in the category of pro-objects $\Pro(\sC)$ (where the object
$\varprojlim_{\gamma\in\Gamma}P_\gamma\in\sC$ is viewed as an object
of $\Pro(\sC)$ via the embedding functor $\sC\rarrow\Pro(\sC)$).

 The following simple categorical construction will be useful for
our discussion of incomplete topological vector spaces/abelian groups
with linear topology in the next Section~\ref{suppl-pro-secn}.
 A \emph{supplemented pro-object} $(C,P)$ in a category $\sC$ is
a pair consisting of an object $C\in\sC$ and a pro-object
$P\in\Pro(\sC)$, endowed with a morphism $\pi\:C\rarrow P$ in
$\Pro(\sC)$.
 Morphisms of supplemented pro-objects are defined in the obvious way.
 We will denote the category of supplemented pro-objects in $\sC$
by $\Pro^\su(\sC)$.

 We will say that a supplemented pro-object $(C,P)$ in $\sC$ is
\emph{sup-epimorphic} if the morphism $\pi\:C\rarrow P$ is
an epimorphism in the abelian category $\Pro(\sC)$.
 Let $(P_\gamma)_{\gamma\in\Gamma}$ be a directed projective system
in $\sC$ such that $P=\plim_{\gamma\in\Gamma}P_\gamma$; then
the morphism $\pi$ is represented by a compatible cone of
morphisms $\pi_\gamma\:C\rarrow P_\gamma$.
 Denote by $Q_\gamma=\im(\pi_\gamma)\subset P_\gamma$ the images of
the morphisms~$\pi_\gamma$.
 It is clear from the above discussion of kernels and cokernels
in $\Pro(\sC)$ that $\pi$~is an epimorphism in $\Pro(\sC)$ if and
only if $\plim_{\gamma\in\Gamma} P_\gamma/Q_\gamma=0$, i.~e.,
the projective system $(P_\gamma/Q_\gamma)_{\gamma\in\Gamma}$ is
pro-zero.

 Equivalently, this means that the morphism
$\plim_{\gamma\in\Gamma}Q_\gamma\rarrow\plim_{\gamma\in\Gamma}P_\gamma$
induced by the projective system of monomorphisms $Q_\gamma\rarrow
P_\gamma$ is an isomorphism in $\Pro(\sC)$.
 If this is the case, one can replace the projective system
$(P_\gamma)_{\gamma\in\Gamma}$ by the projective system
$(Q_\gamma)_{\gamma\in\Gamma}$.
 This allows one to assume that all the morphisms
$\pi_\gamma\:C\rarrow P_\gamma$ are epimorpisms in~$\sC$.
 Then all the transition morphisms $P_\delta\rarrow P_\gamma$
in the projective system $(P_\gamma)_{\gamma\in\Gamma}$ are
epimorphisms, too.

 Let $\sC$ be a complete abelian category.
 Then the forgetful functor $\Pro^\su(\sC)\rarrow\Pro(\sC)$ taking
a supplemented pro-object $(C,P)$ to the pro-object $P$ has
a right adjoint functor taking a pro-object $P\in\Pro(\sC)$ to
the supplemented pro-object $(\varprojlim P,\,P)\in\Pro^\su(\sC)$,
where $\pi\:C=\varprojlim P\rarrow P$ is the above adjunction morphism.

 We will say that a pro-object $P\in\Pro(\sC)$ is
\emph{limit-epimorphic} if the supplemented pro-object
$(\varprojlim P,\,P)$ is sup-epimorphic.
 Simply put, a pro-object $P\in\Pro(\sC)$ is limit-epimorphic if and
only if there exists a directed projective system
$(P_\gamma)_{\gamma\in\Gamma}$ in $\sC$ such that $P\simeq
\plim_{\gamma\in\Gamma}P_\gamma$ in $\Pro(\sC)$ and the projection
morphism $\varprojlim_{\gamma\in\Gamma}P_\gamma\rarrow P_\delta$
is an epimorphism in $\sC$ for every $\delta\in\Gamma$.

 A pro-object $P$ in an abelian category $\sC$ is said to be
\emph{strict} if there exists a directed projective system
$(P_\gamma)_{\gamma\in\Gamma}$ in $\sC$ such that all the transition
morphisms $P_\delta\rarrow P_\gamma$, \ $\gamma<\delta\in\Gamma$
are epimorphisms in $\sC$ and $P\simeq\plim_{\gamma\in\Gamma}P_\gamma$
in $\Pro(\sC)$.
 Following the discussion above, in any sup-epimorphic supplemented
pro-object $(C,P)\in\Pro^\su(\sC)$, the pro-object $P\in\Pro(\sC)$
is strict.
 In particular, assuming that $\sC$ is complete, any limit-epimorphic
pro-object in $\sC$ is strict.

\begin{rem}
 It is well-known that any pro-vector space or pro-abelian group
represented by a \emph{countable} directed projective system of
epimorphisms in $\Vect_k$ or $\Ab$ is limit-epimorphic in
the sense of the above definition.
 However, generally speaking, a strict pro-vector space need
\emph{not} be limit-epimorphic.
 For a counterexample of a directed projective system of surjective
linear maps between nonzero, countably-dimensional vector spaces,
indexed by the first uncountable ordinal~$\aleph_1$, whose projective
limit vanishes, see~\cite[Section~3]{HS}.
\end{rem}

\begin{prop}
 For any abelian category\/ $\sC$, the full subcategory of strict
pro-objects is closed under quotients and extensions in\/ $\Pro(\sC)$.
\end{prop}

\begin{proof}
 Let $f\:P\rarrow Q$ be a morphism in $\Pro(\sC)$ with a strict
pro-object~$P$.
 Then the construction from the above discussion of kernels and
cokernels in $\Pro(\sC)$ allows to represent~$f$ as the morphism
$f=\plim_{\xi\in\Xi}f_\xi\:\plim_{\xi\in\Xi}P_\xi\rarrow
\plim_{\xi\in\Xi}Q_\xi$ for some projective system of morphisms
$f_\xi\:P_\xi\rarrow Q_\xi$ in $\sC$ indexed by a directed poset~$\Xi$.
 Moreover, following the construction, one can choose
$(P_\xi)_{\xi\in\Xi}$ to be a projective system of epimorphisms
in~$\sC$.
 Then $(\im(f_\xi))_{\xi\in\Xi}$ is also a projective system of
epimorphisms in~$\sC$.
 If the morphism $P\rarrow Q$ is an epimorphism in $\Pro(\sC)$,
then $Q=\plim_{\xi\in\Xi}Q_\xi\simeq\plim_{\xi\in\Xi}\im(f_\xi)$
in $\Pro(\sC)$.
 This proves that the full subcategory of strict pro-objects is closed
under quotients.

 Let $0\rarrow P'\rarrow P\rarrow P''\rarrow0$ be a short exact
sequence in $\Pro(\sC)$.
 According to the same argument above, one can represent this short
exact sequence as the $\plim$ of a projective system of short
exact sequences $0\rarrow P'_\gamma\rarrow P_\gamma\rarrow P''_\gamma
\rarrow0$ indexed by a directed poset~$\Gamma$.
 Let $(Q'_\delta)_{\delta\in\Delta}$ be a directed projective
system in $\sC$ such that $\plim_{\delta\in\Delta}Q'_\delta\simeq
\plim_{\gamma\in\Gamma}P'_\gamma$ in $\Pro(\sC)$.
 Applying the same construction to the (iso)morphism
$\plim_{\gamma\in\Gamma}P'_\gamma\rarrow
\plim_{\delta\in\Delta}Q'_\delta$, we obtain a directed poset $\Xi$,
a projective system of short exact sequences $0\rarrow P'_\xi
\rarrow P_\xi\rarrow P''_\xi\rarrow0$ indexed by $\Xi$, whose
$\plim$ is the original short exact sequence $0\rarrow P'\rarrow P
\rarrow P''\rarrow0$, and a projective system of morphisms
$P'_\xi\rarrow Q'_\xi$ indexed by $\Xi$, whose $\plim$ is
the isomorphism $\plim_{\gamma\in\Gamma}P'_\gamma\rarrow
\plim_{\delta\in\Delta}Q'_\delta$.
 Moreover, following the construction, if
$(Q'_\delta)_{\delta\in\Delta}$ is a projective system of
epimorphisms in $\sC$, then so is $(Q'_\xi)_{\xi\in\Xi}$.

 Now, for every $\xi\in\Xi$, consider the pushout $0\rarrow Q'_\xi
\rarrow Q_\xi\rarrow P''_\xi\rarrow0$ of the short exact sequence
$0\rarrow P'_\xi\rarrow P_\xi\rarrow P''_\xi\rarrow0$ by
the morphism $P'_\xi\rarrow Q'_\xi$.
 Then we have an isomorphism $\plim_{\xi\in\Xi}Q_\xi\simeq
\plim_{\xi\in\Xi}P_\xi$ in $\Pro(\sC)$.
 The original short exact sequence $0\rarrow P'\rarrow P\rarrow P''
\rarrow0$ in $\Pro(\sC)$ can be obtained (up to an isomorphism of
short exact sequences in $\Pro(\sC)$) by applying $\plim$ to
the directed projective system of short exact sequences $0\rarrow
Q'_\xi\rarrow Q_\xi\rarrow P''_\xi\rarrow0$ in~$\sC$.

 Dually, let $(S''_\lambda)_{\lambda\in\Lambda}$ be a directed
projective system in $\sC$ such that $\plim_{\lambda\in\Lambda}
S''_\lambda\simeq\plim_{\gamma\in\Gamma}P''_\gamma$ in $\Pro(\sC)$.
 Applying once again the same construction to the (iso)mor\-phism
$\plim_{\lambda\in\Lambda}S''_\lambda\rarrow
\plim_{\xi\in\Xi}P''_\xi$, we obtain a directed poset $\Upsilon$,
a projective system of short exact sequences $0\rarrow Q'_\upsilon
\rarrow Q_\upsilon\rarrow P''_\upsilon\rarrow0$ indexed by $\Upsilon$,
whose $\plim$ is the original short exact sequence $0\rarrow P'
\rarrow P\rarrow P''\rarrow0$, and a projective system of morphisms
$S''_\upsilon\rarrow P''_\upsilon$ indexed by $\Upsilon$, whose
$\plim$ is the isomorphism $\plim_{\lambda\in\Lambda}S''_\lambda
\simeq\plim_{\gamma\in\Gamma}P''_\gamma$.
 Moreover, if $(S''_\lambda)_{\lambda\in\Lambda}$ is a projective
system of epimorphisms in $\sC$, then so is
$(S''_\upsilon)_{\upsilon\in\Upsilon}$.
 Besides, if $(Q'_\xi)_{\xi\in\Xi}$ is a projective system of
epimorphisms in $\sC$, then so is $(Q'_\upsilon)_{\upsilon\in\Upsilon}$.

 Finally, for every $\upsilon\in\Upsilon$, we consider the pullback
$0\rarrow Q'_\upsilon\rarrow T_\upsilon\rarrow S''_\upsilon\rarrow0$
of the short exact sequence $0\rarrow Q'_\upsilon\rarrow Q_\upsilon
\rarrow P''_\upsilon\rarrow0$ by the morphism $S''_\upsilon\rarrow
P''_\upsilon$.
 Then we have an isomorphism $\plim_{\upsilon\in\Upsilon}T_\upsilon
\simeq\plim_{\upsilon\in\Upsilon}Q_\upsilon$ in $\Pro(\sC)$.
  The original short exact sequence $0\rarrow P'\rarrow P\rarrow P''
\rarrow0$ in $\Pro(\sC)$ can be obtained (up to an isomorphism) by
applying $\plim$ to the directed projective system of short exact
sequences $0\rarrow Q'_\upsilon\rarrow T_\upsilon\rarrow S''_\upsilon
\rarrow0$ in~$\sC$.

 We have shown that any short exact sequence $0\rarrow P'\rarrow P
\rarrow P''\rarrow0$ in $\Pro(\sC)$ with strict pro-objects $P'$
and $P''$ can be obtained by applying $\plim$ to a directed
projective system of short exact sequences $0\rarrow
Q'_\upsilon\rarrow T_\upsilon\rarrow S''_\upsilon\rarrow0$ in $\sC$
such that both $(Q'_\upsilon)_{\upsilon\in\Upsilon}$ and
$(S''_\upsilon)_{\upsilon\in\Upsilon}$ are projective systems
of epimorphisms.
 It remains to observe that if $(Q'_\upsilon)_{\upsilon\in\Upsilon}$
and $(S''_\upsilon)_{\upsilon\in\Upsilon}$ are projective systems
of epimorphisms in $\sC$, then so is $(T_\upsilon)_{\upsilon\in
\Upsilon}$, because the class of all epimorphisms is closed under
extensions in the category of morphisms in~$\sC$.
\end{proof}

\begin{lem}
 For any complete abelian category\/ $\sC$, the full subcategory of
limit-epimorphic pro-objects is closed under quotients in\/ $\Pro(\sC)$.
\end{lem}

\begin{proof}
 If $P\rarrow Q$ is an epimorphism in $\Pro(\sC)$ and $\varprojlim P
\rarrow P$ is also an epimorphism in $\Pro(\sC)$ (where
$\varprojlim P\in\sC\subset\Pro(\sC)$), then it follows from
commutativity of the square diagram $\varprojlim P\rarrow P\rarrow Q$,
\ $\varprojlim P\rarrow\varprojlim Q\rarrow Q$ that $\varprojlim Q
\rarrow Q$ is an epimorphism in $\Pro(\sC)$.
\end{proof}

 We will denote the full subcategory of limit-epimorphic pro-objects
by $\Pro_\lie(\sC)\subset\Pro(\sC)$.
 The following theorem explains why we are interested in
limit-epimorphic pro-objects in connection with topological algebra.

\begin{thm} \label{complete-top-vector-spaces-as-proobjects}
\textup{(a)} There is a natural equivalence of additive categories\/
$\Top^\scc_\boZ\simeq\Pro_\lie(\Ab)$ assigning to a limit-epimorphic
pro-abelian group $P$ the abelian group $A=\varprojlim P$ endowed
with the topology of projective limit of discrete abelian groups. \par
\textup{(b)} There is a natural equivalence of additive categories\/
$\Top^\scc_k\simeq\Pro_\lie(\Vect_k)$ assigning to a limit-epimorphic
pro-vector space $P$ the vector space $V=\varprojlim P$ endowed
with the topology of projective limit of discrete vector spaces.
\end{thm}

\begin{proof}
 Both the equivalences are almost obvious (see
Propositions~\ref{nonseparated-top-vector-spaces-as-suppl-proobjects}%
\+-\ref{incomplete-top-vector-spaces-as-suppl-proobjects} below for
generalizations).
 Let us explain part~(b); part~(a) is similar.

 Given a pro-vector space $P=\plim_{\gamma\in\Gamma}P_\gamma$,
the projective limit topology (of discrete vector spaces~$P_\gamma$)
on $V=\varprojlim P=\varprojlim_{\gamma\in\Gamma}P_\gamma$ has
a base of neighborhoods of zero consisting of the kernels of
the projection maps $V\rarrow P_\gamma$.
 One can readily check that this topology on $\varprojlim P$ is
well-defined and functorial, i.~e., for any morphism of pro-vector
spaces $\plim_{\gamma\in\Gamma}P_\gamma\rarrow\plim_{\delta\in\Delta}
Q_\delta$ the induced map of projective limits
$\varprojlim_{\gamma\in\Gamma}P_\gamma\rarrow\varprojlim_{\delta\in
\Delta}Q_\delta$ is continuous.
 Moreover, the topological vector space $\varprojlim_{\gamma\in\Gamma}
P_\gamma$ is separated and complete (in fact, it is a closed subspace
in the complete, separated topological vector space
$\prod_{\gamma\in\Gamma}P_\gamma$, with the product topology of
discrete vector spaces on $\prod_{\gamma\in\Gamma}P_\gamma$ and
the induced topology of a subspace in $\prod_{\gamma\in\Gamma}P_\gamma$
on $\varprojlim_{\gamma\in\Gamma}P_\gamma$).
 This defines the desired functor $\Pro_\lie(\Vect_k)\subset
\Pro(\Vect_k)\rarrow\Top^\scc_k$.

 The inverse functor $\Top^\scc_k\subset\Top_k\rarrow\Pro_\lie(\Vect_k)$
assigns to a topological vector space $V$ the pro-vector space
$P=\plim_{U\subset V}V/U$, where $U$ ranges over the open vector
subspaces in~$V$.
 The pro-vector space $P$ is limit-epimorphic, because the composition
$V\rarrow \varprojlim P\rarrow V/U$ is surjective for all open vector
subspaces $U\subset V$, and consequently the map $\varprojlim P
\rarrow V/U$ is surjective.
 In fact, we have constructed a pair of adjoint functors between
the categories $\Top_k$ and $\Pro(\Vect_k)$, with the left adjoint
functor $V\longmapsto\plim_{U\subset V}V/U$ and the right adjoint
functor $P\longmapsto\varprojlim P$.
 Moreover, the former functor takes values inside
$\Pro_\lie(\Vect_k)\subset\Pro(\Vect_k)$, while the latter one lands
within $\Top^\scc_k\subset\Top_k$.
 Checking that the restrictions of these functors to $\Top^\scc_k$
and $\Pro_\lie(\Vect_k)$ are mutually inverse equivalences is left
to the reader.
\end{proof}

 It is easy to see that the class of all limit-epimorphic pro-objects
is \emph{not} closed under subobjects in $\Pro(\sC)$.
 In fact, for any directed projective system $(P_\gamma)_{\gamma\in
\Gamma}$ one can construct a directed projective system of split
monomorphisms $P_\gamma\rarrow Q_\gamma$ in $\sC$ such that
$\varprojlim_{\gamma\in\Gamma}Q_\gamma\rarrow Q_\delta$ is
a split epimorphism in $\sC$ for every $\delta\in\Gamma$.
 It suffices to put $Q_\delta=\prod_{\gamma\le\delta}P_\gamma$
for every $\delta\in\Gamma$.
 The following counterexamples show that the class of all
limit-epimorphic pro-objects is \emph{not} closed under extensions
in $\Pro(\sC)$, either (generally speaking).

\begin{prop} \label{does-not-inherit-prop}
\textup{(a)} The full subcategory\/ $\Top^\scc_\boZ\simeq\Pro_\lie(\Ab)
\subset\Pro(\Ab)$ does \emph{not} inherit an exact category structure
from the abelian exact structure of the abelian category\/ $\Pro(\Ab)$.
 In fact, a short sequence in\/ $\Top^\scc_\boZ$ is exact in\/
$\Pro(\Ab)$ if and only if it satisfies Ex1 in\/ $\Top^\scc_\boZ$
(but the class of all short sequences satisfying Ex1 is not
an exact category structure on\/ $\Top^\scc_\boZ$).
 In particular, the full subcategory\/ $\Pro_\lie(\Ab)$ is
\emph{not} closed under extensions in\/ $\Pro(\Ab)$. \par
\textup{(b)} The full subcategory\/ $\Top^\scc_k\simeq\Pro_\lie(\Vect_k)
\subset\Pro(\Vect_k)$ does \emph{not} inherit an exact category
structure from the abelian exact structure of the abelian category\/
$\Pro(\Vect_k)$.
 In fact, a short sequence in\/ $\Top^\scc_k$ is exact in\/
$\Pro(\Vect_k)$ if and only if it satisfies Ex1 in\/ $\Top^\scc_k$
(but the class of all short sequences satisfying Ex1 is not
an exact category structure on\/ $\Top^\scc_k$).
 In particular, the full subcategory\/ $\Pro_\lie(\Vect_k)$ is
\emph{not} closed under extensions in\/ $\Pro(\Vect_k)$.
\end{prop}

\begin{proof}
 Let us explain part~(b); part~(a) is similar.
 Clearly, any short sequence in $\Pro_\lie(\Vect_k)$ which satisfies
Ex1 in $\Pro(\Vect_k)$ also satisfies Ex1 in $\Pro_\lie(\Vect_k)$.
 Conversely, the class of all short sequences satisfying Ex1 in
$\Top^\scc_k$ was described in
Proposition~\ref{top-groups-spaces-kernels-cokernels-prop}.
 Such sequences have the form $0\rarrow\fK\rarrow\fV\rarrow
(\fV/\fK)\sphat\,\rarrow0$, where $\fV$ is a complete, separated
topological vector space, $\fK\subset\fV$ is a closed subspace with
the induced topology, $\fV/\fK$ is the quotient space with the quotient
topology, and $(\fV/\fK)\sphat\,$ is the completion of $\fV/\fK$,
endowed with the completion topology.
 In this context, consider the projective system of short exact
sequences of vector spaces
\begin{equation} \label{quotient-completion-sequence}
 0\lrarrow(\fK+\fU)/\fU\lrarrow\fV/\fU\lrarrow\fV/(\fK+\fU)
 \rarrow0
\end{equation}
indexed by the directed poset of all open subspaces $\fU\subset\fV$.
 Passing to $\plim$ produces from this projective system a short
exact sequence of pro-objects in $\Vect_k$ corresponding to
the short sequence $0\rarrow\fK\rarrow\fV\rarrow
(\fV/\fK)\sphat\,\rarrow0$ in $\Top^\scc_k$.

 The class of all short sequences satisfying Ex1 in $\Top^\scc_k$ is
not an exact category structure as the category $\Top^\scc_k$ is
not quasi-abelian (see Corollary~\ref{not-left-quasi-abelian-cor}).
 In view of Example~\ref{inherited-exact-structure}, it follows that
the full subcategory $\Pro_\lie(\Vect_k)$ cannot be closed under
extensions in $\Pro(\Vect_k)$.
\end{proof}
 
\begin{ex}
 Let us spell out an explicit example of a short exact sequence of
pro-vector spaces whose leftmost and rightmost terms are
limit-epimorphic, but the middle term isn't.
 A straightforward adaptation of the counterexample
from Corollary~\ref{not-left-quasi-abelian-cor} and
Examples~\ref{not-left-semi-abelian-examples} is sufficient
for this purpose.

 Let $\fV$ be a complete, separated topological vector space with
a closed subspace $\fK\subset\fV$ such that the quotient space
$C=\fV/\fK$ is incomplete in the quotient topology.
 Put $\fC=C\sphat\,$, and choose a vector $x\in\fC\setminus C$.
 Consider the short exact sequence of pro-vector spaces
$$
 0\lrarrow\plim_{\fU\subset\fV}(\fK+\fU)/\fU\lrarrow
 \plim_{\fU\subset\fV}\fV/\fU\lrarrow\plim_{\fU\subset\fV}
 \fV/(\fK+\fU)\lrarrow0
$$
obtained by applying $\plim$ to~\eqref{quotient-completion-sequence}.

 For every open subspace $\fU\subset\fV$, denote by $W_\fU\subset C$
the open subspace $W_\fU=(\fK+\fU)/\fK\subset\fV/\fK=C$, and let
$\fW_\fU\subset\fC$ be the related open subspace in~$\fC$.
 Then we have natural isomorphisms of discrete quotient spaces
$\fV/(\fK+\fU)\simeq C/W_\fU\simeq\fC/\fW_\fU$.
 Hence the choice of an element $x\in\fC$ induces a projective system
of linear maps $k\overset x\rarrow\fV/(\fK+\fU)$ indexed by
the open subspaces $\fU\subset\fV$.
 Taking the pullbacks of~\eqref{quotient-completion-sequence} with
respect to these maps, we obtain a projective system of pullback
diagrams of short exact sequences in $\Vect_k$ and the related
pullback diagram of short exact sequences of pro-vector spaces
\begin{equation} \label{provector-counterex}
\begin{gathered}
\xymatrix{
 \plim_{\fU\subset\fV}(\fK+\fU)/\fU \ar[r] \ar[rd]
 & \plim_{\fU\subset\fV}\fV/\fU \ar[r]
 & \plim_{\fU\subset\fV}\fV/(\fK+\fU) \\
 & \plim_{\fU\subset\fV}(kx\sqcap_{\fV/(\fK+\fU)}\fV/\fU)
 \ar[u]\ar[r] & \plim_{\fU\subset\fV}kx \ar[u]
}
\end{gathered}
\end{equation}

 The lower line in~\eqref{provector-counterex} is a short exact
sequence in $\Pro(\Vect_k)$ (obtained by taking $\plim$ of
the projective system of short exact sequences of vector spaces
$0\rarrow(\fK+\nobreak\fU)/\fU\rarrow kx\sqcap_{\fV/(\fK+\fU)}\fV/\fU
\rarrow kx\rarrow0$ indexed by the directed poset of open
subspaces $\fU\subset\fV$).
 In this short exact sequence of pro-vector spaces, the leftmost term
$\plim_{\fU\subset\fV}(\fK+\fU)/\fU$ is a limit-epimorphic pro-vector
space (corresponding to the complete, separated topological vector
space~$\fK$).
 The rightmost term $\plim_{\fU\subset\fV}kx=kx\in\Vect_k\subset
\Pro(\Vect_k)$ is also a limit-epimorphic pro-vector space
(corresponding to the discrete one-dimensional
topological vector space~$kx$).

 However, the middle term
$S=\plim_{\fU\subset\fV}(kx\sqcap_{\fV/(\fK+\fU)}\fV/\fU)$
is \emph{not} limit-epimorphic.
 In fact, one has $\varprojlim S=kx\sqcap_\fC\fV=\fK$, so the image of
the morphism $\varprojlim S\rarrow S$ in $\Pro(\Vect_k)$ is
the subobject $\plim_{\fU\subset\fV}(\fK+\fU)/\fU=
\ker(S\to kx)\subset S$.
\end{ex}

\Section{Supplemented Pro-Vector Spaces} \label{suppl-pro-secn}

 The definition of the category of supplemented pro-objects
$\Pro^\su(\sC)$ in an abelian category $\sC$ was given in
the previous Section~\ref{pro-vector-spaces-secn}.
 The notion of a supplemented pro-object allows to formulate
a category-theoretic interpretation or generalization of
the theory of topological abelian groups/vector spaces with
linear topology.
 Both the incomplete and the complete topological vector spaces can be
interpreted as full subcategories in supplemented pro-vector spaces.

\begin{lem}
 For any abelian category\/ $\sC$, the category of supplemented
pro-objects\/ $\Pro^\su(\sC)$ is abelian.
 The forgetful functors\/ $\Pro^\su(\sC)\rarrow\Pro(\sC)$ and\/
$\Pro^\su(\sC)\rarrow\sC$ are exact.
\end{lem}

\begin{proof}
 The category of morphisms in an abelian category is abelian;
and $\Pro^\su(\sC)$ is a full additive subcategory closed under
kernels and cokernels in the category of morphisms in $\Pro(\sC)$.
\end{proof}

 We refer to Section~\ref{pro-vector-spaces-secn} for the definition of
a \emph{sup-epimorphic} supplemented pro-object.
 Furthermore, we will say that a supplemented pro-object
$(C,P)$ in $\sC$ is \emph{jointly sup-monomorphic} if for
every nonzero morphism $E\rarrow C$ in $\sC$ the composition
$E\rarrow C\rarrow P$ is a nonzero morphism in $\Pro(\sC)$.
 Equivalently, this means that for any nonzero morphism $E\rarrow C$
in $\sC$ there exists an object $D\in\sC$ and a morphism $P\rarrow D$
in $\Pro(\sC)$ such that the composition $E\rarrow C\rarrow P\rarrow D$
is a nonzero morphism in~$\sC$.
 Let $(P_\gamma)_{\gamma\in\Gamma}$ be a directed projective system
in $\sC$ such that $P=\plim_{\gamma\in\Gamma}P_\gamma$, and let 
$\pi_\gamma\:C\rarrow P_\gamma$, \,$\gamma\in\Gamma$, be
a compatible cone of morphisms representing the morphism~$\pi$.
 Then $(C,P)$ is jointly sup-monomorphic if and only if,
for every nonzero morphism $E\rarrow C$ in $\sC$, there exists
$\gamma\in\Gamma$ such that the composition $E\rarrow C\rarrow P
\rarrow P_\gamma$ is nonzero.

 We will denote the full additive subcategory of sup-epimorphic
supplemented pro-objects by $\Pro^\su_\se(\sC)\subset\Pro^\su(\sC)$.
 The full additive subcategory of sup-epimorphic, jointly
sup-monomorphic supplemented pro-objects will be denoted by
$\Pro^\su_{\se,\jsm}(\sC)\subset\Pro^\su_\se(\sC)\subset\Pro^\su(\sC)$.

\begin{lem} \label{subcategories-of-supplemented-proobjects}
 Let\/ $\sC$ be an abelian category.  Then \par
\textup{(a)} the full subcategory\/ $\Pro^\su_\se(\sC)$
of sup-epimorphic supplemented pro-objects is closed under quotients
and extensions in the abelian category\/ $\Pro^\su(\sC)$; \par
\textup{(b)} the full subcategory of jointly sup-monomorphic
supplemented pro-objects is closed under subobjects and extensions
in the abelian category\/ $\Pro^\su(\sC)$; \par
\textup{(c)} the full subcategory\/ $\Pro^\su_{\se,\jsm}(\sC)$
of sup-epimorphic, jointly sup-monomorphic supplemented pro-objects
is closed under subobjects in the additive category of sup-epimorphic
supplemented pro-objects in\/~$\sC$.
\end{lem}

\begin{proof}
 We will only prove part~(c) and closedness under extensions in
part~(b), which is not difficult, but all the other assertions of
the lemma are even easier.

 Let $0\rarrow(C',P')\rarrow(C,P)\rarrow(C'',P'')\rarrow0$ be
a short exact sequence in the abelian category $\Pro^\su(\sC)$.
 Assume that the supplemented pro-objects $(C',P')$ and $(C'',P'')$
are jointly sup-monomorphic, and let $E\rarrow C$ be a morphism
in $\sC$ such that the composition $E\rarrow C\rarrow P$ is
a zero morphism in $\Pro(\sC)$.
 Then the composition $E\rarrow C\rarrow C''\rarrow P''$ is also
a zero morphism in $\Pro(\sC)$, hence the composition $E\rarrow C
\rarrow C''$ is a zero morphism in~$\sC$.
 It follows that the morphism $E\rarrow C$ factorizes through
the monomorphism $C'\rarrow C$, leading to a morphism $E\rarrow C'$.
 Now the composition $E\rarrow C'\rarrow P'$ is a zero morphism
in $\Pro(\sC)$, since the morphism $P'\rarrow P''$ is a monomorphism.
 Consequently, the morphism $E\rarrow C'$ is zero, and therefore
the original morphism $E\rarrow C$ is zero as well.

 Let $(f,g)\:(C',P')\rarrow(C'',P'')$ be a morphism of sup-epimorphic
supplemented pro-objects.
 Put $K=\ker(f)\in\sC$ and $Q=\ker(g)\in\Pro(\sC)$.
 Then there is the induced morphism $K\rarrow Q$ in $\Pro(\sC)$,
making $(K,Q)$ a supplemented pro-object.
 Denote by $S$ the image of the morphism $K\rarrow Q$ in $\Pro(\sC)$.
 Then $(K,S)$ is a sup-epimorphic supplemented pro-object, and one
can check that the composition $(K,S)\rarrow(K,Q)\rarrow(C',P')$ is
the kernel of the morphism $(f,g)$ in the category of sup-epimorphic
supplemented pro-objects in~$\sC$.
 In particular, if $(f,g)$ is a monomorphism in the category of
sup-epimorphic supplemented pro-objects, then $K=0$ and therefore
$f\:C'\rarrow C''$ is a monomorphism in~$\sC$.

 Now assume that $(C'',P'')$ is a jointly sup-monomorphic sup-epimorphic
supplemented pro-object, so the composition $E\rarrow C''\rarrow P''$ is
nonzero in $\Pro(\sC)$ for any nonzero morphism $E\rarrow C''$ in~$\sC$.
 Let $E\rarrow C'$ be a nonzero morphism in $\sC$; then the composition
$E\rarrow C'\rarrow C''$ is nonzero as well.
 It follows that the composition $E\rarrow C'\rarrow C''\rarrow P''$
is nonzero, and consequently the composition $E\rarrow C'\rarrow P'$
cannot vanish.
\end{proof}

\begin{lem} \label{quasi-abelian-reflective-subcategory}
\textup{(a)} Let\/ $\sA$ be a left semi-abelian additive category and\/
$\sB\subset\sA$ be a reflective full subcategory with the reflector\/
$\Psi\:\sA\rarrow\sB$ (i.~e., $\Psi$~is left adjoint to
the inclusion functor\/ $\sB\rarrow\sA$).
 Then the full subcategory\/ $\sB$ is closed under subobjects in\/
$\sA$ if and only if the adjunction morphism\/ $\psi_A\:A\rarrow
\Psi(A)$ is a cokernel in\/ $\sA$ for every $A\in\sA$. \par
\textup{(b)} Let\/ $\sA$ be a quasi-abelian additive category and\/
$\sB\subset\sA$ be a reflective full subcategory closed under
subobjects and extensions.
 Then the additive category\/ $\sB$ is quasi-abelian, and
the quasi-abelian exact category structure on\/ $\sB$ coincides with
the exact category structure inherited from the quasi-abelian exact
category structure on\/~$\sA$.
\end{lem}

\begin{proof}
 Part~(a): assume that $\psi_A\:A\rarrow\Psi(A)$ is a cokernel
in $\sA$ for all $A\in\sA$.
 Let $B\in\sB$ be an object and $A\rarrow B$ be a monomorphism
in~$\sA$.
 Then the composition $A\rarrow B\rarrow\Psi(B)$ is a monomorphism,
too, since $\psi_B\:B\rarrow\Psi(B)$ is an isomorphism.
 As the composition $A\rarrow\Psi(A)\rarrow\Psi(B)$ is the same
morphism, it follows that $\psi_A\:A\rarrow\Psi(A)$ is
a monomorphism.
 Now a morphism that is simultaneously a cokernel and a monomorphism
must be an isomorphism; hence $\psi_A$~is an isomorphism, and
$A\in\sB$ because $\Psi(A)\in\sB$.

 Conversely, assume that $\sB$ is closed under subobjects in~$\sA$.
 The assumption that $\sA$ is left semi-abelian means that any morphism
$f\:C\rarrow D$ in $\sA$ factorizes as a cokernel $C\rarrow\coim(f)$
followed by a monomorphism $\coim(f)\rarrow D$.
 Given an object $A\in\sA$, denote by $B=\coim(\psi_A)$ the coimage
of the morphism $\psi_A\:A\rarrow\Psi(A)$.
 Since $\sB$ is closed under subobjects in $\sA$ and $\Psi(A)\in\sB$,
we have $B\in\sB$.
 In view of the universal property of the morphism~$\psi_A$,
the epimorphism $A\rarrow B$ factorizes uniquely through~$\psi_A$;
so we get a morphism $\Psi(A)\rarrow B$.
 It follows easily that the pair of morphisms $B\rarrow\Psi(A)$
and $\Psi(A)\rarrow B$ are mutually inverse isomorphisms.
 It remains to recall that the morphism $A\rarrow B$ is a cokernel
by construction.

 Part~(b): first of all, a reflective full subcategory $\sB$ in
a category $\sA$ with kernels and cokernels also has kernels
and cokernels.
 In fact, any reflective full subcategory $\sB\subset\sA$ is closed
under kernels; and the cokernel of a morphism $f\:B'\rarrow B''$ in
$\sB$ can be constructed by applying the reflector $\Psi$ to
the cokernel of~$f$ in~$\sA$.

 Furthermore, if $\sA$ is left semi-abelian and all the adjunction
morphisms $\psi_A\:A\rarrow\Psi(A)$ are cokernels in $\sA$,
then a morphism $p\:B\rarrow C$ in $\sB$ is a cokernel if and only
if $p$~is a cokernel in~$\sA$.
 Indeed, let $i\:K\rarrow B$ be a kernel of~$p$ in $\sA$, then
$K\in\sB$ and $i$~is a kernel of~$p$ in~$\sB$.
 If $p$~is a cokernel of~$i$ in~$\sA$, then $p$~is also a cokernel
of~$i$ in~$\sB$.
 Conversely, denote by $q\:B\rarrow A$ the cokernel of~$i$ in~$\sA$.
 If $p$~is a cokernel of~$i$ in~$\sB$, then $p$~is the composition
$B\overset q\rarrow A\overset{\psi_A}\rarrow\Psi(A)\simeq C$.
 Since $\sA$ is left semi-abelian, the composition of two cokernels
is a cokernel in~$\sA$.

 Under the same assumptions, a morphism $i\:K\rarrow B$ in $\sB$ is
a kernel if and only if $i$~is a kernel in $\sA$ \emph{and}
the cokernel $\coker_\sA(i)$ of the morphism~$i$ in $\sA$
belongs to~$\sB$.
 Indeed, if $i$~is a kernel in $\sA$, then $i$~is a kernel of
the morphism $B\rarrow\coker_\sA(i)=C$.
 If $C\in\sB$, then $i$~is also a kernel of the morphism $B\rarrow C$
in~$\sB$.
 Conversely, assume that $i$~is a kernel in~$\sB$.
 The cokernel of~$i$ in $\sB$ is computable as $\coker_\sB(i)=
\Psi(C)$.
 So $i$~is a kernel of the morphism $B\rarrow\Psi(C)$ in~$\sB$.
 Then $i$~is also a kernel of the morphism $B\rarrow\Psi(C)$
in~$\sA$.
 By our assumptions, the composition of two cokernels $B\rarrow C
\rarrow\Psi(C)$ is a cokernel in~$\sA$.
 So $B\rarrow\Psi(C)$ is a cokernel of the morphism~$i$ in~$\sA$.
 It follows that $C\rarrow\Psi(C)$ is an isomorphism and $C\in\sB$.

 Now it is clear that a short sequence in $\sB$ satisfies Ex1 if
and only if it satisfies Ex1 in~$\sA$.
 Since $\sB$ is closed under extensions in $\sA$ by assumption,
$\sB$ inherits an exact category structure from the quasi-abelian
exact category structure on~$\sA$
(by Example~\ref{inherited-exact-structure}).
 So the class of all short sequences satisfying Ex1 is an exact
category structure on $\sB$; in other words, this means that
$\sB$ is quasi-abelian. 
\end{proof}

\begin{cor} \label{sup-epimorphic-quasi-abelian-cor}
\textup{(a)} For any abelian category\/ $\sC$,
the additive category\/ $\Pro^\su_\se(\sC)$ of sup-epimorphic
supplemented pro-objects in\/ $\sC$ is quasi-abelian.
 The exact category structure on\/ $\Pro^\su_\se(\sC)$ inherited
from the abelian exact structure of the ambient abelian category
of supplemented pro-objects\/ $\Pro^\su(\sC)$ coincides with
the quasi-abelian exact category structure on\/ $\Pro^\su_\se(\sC)$.
\par
\textup{(b)} Assume that the intersections of families of subobjects
in any given object exist in an abelian category\/~$\sC$.
 Then the additive category\/ $\Pro^\su_{\se,\jsm}(\sC)$ of
sup-epimorphic, jointly sup-monomorphic supplemented pro-objects
in\/ $\sC$ is quasi-abelian.
 The exact category structure on\/ $\Pro^\su_{\se,\jsm}(\sC)$ inherited
from the quasi-abelian exact structure of the ambient quasi-abelian
category of sup-epimorphic supplemented pro-objects\/
$\Pro^\su_\se(\sC)$ coincides with the quasi-abelian exact category
structure on\/ $\Pro^\su_{\se,\jsm}(\sC)$.
\end{cor}

\begin{proof}
 Part~(a) is a particular case of the opposite version of
Lemma~\ref{quasi-abelian-reflective-subcategory} (with an abelian
ambient category).
 The full subcategory $\Pro^\su_\se(\sC)$ is closed under quotients
and extensions in $\Pro^\su(\sC)$ by
Lemma~\ref{subcategories-of-supplemented-proobjects}(a).
 Furthermore, the full subcategory $\Pro^\su_\se(\sC)$ is coreflective
in $\Pro^\su(\sC)$: the coreflector $\Phi\:\Pro^\su(\sC)\rarrow
\Pro^\su_\se(\sC)$ assigns to any supplemented pro-object
$(C,P)\in\Pro^\su(\sC)$ the sup-epimorphic  supplemented pro-object
$\Phi(C,P)=(C,Q)$, where $Q$ is the image of the morphism
$\pi\:C\rarrow P$ in the abelian category $\Pro(\sC)$.
 The adjunction morphism $\phi_{(C,P)}\:(C,Q)\rarrow(C,P)$ is
a monomorphism in the abelian category $\Pro^\su(\sC)$.

 Part~(b) is a particular case of
Lemma~\ref{quasi-abelian-reflective-subcategory}.
 The full subcategory $\Pro^\su_{\se,\jsm}(\sC)$ is closed under
subobjects and extensions in $\Pro^\su_\se(\sC)$ by part~(a) and
Lemma~\ref{subcategories-of-supplemented-proobjects}(b\+-c).
 The full subcategory $\Pro^\su_{\se,\jsm}(\sC)$ is reflective in
$\Pro^\su_\se(\sC)$: the reflector $\Psi\:\Pro^\su_\se(\sC)\rarrow
\Pro^\su_{\se,\jsm}(\sC)$ assigns to any sup-epimorphic supplemented
pro-object $(C,P)\in\Pro^\su_\se(\sC)$ the sup-epimorphic, jointly
sup-monomorphic supplemented pro-object $\Psi(C,P)=(C/K,P)\in
\Pro^\su_{\se,\jsm}(\sC)$.
 Here $K\in\sC$ is a subobject in $C$ constructed as follows.
 Let $(P_\gamma)_{\xi\in\Gamma}$ be a directed projective system
in $\sC$ such that $P=\plim_{\gamma\in\Gamma}P_\gamma$, and let
$\pi_\gamma\:C\rarrow P_\gamma$, \,$\gamma\in\Gamma$, be
the compatible cone of morphisms in $\sC$ representing the morphism
$\pi\:C\rarrow P$ in $\Pro(\sC)$.
 Then $K$ is the intersection of the subobjects $\ker(\pi_\gamma)$
in the object $C\in\sC$.
 The adjunction morphism $\psi_{(C,P)}\:(C,P)\rarrow(C/K,P)$ is
a cokernel of the morphism $(K,0)\rarrow(C,P)$ in
$\Pro^\su_\se(\sC)$.
\end{proof}

\begin{prop} \label{nonseparated-top-vector-spaces-as-suppl-proobjects}
\textup{(a)} There is a natural equivalence of additive categories\/
$\Top_\boZ\simeq\Pro^\su_\se(\Ab)$ assigning to a topological
abelian group $A$ with linear topology the sup-epi\-morphic
supplemented pro-abelian group $(C,P)$ with $C=A$ and
$P=\plim_{U\subset A}A/U$, where $U$ ranges over the directed poset
of all open subgroups in~$A$. \par
\textup{(b)} There is a natural equivalence of additive categories\/
$\Top_k\simeq\Pro^\su_\se(\Vect_k)$ assigning to a topological vector
space $V$ with linear topology the sup-epimorphic supplemented
pro-vector space $(C,P)$ with $C=V$ and $P=\plim_{U\subset V}V/U$,
where $U$ ranges over the directed poset of all open subspaces in~$V$.
\end{prop}

\begin{proof}
 Both the assertions are essentially obvious.
 Let us say a few words about~(b).
 Given a topological vector space $V$, the morphism of pro-vector
spaces $\pi\:V\rarrow\plim_{U\subset V}V/U$ is defined by
the compatible cone of natural surjections $V\rarrow V/U$.
 The morphism~$\pi$ is an epimorphism in $\Pro(\Vect_k)$, since
all the maps $V\rarrow V/U$ are epimorphisms in $\Vect_k$.
 This defines the functor $\Top_k\rarrow\Pro^\su_\se(\Vect_k)$.
 The inverse functor assigns to every supplemented pro-vector space
$(C,P)$ the vector space $C\in\Vect_k$ endowed with the topology
in which a vector subspace $U\subset C$ is open if and only if $U$,
viewed as a subobject of the object $C\in\Vect_k\subset\Pro(\Vect_k)$,
contains the subobject $\ker(\pi)\subset C$, \, $\ker(\pi)\in
\Pro(\Vect_k)$.
 Part~(a) is similar.
\end{proof}

\begin{prop} \label{incomplete-top-vector-spaces-as-suppl-proobjects}
\textup{(a)} The equivalence of additive categories of topological
abelian groups and sup-epimorphic supplemented pro-abelian groups\/
$\Top_\boZ\simeq\Pro^\su_\se(\Ab)$ from
Proposition~\ref{nonseparated-top-vector-spaces-as-suppl-proobjects}(a)
restricts to an equivalence between the full subcategories of
separated topological abelian groups and sup-epimorphic, jointly
sup-monomorphic pro-abelian groups,
$\Top^\s_\boZ\simeq\Pro^\su_{\se,\jsm}(\Ab)$. \par
\textup{(b)} The equivalence of additive categories of topological
vector spaces and sup-epimorphic supplemented pro-vector spaces\/
$\Top_k\simeq\Pro^\su_\se(\Vect_k)$ from
Proposition~\ref{nonseparated-top-vector-spaces-as-suppl-proobjects}(b)
restricts to an equivalence between the full subcategories of
separated topological vector spaces and sup-epimorphic, jointly
sup-monomorphic pro-vector spaces,
$\Top^\s_k\simeq\Pro^\su_{\se,\jsm}(\Vect_k)$.
\end{prop}

\begin{proof}
 Part~(b): given a topological vector space $V$ and the related
(sup-epimorphic) supplemented pro-vector space $(V,P)$, a morphism
of vector spaces $f\:E\rarrow V$ vanishes in composition with
the morphism $\pi\:V\rarrow P$ if and only if the image of~$f$ is
contained in the closure of the zero subgroup in the topological
vector space~$V$.
 Part~(a) is similar.
\end{proof}

 Let $\sC$ be a complete abelian category.
 Then, following the discussion in Section~\ref{pro-vector-spaces-secn},
the exact forgetful functor $\Pro^\su(\sC)\rarrow\Pro(\sC)$ has a right
adjoint functor taking a pro-object $P\in\Pro(\sC)$ to the supplemented
pro-object $(\varprojlim P,\,P)\in\Pro^\su(\sC)$.
 The fully faithful functor $P\longmapsto(\varprojlim P,\,P)$ allows
to consider $\Pro(\sC)$ as a full subcategory in $\Pro^\su(\sC)$.
 Notice that all the supplemented pro-objects in $\Pro(\sC)\subset
\Pro^\su(\sC)$ are jointly sup-monomorphic.
 The following lemma implies that the full subcategory 
$\Pro(\sC)\subset\Pro^\su(\sC)$ is closed under extensions.

\begin{lem} \label{abelian-reflective-closed-under-extensions}
 Let\/ $\sA$ and\/ $\sB$ be abelian categories and\/
$\Theta\:\sA\rarrow\sB$ be an exact functor with a fully faithful
right adjoint functor\/ $\Lambda\:\sB\rarrow\sA$ (so the composition
$\Theta\Lambda\:\sB\rarrow\sB$ is the identity functor).
 Then the essential image $\Lambda(\sB)\subset\sA$ of the functor
$\Lambda$ is closed under extensions as a full subcategory in
the abelian category~$\sA$.
\end{lem}

\begin{proof}
 Let $0\rarrow B'\rarrow A\rarrow B''\rarrow0$ be a short
exact sequence in $\sA$ with $B'$, $B''\in\Lambda(\sB)$.
 Then $0\rarrow\Theta(B')\rarrow\Theta(A)\rarrow\Theta(B'')\rarrow0$
is a short exact sequence in~$\sB$.
 The functor $\Lambda$ is left exact (as a right adjoint functor
between abelian categories), so $0\rarrow\Lambda\Theta(B')\rarrow
\Lambda\Theta(A)\rarrow\Lambda\Theta(B'')$ is a left exact
sequence in~$\sA$.
 Now we have a commutative diagram of adjunction morphism
of short sequences
$$
\xymatrix{
 0\ar[r] & B' \ar[r]\ar@{=}[d] & A \ar[r]\ar[d]
 & B'' \ar[r]\ar@{=}[d] & 0 \\
 0\ar[r] & \Lambda\Theta(B') \ar[r] & \Lambda\Theta(A) \ar[r]
 & \Lambda\Theta(B'')
}
$$
where the morphisms $B'\rarrow\Lambda\Theta(B')$ and $B''\rarrow
\Lambda\Theta(B'')$ are isomorphisms.
 It follows that $\Lambda\Theta(A)\rarrow\Lambda\Theta(B'')$ is
an epimorphism and $A\rarrow\Lambda\Theta(A)$ is an isomorphism;
so $A\in\Lambda(\sB)\subset\sA$.
\end{proof}

\begin{lem} \label{reflective-coreflective-lemma}
 Let\/ $\sA$ and\/ $\sB$ be abelian categories and\/
$\Theta\:\sA\rarrow\sB$ be an exact functor with a fully faithful
right adjoint\/ $\Lambda\:\sB\rarrow\sA$ (so the essential image\/
$\Lambda(\sB)\subset\sA$ of the functor\/ $\Lambda$ is a reflective
full subcategory in\/~$\sA$ with the reflector~$\Lambda\Theta$).
 Furthermore, let\/ $\sE\subset\sA$ be a coreflective full
subcategory with the coreflector\/ $\Phi\:\sA\rarrow\sE$; suppose
that\/ $\sE$ is closed under quotients and extensions in\/~$\sA$.
 Assume that\/ $\Phi(\Lambda(\sB))\subset\Lambda(\sB)$ and\/
$\Lambda\Theta(\sE)\subset\sE$ in\/~$\sA$.
 Then \par
\textup{(a)} the full subcategory\/ $\sF=\Lambda(\sB)\cap\sE\subset\sA$
is right quasi-abelian; \par
\textup{(b)} a morphism in\/ $\sF$ is a kernel in\/ $\sF$ if and only
if it is a monomorphism in\/~$\sA$; \par
\textup{(c)} the full subcategory\/ $\sF\subset\sA$ is closed under
extensions, so it inherits an exact category structure from the abelian
exact structure of the abelian category\/~$\sA$.
\end{lem}

\begin{proof}
 Part~(b): first of all, the full subcategory $\Lambda(\sB)$ is
closed under kernels in $\sA$ (since it is reflective), while
the full subcategory $\sE$ is closed under cokernels in~$\sA$.
 Furthermore, the functor $\Phi|_{\Lambda(\sB)}\:\Lambda(\sB)
\rarrow\sF$ is the coreflector onto $\sF$ in $\Lambda(\sB)$, while
the functor $\Lambda\Theta|_\sE\:\sE\rarrow\sF$ is the reflector
onto $\sF$ in~$\sE$.
 
 Let $f\:F'\rarrow F''$ be a morphism in~$\sF$.
 Then the kernel of~$f$ in $\sF$ can be computed by applying
the coreflector~$\Phi$ to the kernel of~$f$ in $\sA$ (which
belongs to~$\Lambda(\sB)$), that is $\ker_\sF(f)=\Phi(\ker_\sA(f))$;
while the cokernel of~$f$ in $\sF$ can be computed by applying
the reflector $\Lambda\Theta$ to the cokernel of~$f$ in $\sA$
(which belongs to~$\sE$), i.~e., $\coker_\sF(f)=
\Lambda\Theta(\coker_\sA(f))$.
 Put $B=\ker_\sA(f)\in\sA$; then the composition $\Phi(B)\rarrow F'$
of the adjunction morphism $\Phi(B)\rarrow B$ with the natural
morphism $B\rarrow F'$ is a kernel of the morphism~$f$ in~$\sF$.
 The full subcategory $\sE$ is closed under under quotients in $\sA$
by assumption; so the dual version of
Lemma~\ref{quasi-abelian-reflective-subcategory}(a) tells that
the morphism $\Phi(B)\rarrow B$ is a monomorphism in $\sA$,
and the morphism $B\rarrow F'$ is a monomorphism in $\sA$ by
definition.
 It follows that the composition $\Phi(B)\rarrow B\rarrow F'$ is
a monomorphism in $\sA$, too; thus any kernel in $\sF$ is
a monomorphism in~$\sA$.

 Conversely, let $i\:K\rarrow F$ be a morphism in $\sF$ such that
$i$~is a monomorphism in~$\sA$.
 Put $E=\coker_\sA(i)\in\sA$; then the composition $t\:F\rarrow
\Lambda\Theta(E)$ of the natural morphism $F\rarrow E$ with
the adjunction morphism $E\rarrow\Lambda\Theta(E)$ is a cokernel
of the morphism~$i$ in~$\sF$.
 The functor $\Theta$ is exact by assumption, while the functor
$\Lambda$~is left exact as a right adjoint.
 Applying a left exact functor $\Lambda\Theta$ to the short exact
sequence $0\rarrow K\rarrow F\rarrow E\rarrow0$ in $\sA$, we
obtain a left exact sequence $0\rarrow K\rarrow F\rarrow
\Lambda\Theta(E)$ in~$\sA$.
 Hence the adjunction morphism $E\rarrow\Lambda\Theta(E)$ is
a monomorphism in~$\sA$.
 Since $i$~is a kernel of the epimorphism $F\rarrow E$ in $\sA$,
it is also a kernel of the composition $F\rarrow E\rarrow
\Lambda\Theta(E)$.
 Now $t\:F\rarrow\Lambda\Theta(E)$ is a morphism in~$\sF$;
as $i$~is a kernel of~$t$ in $\sA$, so $i$~is
also a kernel of~$t$ in~$\sF$.

 Part~(a): let $i\:K\rarrow F$ and $g\:K\rarrow L$ be two morphisms
in $\sF$ such that $i$~is a kernel in~$\sF$.
 By part~(b), this means that $i$~is a monomorphism in~$\sA$.
 Denote by $E$ the pushout of the morphisms~$i$ and~$g$ computed
in the category~$\sA$; so $E=\coker_\sA((-g,i)\:K\rarrow L\oplus F$.
 Then the object $\Lambda\Theta(E)$ is the pushout of~$i$ and~$g$
in~$\sF$.
 The natural morphism $L\rarrow\Lambda\Theta(E)$ is the composition
$L\rarrow E\rarrow\Lambda\Theta(E)$.
 The morphism $L\rarrow E$ is a monomorphism in $\sA$, since $\sA$
is an abelian category.
 Following the argument in the previous paragraph, the adjunction
morphism $E\rarrow\Lambda\Theta(E)$ is a monomorphism in~$\sA$.
 The composition of two monomorphisms $L\rarrow\Lambda\Theta(E)$
is again a monomorphism in~$\sA$; being a morphism in $\sF$, it is
a kernel in $\sF$ by part~(b).

 Part~(c): the full subcategory $\Lambda(\sB)\subset\sA$ is closed
under extensions by
Lemma~\ref{abelian-reflective-closed-under-extensions}, while
the full subcategory $\sE\subset\sA$ is closed under extensions
by assumption.
 Hence the intersection $\sF=\Lambda(\sB)\cap\sE$ is closed under
extensions in the abelian category $\sA$, and
Example~\ref{inherited-exact-structure} is applicable.
\end{proof}

 Let $\sC$ be a complete additive category.
 Then the composition of fully faithful additive functors
$\Pro_\lie(\sC)\rarrow\Pro(\sC)\rarrow\Pro^\su(\sC)$ allows
to consider the category $\Pro_\lie(\sC)$ of limit-epimorphic
pro-objects in $\sC$ as a full subcategory in the category
$\Pro^\su(\sC)$ of supplemented pro-objects.
 In fact, a pro-object $P\in\Pro(\sC)$ is limit-epimorphic
if and only if the corresponding supplemented pro-object
$(\varprojlim P,\,P)\in\Pro(\sC)$ is sup-epimorphic; so one has
$\Pro_\lie(\sC)=\Pro(\sC)\cap\Pro^\su_\se(\sC)\subset\Pro^\su(\sC)$.

\begin{prop} \label{limit-epimorphic-inside-supplemented-prop}
 For any complete abelian category\/ $\sC$, the additive category\/
$\Pro_\lie(\sC)$ of limit-epimorphic pro-objects in\/ $\sC$ is
right quasi-abelian.
 The full subcategory\/ $\Pro_\lie(\sC)$ is closed under extensions
in the abelian category\/ $\Pro^\su(\sC)$, and the exact category
structure on\/ $\Pro_\lie(\sC)$ inherited from the abelian exact
structure of\/ $\Pro^\su(\sC)$ coincides with the maximal exact
structure on\/ $\Pro_\lie(\sC)$.
\end{prop}

\begin{proof}
 Consider the abelian categories $\sA=\Pro^\su(\sC)$ and
$\sB=\Pro(\sC)$, the forgetful exact functor $\Theta\:\Pro^\su(\sC)
\rarrow\Pro(\sC)$, and the fully faithful functor $\Lambda\:\Pro(\sC)
\rarrow\Pro^\su(\sC)$ right adjoint to~$\Theta$; so
$\Lambda(P)=(\varprojlim P,\,P)$ for every $P\in\Pro(\sC)$.
 Consider further the coreflective full subcategory
$\Pro^\su_\se(\sC)\subset\Pro^\su(\sC)$ with the coreflector
$\Phi\:\Pro^\su(\sC)\rarrow\Pro^\su_\se(\sC)$ described in the proof
of Corollary~\ref{sup-epimorphic-quasi-abelian-cor}(a);
so $\Phi$ takes an arbitrary supplemented pro-object
$(C,P)\in\Pro^\su(\sC)$ to the supplemented pro-object $(C,Q)$,
where $Q$ is the image of the morphism $\pi\:C\rarrow P$
in $\Pro(\sC)$.

 Let $P=\plim_{\gamma\in\Gamma}P_\gamma$ be a pro-object in $\sC$;
denote by $Q_\delta\subset P_\delta$ the image of the projection
morphism $\varprojlim_{\gamma\in\Gamma} P_\gamma\rarrow P_\delta$
in the category~$\sC$.
 Clearly, one has $\varprojlim_{\gamma\in\Gamma}Q_\gamma\simeq
\varprojlim_{\gamma\in\Gamma}P_\gamma$ in~$\sC$, and it follows
that $\Phi(\Lambda(\sB))\subset\Lambda(\sB)$.

 Let $(C,P)$ be a sup-epimorphic supplemented pro-object in $\sC$,
so the morphism $\pi\:C\rarrow P$ is an epimorphism in $\Pro(\sC)$.
 The morphism $C\rarrow P$ factorizes naturally as $C\rarrow
\varprojlim P\rarrow P$, and it follows that $\varprojlim P
\rarrow P$ is also an epimorphism in $\Pro(\sC)$.
 Hence we have $\Lambda\Theta(\sE)\subset\sE$.

 Therefore, Lemma~\ref{reflective-coreflective-lemma} is applicable
to our set of data, and we can conclude that the additive category
$\Pro_\lie(\sC)=\sF=\Lambda(\sB)\cap\sE$ is right quasi-abelian and
closed under extensions as a full subcategory in $\sA=\Pro^\su(\sC)$.
 It remains to explain that the inherited exact category structure
on $\Pro_\lie(\sC)$ from the abelian exact structure on $\Pro^\su(\sC)$
is the maximal one.
 Indeed, let $0\rarrow K\overset i\rarrow F\overset p\rarrow G\rarrow0$
be a short exact sequence in the maximal exact category structure on
$\Pro_\lie(\sC)$.
 Then $i$~is a kernel in $\sF$, hence by
Lemma~\ref{reflective-coreflective-lemma}(b)
\,$i$~is a monomorphism in~$\sA$.
 Furthermore, $p$~is the cokernel of~$i$ in $\sF$; following
the discussion in the proof of
Lemma~\ref{reflective-coreflective-lemma}(b), this means that
$p$~is the composition $F\rarrow E\rarrow G$, where $F\rarrow E$
is the cokernel of~$i$ in $\sA$ and $E\rarrow\Lambda\Theta(E)=G$ is
the adjunction morphism.
 Moreover, following the same argument, the adjunction morphism
$E\rarrow G$ is a monomorphism in~$\sA$.

 Given an object $A\in\sA=\Pro^\su(\sC)$, let us use the notation
$A=(C_A,P_A)$ the related pair of objects $C_A\in\sC$ and
$P_A\in\Pro(\sC)$, with the structure morphism $\pi_A\:C_A\rarrow P_A$.
 Then the morphisms $C_K=\varprojlim P_K\rarrow P_K$ and
$C_F=\varprojlim P_F\rarrow P_F$ are epimorphisms in $\Pro(\sC)$,
the short sequence $0\rarrow P_K\rarrow P_F\rarrow P_E=P_G\rarrow0$ is
exact in $\Pro(\sC)$, the short sequence $0\rarrow C_K\rarrow C_F
\rarrow C_E\rarrow0$ is exact in $\sC$, and the morphisms
$C_E\rarrow P_E$ and $C_G=\varprojlim P_E\rarrow P_E=P_G$ are also
epimorphisms in $\Pro(\sC)$.
 The morphism $C_E\rarrow C_G$ is a monomorphism in $\sC$, and we
have to show that it is an isomorphism.

 Denote by $D$ the supplemented pro-object $D=(C_D,P_D)$ with
$C_D=C_G\in\sC$ and $P_D=C_G\in\sC\subset\Pro(\sC)$; so the morphism
$\pi_D\:C_D\rarrow P_D$ is an isomorphism in $\Pro(\sC)$.
 (This is our analogue of the discrete abelian group $\boZ$ from
the proof of Proposition~\ref{vslt-semi-stable-cokernels-prop}
or the discrete vector space~$kx$ from
Examples~\ref{not-left-semi-abelian-examples}.)
 Clearly, we have $D\in\sF$.
 Denote by $g\:D\rarrow G$ the natural morphism in $\sF$ with
the components $\mathrm{id}\:C_D\rarrow C_G$ and $\pi_G\:
P_D\rarrow P_G$.
 Then the pullback of the pair of morphisms $p$ and~$g$ in
the category $\sF$ can be computed by applying the coreflector $\Phi$
to the pullback $H=(C_F,\>C_G\sqcap_{P_G}P_F)$ of the same
of morphisms in the category~$\sA$.

 Clearly, the image of the morphism~$\pi_H$ is contained the subobject
$C_E\sqcap_{P_E}P_F\subset C_G\sqcap_{P_E}P_F=P_H$.
 In fact, the morphism $C_H=C_F\rarrow C_E\sqcap_{P_E}P_F$ is
an epimorphism in $\Pro(\sC)$, essentially because the morphism
$F\rarrow E$ is a cokernel of the morphism $K\rarrow F$ in
$\Pro^\su(\sC)$ and the morphism $\pi_K\:C_K\rarrow P_K$ is
an epimorphism in $\Pro(\sC)$.
 Thus we have $\Phi(H)=(C_F,\>C_E\sqcap_{P_E}P_F)\in\sF$.

 By axiom Ex2(b) or Ex2$'$(b), the class of all short exact sequences
in the maximal exact category structure on $\sF$ must be closed
under pullbacks; in other words, the morphism $p\:F\rarrow G$ has to
be a semi-stable cokernel in $\sF$, which means that the natural
morphism $q\:\Phi(H)\rarrow D$ must be a cokernel in~$\sF$.
 One has $\ker_\sF(q)=\ker_\sA(q)=K$ and $\coker_\sA(K\to\Phi(H))=
(C_E,C_E)$, hence $\coker_\sF(K\to\Phi(H))=\Lambda\Theta(C_E,C_E)=
(C_E,C_E)$.
 Hence $q$~is a cokernel in $\sF$ if and only if $C_E\rarrow C_G$
is a isomorphism.
 In the latter case, the short sequence $0\rarrow K\rarrow F
\rarrow G\rarrow0$ is exact in~$\sA$, and we are done.
\end{proof}

\begin{conc} \label{proobjects-conclusion}
 Viewed as a full subcategory in the abelian category of pro-vector
spaces $\Pro(\Vect_k)$, the category of complete, separated topological
vector spaces with linear topology $\Top^\scc_k$ is not well-behaved.
 It is \emph{not} closed under extensions and does \emph{not} inherit
an exact category structure
(see Proposition~\ref{does-not-inherit-prop}(b)).

 Viewed as a full subcategory in the abelian category of supplemented
pro-vector spaces $\Pro^\su(\Vect_k)$, the additive category
$\Top^\scc_k$ is better behaved.
 It is closed under extensions and inherits an exact category structure;
the inherited exact category structure is the maximal exact category
structure on $\Top^\scc_k$.
 (Compare Theorem~\ref{complete-top-vector-spaces-as-proobjects}(b)
with Proposition~\ref{limit-epimorphic-inside-supplemented-prop}.)

 The theory developed in
Sections~\ref{pro-vector-spaces-secn}\+-\ref{suppl-pro-secn}
extends the notion of a topological abelian group/vector space with
linear topology from the case of abelian groups or vector spaces
to arbitrary abelian categories $\sC$ with infinite products.
\end{conc}

\Section{The Strong Exact Category Structure on VSLTs}
\label{strong-exact-structure-secn}

 The following construction plays a key role in the theory of
contramodules over topological rings~\cite[Remark~A.3]{Psemi},
\cite[Section~1.2]{Pweak}, \cite[Section~2.1]{Prev},
\cite[Sections~1.2 and~5]{PR}, \cite[Section~6.2]{PS},
\cite[Section~2.7]{Pcoun}, \cite[Sections~2.5\+-2.7]{Pproperf},
\cite[Section~1]{PS3}.

 Let $\fA$ be a complete, separated topological abelian group with
linear topology, and let $X$ be a set.
 A family of elements $(a_x\in\fA)_{x\in X}$ is said to \emph{converge
to zero} in the topology of $\fA$ if for every open subgroup $\fU\subset
\fA$ the set $\{x\in X\mid a_x\notin\fU\}$ is finite.
 In the terminology of~\cite[Section~1]{Niel}, such a family of
elements $(a_x)_{x\in X}$ in a complete, separated topological
abelian group $\fA$ would be called ``$S$\+Cauchy'' or ``summable''.

 We will consider infinite formal linear combinations $\sum_{x\in X}
a_xx$ of elements of $X$ with families of coefficients $(a_x)_{x\in X}$
converging to zero in the topology of~$\fA$.
 Such infinite formal linear combinations form an abelian group, which
we will denote by $\fA[[X]]$.
 Equivalently, one can define $\fA[[X]]$ as the projective limit
$$
 \fA[[X]]=\varprojlim\nolimits_{\fU\subset\fA}(\fA/\fU)[X],
$$
where $\fU$ ranges over the poset of all open subgroups in $\fA$ and,
for any abelian group $A$, the notation $A[X]=A^{(X)}$ stands for
the direct sum of $X$ copies of $A$, interpreted as the set of all
finite formal linear combinations of elements of $X$ with
the coefficients in~$A$.
 Furthermore, one can endow $\fA[[X]]$ with the topology of projective
limit of discrete abelian groups $(\fA/\fU)[X]$, in which the kernels
of the projection maps $\fA[[X]]\rarrow(\fA/\fU)[X]$ form a base of
neighborhoods of zero.
 Obviously, this makes $\fA[[X]]$ a complete, separated topological
abelian group.

 The rule assigning the topological abelian group $\fA[[X]]$ to
a topological abelian group $\fA\in\Top^\scc_\boZ$ and a set $X$ is
a covariant functor of two arguments
\begin{equation} \label{abelian-convergent-linear-combinations-functor}
 {-}[[{-}]]\:\Top^\scc_\boZ\times\Sets\lrarrow\Top^\scc_\boZ.
\end{equation}
 Specifically, if $g\:\fA\rarrow\fB$ is a continuous homomorphism
of complete, separated topological abelian groups and $X$ is a set,
then the induced continuous homomorphism
$g[[X]]\:\fA[[X]]\rarrow\fB[[X]]$ takes a formal linear combination
$\sum_{x\in X}a_xx\in\fA[[X]]$ to the formal linear combination
$\sum_{x\in X}g(a_x)x\in\fB[[X]]$.
 If $\fA$ is a complete, separated topological abelian group and
$f\:X\rarrow Y$ is a map of sets, then the induced continuous
homomorphism $\fA[[f]]\:\fA[[X]]\rarrow\fA[[Y]]$ takes 
a formal linear combination $\sum_{x\in X}a_xx\in\fA[[X]]$ to
the formal linear combination $\sum_{y\in Y}c_yy\in\fA[[Y]]$,
where $c_y=\sum_{x:f(x)=y}a_x$ is the sum of a zero-convergent
family of elements in $\fA$, defined as the limit of finite partial
sums in the topology of~$\fA$.

 For any complete, separated topological vector space $\fV$ with
linear topology and any set $X$, the topological abelian group
$\fV[[X]]$ is naturally a (complete, separated) topological vector
space with linear topology, too.
 In other words, one has a covariant functor of two arguments
\begin{equation} \label{vector-convergent-linear-combinations-functor}
 {-}[[{-}]]\:\Top^\scc_k\times\Sets\lrarrow\Top^\scc_k
\end{equation}
which agrees with
the functor~\eqref{abelian-convergent-linear-combinations-functor}
and the forgetful functor $\Top^\scc_k\rarrow\Top^\scc_\boZ$.

 We will say that a continuous homomorphism of complete, separated
topological abelian groups $p\:\fA\rarrow\fC$ is \emph{strongly
surjective} if the map $p[[X]]\:\fA[[X]]\rarrow\fC[[X]]$ is surjective
for every set~$X$.
 A continuous homomorphism of complete, separated topological vector
spaces is \emph{strongly surjective} if it is strongly surjective
as a map of topological abelian groups.

 Furthermore, let $\fA$ be a complete, separated topological abelian
group (or vector space), and let $i\:\fK\rarrow\fA$ be a closed
injective morphism of topological abelian groups/vector spaces.
 We will say that the map $i$~is \emph{strongly closed} if
the quotient group/vector space $\fA/i(\fK)$ is complete in
the quotient topology and the quotient map $\fA\rarrow\fA/i(\fK)$
is strongly surjective.
 In this case, the closed subgroup/subspace $i(\fK)\subset\fA$ will
be also called \emph{strongly closed}.
 The notion of a strongly closed subgroup in a topological group is
important for the paper~\cite{Pproperf}, where it is discussed
in~\cite[Sections~2.11\+-2.12]{Pproperf}, and even more important
in the paper~\cite{PS3}.

 It is clear from Corollary~\ref{vslt-stable-kernels-cor} that any
strongly closed injective map is stably closed (in the sense of
the definition at the end of Section~\ref{maximal-exact-VSLTs-secn}).
 The counterexample following below shows that the inverse implication
fails.
 In other words, a surjective continuous open linear map (of complete,
separated topological vector spaces) need \emph{not} be strongly
surjective.
 
\begin{rem}
 It may be worthwhile to emphasize what the previous definitions mean.
 Let $p\:\fA\rarrow\fC$ be a surjective continuous homomorphism of
complete, separated topological abelian groups.
 Let us even assume that $p$~is open.
 What does it mean that $p$~is strongly surjective?

 By the definition, the relevant question is the following one.
 Let $(c_x\in\fC_x)_{x\in X}$ be a family of elements in $\fC$,
indexed by a set $X$ and converging to zero in the topology of~$\fC$.
 Can one lift the family $(c_x)_{x\in X}$ to a family of elements
$(a_x\in\fA)_{x\in X}$ such that $p(a_x)=c_x$ for every $x\in X$
\emph{and} the family of elements $(a_x)_{x\in X}$ converges to
zero in the topology of~$\fA$\,?

 One can lift every single element $c_x\in\fC$ to an element $a_x\in\fA$,
since the map~$p$ is surjective by assumption.
 But can one lift a whole zero-convergent family of elements in such
a way that it remains zero-convergent?
 Here is an example showing that this \emph{cannot} be done
(generally speaking). 
\end{rem}

\begin{ex} \label{strong-is-not-maximal}
 Let $Q$ be a separated topological abelian group and $I$ be
an infinite set.
 Consider the complete, separated topological abelian group
$\fA_I(Q)=Q^{(I)}$ in the modified coproduct topology, as in
Theorem~\ref{top-abelian-main-theorem}.
 We claim that \emph{no infinite family of nonzero elements in\/
$\fA_I(Q)$ converges to zero in the modified coproduct topology}.
 In other words, this means that the natural embedding of (abstract,
nontopological) abelian groups $\fA[X]\rarrow\fA[[X]]$, defined in
the obvious way for any complete, separated topological abelian
group $\fA$, is a bijection for $\fA=\fA_I(Q)$.
 A version of this assertion is included into the formulation
of~\cite[Proposition~11.1]{RD}.

 Indeed, let $(a_x\in\fA)_{x\in X}$ be a family of nonzero elements
in~$\fA$.
 For every $x\in X$, the element $a_x\in\fA$ can be viewed as
a family of elements $a_x=(q_{x,i}\in Q)_{i\in I}$ with $q_{x,i}=0$
for all but a finite subset of indices~$i$.
 Given a subset $Y\subset X$, denote by $J_Y\subset I$ the subset
of all indices $j\in I$ for which there exists an index $y\in Y$
such that $q_{y,j}\ne0$ in~$Q$.
 Consider two possibilities separately: either there exists a subset
$Y\subset X$ such that the complement $X\setminus Y$ is finite and
the set $J_Y$ is finite, or for any subset $Y\subset X$ with a finite
complement $X\setminus Y$ the set $J_Y$ is infinite.

 In the first case, put $J=J_Y$ and consider the open subgroup
$U=\bigoplus_{j\in J}\{0\}\oplus\bigoplus_{l\in I\setminus J}Q
\subset\bigoplus_{i\in I}Q=\fA$.
 For every $y\in Y$, we have $q_{y,l}=0$ for all $l\in I\setminus J$
(by the definition of $J=J_Y$).
 Since $a_y\ne0$, there should exist an index $j\in J$ such that
$q_{y,j}\ne0$.
 Hence $a_y\notin U$.
 Since the family of elements $(a_x)_{x\in X}$ converges to zero in
$\fA$ by assumption, it follows that the set $Y$ must be finite.
 As the complement $X\setminus Y$ is finite, too, we can conclude
that the set of indices $X$ is finite, as desired.

 In the second case, choose a pair of indices $j_1\in I$ and
$x_1\in X$ such that $q_{x_1,j_1}\ne0$.
 Since the set $J_{X\setminus\{x_1\}}$ is infinite, we can choose
a pair of indices $j_2\in I$ and $x_2\in X$ such that $j_2\ne j_1$,
\ $x_2\ne x_1$, and $q_{x_2,j_2}\ne0$.
 Since the set $J_{X\setminus\{x_1,x_2\}}$ is infinite, we can
choose a pair of indices $j_3\in I$ and $x_3\in X$ such that
$j_1\ne j_3\ne j_2$,\ $x_1\ne x_3\ne x_2$, and $q_{x_3,j_3}\ne0$.
 Proceeding in this way, we choose a sequence of pairwise distinct
indices $j_s\in I$ and a sequence of pairwise disctinct indices
$x_s\in X$, \ $s=1$, $2$, $3$,~\dots, such that $q_{x_s,j_s}\ne0$
in $Q$ for all $s\ge1$.

 Denote by $J$ the subset $\{j_1,j_2,j_3,\dotsc\}\subset I$.
 For every $j=j_s\in J$, choose an open subgroup $U_j\subset Q$
such that $q_{x_s,j_s}\notin U_j$.
 For every $l\in I\setminus J$, put $U_l=Q$.
 Consider the open subgroup $U=\bigoplus_{i\in I}U_i\subset
\bigoplus_{i\in I}Q=\fA$.
 For every $s\ge1$, we have $a_{x_s}\notin U$.
 As the subset of indices $\{x_1,x_2,x_2,\dotsc\}\subset X$ is
infinite, this contradicts the assumption that the family of
elements $(a_x)_{x\in X}$ converges to zero in~$\fA$.
 Thus the second case is impossible, and we have proved the claim.

 Now let $\fC$ be a complete, separated topological vector space in
which an infinite family of nonzero elements $(c_x\in\fC)_{x\in X}$
converging to zero in the topology of $\fC$ does exist; so
$\fC[X]\varsubsetneq\fC[[X]]$.
 For example, one can take $\fC=k^\omega=\prod_{n=0}^\infty k$
(with the product topology of discrete one-dimensional vector
spaces~$k$), as in Corollary~\ref{not-left-quasi-abelian-cor}.
 Then the topological basis $(e_n)_{n=0}^\infty$ of the linearly
compact topological vector space $\fC$ is an infinite family of
nonzero vectors converging to zero in~$\fC$.

 Let $I$ be any infinite set; it suffices to take $I=\omega$.
 According to the proof of Theorem~\ref{top-abelian-main-theorem},
the summation map $\Sigma\:\fV=\fA_I(\fC)\rarrow\fC$ is open and
continuous; it is also obviously a surjective homomorphism of
vector spaces.
 However, the image of the map $\Sigma[[X]]\:\fV[[X]]\rarrow\fC[[X]]$
is contained in $\fC[X]\varsubsetneq\fC[[X]]$, because
$\fV[[X]]=\fV[X]$ according to the argument above.
 In other words, the infinite zero-convergent family of nonzero vectors
$(c_x\in\fC)_{x\in X}$ cannot be lifted to a zero-convergent
family of vectors in $\fV$, as there are no infinite zero-convergent
families of nonzero vectors in $\fV=\fA_I(\fC)$.

 Thus the map~$\Sigma$ is \emph{not} strongly surjective.
 The kernel $\fK\subset\fV$ of the continuous homomorphism $\Sigma$
is a stably closed, but \emph{not} strongly closed vector subspace
in the complete, separated topological vector space~$\fV$.
\end{ex}

\begin{prop} \label{doublebracket-preserves-kers-cokers}
\textup{(a)} For any set $X$, the functor\/ ${-}[[X]]\:\Top^\scc_\boZ
\rarrow\Top^\scc_\boZ$ preserves kernels and cokernels. \par
\textup{(b)} For any set $X$, the functor\/ ${-}[[X]]\:\Top^\scc_k
\rarrow\Top^\scc_k$ preserves kernels and cokernels.
\end{prop}

\begin{proof}
 Let us discuss part~(a).
 Recall that the forgetful functor $\Top^\scc_\boZ\rarrow\Ab$ preserves
kernels, but not cokernels.
 Accordingly, in the context of the proposition, the claim is that, for
any set $X$, the functor ${-}[[X]]\:\Top^\scc_\boZ\rarrow\Ab$ preserves
the kernels (as one can easily see); but it does \emph{not} preserve
the cokernels.

 The assertion of part~(a) for the kernel and cokernel of a morphism
$f\:\fA\rarrow\fB$ in $\Top^\scc_\boZ$ reduces to the following
properties:
\begin{itemize}
\item for any topological group $\fA\in\Top^\scc_\boZ$ and a closed
subgroup $\fK\subset\fA$ with the induced topology, the topology of
$\fK[[X]]\in\Top^\scc_\boZ$ coincides with the induced topology on
$\fK[[X]]$ as a subgroup in $\fA[[X]]$;
\item for any morphism $g\:\fA\rarrow\fL$ in $\Top^\scc_\boZ$ with
$g(\fA)$ dense in $\fL$, the image of $\fA[[X]]$ is dense in $\fL[[X]]$;
in fact, the subgroup $g(\fA)[X]\subset\fL[[X]]$ is already dense;
\item for any for any topological group $\fB\in\Top^\scc_\boZ$,
a closed subgroup $\fL\subset\fB$, and the completion $\fC$ of
the quotient group $\fB/\fL$ in its quotient topology, one has
a natural isomorphism between the topological group $\fC[[X]]$ and
the completion of the quotient group $\fB[[X]]/\fL[[X]]$ in its
quotient topology.
\end{itemize}
 The latter property holds because both the topological groups in
question can be identified with the projective limit of discrete
groups $\fB[X]/(\fU[X]+\fL[X])\simeq(\fB/(\fU+\nobreak\fL))[X]$, where
$\fU$ ranges over the open subgroups of~$\fB$.
 We leave further details to the reader.
\end{proof}

\begin{cor} \label{converging-combinations-not-exact-for-maximal}
\textup{(a)} For any infinite set $X$, the functor\/
${-}[[X]]\:\Top^\scc_\boZ\rarrow\Top^\scc_\boZ$ is \emph{not} exact
with respect to the maximal exact structure on\/ $\Top^\scc_\boZ$. \par
\textup{(b)} For any infinite set $X$, the functor\/
${-}[[X]]\:\Top^\scc_k\rarrow\Top^\scc_k$ is \emph{not} exact
with respect to the maximal exact structure on\/ $\Top^\scc_k$. 
\end{cor}

\begin{proof}
 This is a direct corollary of Example~\ref{strong-is-not-maximal}.
\end{proof}

\begin{thm} \label{strong-exact-structure-theorem}
\textup{(a)} There is an exact category structure on
the additive category\/ $\Top^\scc_\boZ$, called the \emph{strong
exact structure}, in which a short sequence\/
$0\rarrow\fK\rarrow\fA\rarrow\fC\rarrow0$ is exact if and only if
it satisfies Ex1 and the induced map\/ $\fA[[X]]\rarrow\fC[[X]]$ is
surjective for every set~$X$.
 The admissible monomorphisms in this exact structure are the strongly
closed injective maps, and the admissible epimorphisms are
the open strongly surjective maps.
 The strong exact structure on\/ $\Top^\scc_\boZ$ is different from
(i.~e., has fewer short exact sequences than) the maximal exact
structure. \par
\textup{(b)} There is an exact category structure on
the additive category\/ $\Top^\scc_k$, called the \emph{strong
exact structure}, in which a short sequence\/
$0\rarrow\fK\rarrow\fV\rarrow\fC\rarrow0$ is exact if and only if
it satisfies Ex1 and the induced map\/ $\fV[[X]]\rarrow\fC[[X]]$ is
surjective for every set~$X$.
 The admissible monomorphisms in this exact structure are the strongly
closed injective maps, and the admissible epimorphisms are
the open strongly surjective maps.
 The strong exact structure on\/ $\Top^\scc_k$ is different from
(i.~e., has fewer short exact sequences than) the maximal exact
structure.
\end{thm}

\begin{proof}
 Let us explain part~(a); part~(b) is similar.
 The argument is based on Example~\ref{psi-exact-structure}.
 Consider the exact category $\Top^\scc_\boZ$ with its maximal
exact category structure.
 Let us view the category of abelian groups $\Ab$ as an exact
category with the abelian exact structure.
 Fix a set~$X$, and consider the functor $\Psi_X\:\Top^\scc_\boZ
\rarrow\Ab$, \ $\Psi_X(\fA)=\fA[[X]]$.
 Following Proposition~\ref{doublebracket-preserves-kers-cokers},
the functor $\Psi_X$ preserves all kernels; in particular, it preserves
the kernels of admissible epimorphisms.

 Alternatively, consider the functor $\Psi'_X\:\Top^\scc_\boZ
\rarrow\Top^\scc_\boZ$, \ $\Psi'_X(\fA)=\fA[[X]]$, viewing both
the source and the target category $\Top^\scc_\boZ$ as an exact
category with the maximal exact structure.
 Following Proposition~\ref{doublebracket-preserves-kers-cokers},
the functor $\Psi'_X$ preserves all kernels and cokernels;
in particular, it preserves the kernels of admissible epimorphisms
and the cokernels of admissible monomorphisms.

 Thus, for any one of the functors $\Psi_X$ or $\Psi'_X$,
the construction of Example~\ref{psi-exact-structure} is applicable,
and it produces a new exact category structure on $\Top^\scc_\boZ$.
 It is clear from the discussion of the maximal exact structure on
$\Top^\scc_\boZ$ in Section~\ref{maximal-exact-VSLTs-secn}
(see Proposition~\ref{vslt-semi-stable-cokernels-prop}) that with both
approaches the same new exact category structure is produced.
 It can be called the \emph{$\Psi_X$\+exact structure} or
the \emph{$\Psi'_X$\+exact structure} (which is the same)
on $\Top^\scc_\boZ$.
 In the $\Psi_X$\+exact structure with a nonempty set $X$, a short
sequence $0\rarrow\fK\rarrow\fV\rarrow\fC\rarrow0$ is exact if and
only if it satisfies Ex1 and the induced map $\fV[[X]]\rarrow\fC[[X]]$
is surjective.

 The strong exact structure on $\Top^\scc_\boZ$ is the intersection of
the $\Psi_X$\+exact structures taken over all sets~$X$.
 In other words, a short sequence is exact in the strong exact structure
if and only if it exact in the $\Psi_X$\+exact structure for every
set~$X$.
 Clearly, the intersection of any nonempty collection of exact
structures on an additive category is an exact structure.
 It is obvious from the discussion of strongly surjective maps and
strongly closed injective maps above in this section that the strong
exact category structure on $\Top^\scc_\boZ$, defined in this way,
has the properties listed in the proposition.

 Finally, the strong exact structure on $\Top^\scc_\boZ$ differs from
the maximal exact structure according to
Corollary~\ref{converging-combinations-not-exact-for-maximal}.
 In fact, following the arguments in
Example~\ref{strong-is-not-maximal}, for a countable set $X$ already
the $\Psi_X$\+exact structure has fewer short exact sequences than
the maximal exact structure.
\end{proof}

 See~\cite[Lemma~2.4]{Pproperf} for a formulation and direct proof of
some specific properties related to the existence of the strong
exact structure.

 The short sequences that are exact in the strong exact category
structure on $\Top^\scc_\boZ$ or $\Top^\scc_k$ are called
the \emph{strong short exact sequences}.

 The full subcategories of (complete, separated) topological abelian
groups/vector spaces with a countable base of neighborhoods of zero
$\Top^{\omega,\scc}_\boZ$ and $\Top^{\omega,\scc}_k$ inherit the strong
exact category structures of the ambient additive categories
$\Top^\scc_\boZ$ and $\Top^\scc_k$.
 Moreover, the inherited exact structures on the quasi-abelian
categories $\Top^{\omega,\scc}_\boZ$ and $\Top^{\omega,\scc}_k$
coincide with their quasi-abelian exact structures (so for topological
abelian groups with a countable base of neighborhoods of zero there is
no difference between the maximal and strong exact structures).
 In fact, the following stronger version of
Proposition~\ref{countable-base-kernel} holds.
 
\begin{prop}
 Let\/ $\fA$ be a complete, separated topological abelian group, and
let\/ $\fK\subset\fA$ be a closed subgroup.
 Assume that the topological abelian group\/ $\fK$ has a countable
base of neighborhoods of zero.
 Then, for any set $X$, the map\/ $\fA[[X]]\rarrow (\fA/\fK)[[X]]$
induced by the open continuous homomorphism\/
$\fA\rarrow\fA/\fK$ is surjective.
 In other words, $\fK$ is a strongly closed subgroup in\/~$\fA$.
\end{prop}

\begin{proof}
 Denote by $\fC$ the quotient group $\fC=\fA/\fK$ with the quotient
topology; by Proposition~\ref{countable-base-kernel}, the topological
abelian group $\fC$ is complete.
 Consider the projective system of short exact sequences of abelian
groups
$$
 0\lrarrow\fK/(\fU\cap\fK)\lrarrow\fA/\fU\lrarrow
 \fA/(\fU+\fK)\lrarrow0
$$
indexed by the directed poset of open subgroups $\fU\subset\fA$.
 Passing to the direct sum over the set~$X$, we obtain a projective
system of short exact sequences of abelian groups
$$
 0\lrarrow(\fK/(\fU\cap\fK))[X]\lrarrow(\fA/\fU)[X]\lrarrow
 (\fA/(\fU+\fK))[X]\lrarrow0.
$$
 The map $\fA[[X]]\rarrow\fC[[X]]$ which we are interested in
is obtained by taking the projective limit of the projective system
of surjective homomorphisms $(\fA/\fU)[X]\lrarrow
(\fA/(\fU+\fK))[X]$.
 Since the topological abelian group $\fK$ has a countable base of
neighborhoods of zero and the transition maps in the projective system
of abelian groups $(\fK/(\fU\cap\fK))[X]$ are surjective, the argument
from the proof of Proposition~\ref{countable-base-kernel} shows that
$\varprojlim^1_{\fU\subset\fA}((\fK/(\fU\cap\fK))[X])=0$.
 Hence the desired surjectivity of the map of projective limits.
\end{proof}

\Section{Beilinson's Yoga of Tensor Product Operations}
\label{tensor-product-yoga-secn}

 In this section we discuss uncompleted versions of the topological
tensor product operations defined in~\cite[Section~1.1]{Beil}.
 The discussion of completed tensor products is postponed until
the next Section~\ref{refined-exact-struct-secn}.

 Let $U$ and $V$ be topological vector spaces (with linear topology).
 We will define three linear topologies on the tensor product space
$U\ot_kV$.
 The resulting topological vector spaces will be denoted by
$U\ot^*V$, \ $U\ot^\la V$, and $U\ot^!V$.

 A vector subspace $E\subset U\ot_kV$ is said to be open in $U\ot^*V$
if the following conditions hold:
\begin{itemize}
\item there exist open subspaces $P\subset U$ and $Q\subset V$ such
that $P\ot_kQ\subset E$;
\item for every vector $u\in U$, there exists an open subspace
$Q_u\subset V$ such that $u\ot Q_u\subset E$; and
\item for for every vector $v\in V$, there exists an open subspace
$P_v\subset U$ such that $P_v\ot v\subset E$.
\end{itemize}

 A vector subspace $E\subset U\ot_kV$ is said to be open in $U\ot^\la V$
if the following conditions hold:
\begin{itemize}
\item there exists an open subspace $P\subset U$ such that
$P\ot_kV\subset E$; and
\item for every vector $u\in U$, there exists an open subspace
$Q_u\subset V$ such that $u\ot Q_u\subset E$.
\end{itemize}

 A vector subspace $E\subset U\ot_kV$ is said to be open in $U\ot^!V$
if there exist open subspaces $P\subset U$ and $Q\subset V$ such that
$P\ot_kV+U\ot_kQ\subset E$.
 In other words, the vector subspaces $P\ot_kV+U\ot_kQ$, where
$P\subset U$ and $Q\subset V$ are open subspaces, form a base of
neighborhoods of zero in $U\ot^!V$.

 Clearly, the identity maps $\mathrm{id}_{U\ot_kV}$ are continuous
homomorphisms of topological vector spaces $U\ot^*V\rarrow U\ot^\la V
\rarrow U\ot^!V$.
 So one can say that the $*$\+topology is the finest one of the three
topologies, while the $!$\+topology is the coarsest one of the three.
 The $*$-tensor product and the $!$\+tensor product are commutative
(symmetric) operations, while the $\la$\+tensor product is not.
 The opposite operation to the $\la$\+tensor product can be denoted by
$V\ot^\to U=U\ot^\la V$.

\begin{rem} \label{tensor-product-not-complete-remark}
 The reader should be warned that our notation for the three
tensor product operations is actually different from the notation
in~\cite{Beil}, in that the symbols $\ot^*$ and $\ot^!$ are used to
denote the \emph{completed} tensor products in~\cite{Beil}.
 We use them to denote the uncompleted tensor products.

 Thus, in our notation, the topological vector spaces $\fU\ot^*\fV$, \
$\fU\ot^\la\fV$, and $\fU\ot^!\fV$ need \emph{not} be complete even
for complete topological vector spaces $\fU$ and $\fV$.
 For example, when $\fU$ and $\fV$ are linearly compact
(profinite-dimensional) topological vector spaces, all the three
topologies on the tensor product $\fU\ot_k\fV$ coincide with each
other, and the vector space $\fU\ot_k\fV$ is \emph{incomplete} in
this topology (whenever both $\fU$ and $\fV$ are infinite-dimensional).
 The related completion is a linearly compact topological vector
space $\fU\cot\fV$ (the usual tensor product in the category
of linearly compact topological vector spaces).
\end{rem}

\begin{rem} \label{three-incomplete-topologies-motivation-remark}
 The following observations (essentially taken from~\cite[Remark~(ii)
in Section~1.1]{Beil} motivate the definitions of the three tensor
product topologies.
 Let $U$, $V$, and $W$ be three topological vector spaces, and let
$\phi\:U\times V\rarrow W$ be a bilinear map.
 Then the map~$\phi$ is continuous (as a function of two variables) if
and only if the related linear map $\phi^\ot\:U\ot_k V\rarrow W$ is
continuous in the $*$\+topology on $U\ot_kV$, i.~e., the linear map
$\phi^\ot\:U\ot^*V\rarrow W$ is continuous.

 A \emph{topological algebra} $R$ (over~$k$) is a topological vector
space (with linear topology) endowed with an associative $k$\+algebra
structure such that the multiplication map $\cdot\,\:R\times R\rarrow R$
is continuous.
 According to the previous paragraph, the multiplication in
a topological algebra $R$ can be described as a continuous linear
map $\mu\:R\ot^*R\rarrow R$.
 Now a topological algebra $R$ has a base of neighborhoods of zero
consisting of open right ideals if and only if its multiplication
map is continuous in the $\la$\+topology, i.~e., the linear map
$\mu\:R\ot^\la R\rarrow R$ is continuous.
 A topological algebra $R$ has a base of neighborhoods of zero
consisting of open two-sided ideals if and only if its multiplication
map is continuous in the $!$\+topology, i.~e., the linear map
$\mu\:R\ot^!R\rarrow R$ is continuous.
\end{rem}

\begin{lem} \label{tensor-continuous}
 All the three tensor product operations\/ $\ot^*$, \ $\ot^\la$, and\/
$\ot^!$ are functors of two arguments\/ $\Top_k\times\Top_k\rarrow
\Top_k$.
 In other words, for any continuous linear maps of topological vector
spaces $f\:U'\rarrow U''$ and $g\:V'\rarrow V''$, the linear map
$f\ot g\: U'\ot^*V'\rarrow U''\ot^*V''$ is continuous; and similarly for
the\/ $\la$\+topology and the\/ $!$\+topology on the tensor products.
\end{lem}

\begin{proof}
 Let us sketch a proof for the $*$\+topology; the arguments for
the other two topologies are similar.
 Let $E''\subset U''\ot^*V''$ be an open subspace and
$E'=(f\ot\nobreak g)^{-1}(E'')$ be its preimage.
 We have to show that $E'\subset U'\ot^*V'$ is an open subspace.
 Let $P''\subset U''$ and $Q''\subset V''$ be open subspaces such
that $P''\ot_k Q''\subset E''$, and let $P'=f^{-1}(P'')$ and $Q'=
g^{-1}(Q'')$ be their preimages.
 Then $P'\subset U'$ and $P''\subset U''$ are open subspaces,
and $P'\ot_k Q'\subset E'$.
 Let $u'\in U'$ be a vector; put $u''=f(u')$.
 Let $Q_{u''}\subset V''$ be an open subspace such that
$u''\ot Q_{u''}\subset E''$.
 Put $Q_{u'}=g^{-1}(Q_{u''})$.
 Then $Q_{u'}\subset V'$ is an open subspace, and $u'\ot Q_{u'}
\subset E'$.
\end{proof}

\begin{lem} \label{tensor-preserves-open-maps}
 Let $p\:U\rarrow C$ and $q\:V\rarrow D$ be open surjective linear maps
of topological vector spaces.  Then \par
\textup{(a)} $f\ot g\: U\ot^*V\rarrow C\ot^*D$ is an open surjective
linear map; \par
\textup{(b)} $f\ot g\: U\ot^\la V\rarrow C\ot^\la D$ is an open
surjective linear map; \par
\textup{(c)} $f\ot g\: U\ot^!V\rarrow C\ot^!D$ is an open surjective
linear map.
\end{lem}

\begin{proof}
 Let us explain part~(a); parts~(b) and~(c) are similar.
 Let $E\subset U\ot^*V$ be an open subspace; we have to show that
$(f\ot g)(E)\subset C\ot^*D$ is an open subspace.
 Let $P\subset U$ and $Q\subset V$ be open subspaces such that
$P\ot_kQ\subset E$.
 Then $f(P)\subset C$ and $g(Q)\subset D$ are open subspaces, and
$f(P)\ot_kg(Q)\subset(f\ot g)(E)$.
 Let $c\in C$ be a vector; choose a vector $u\in U$ such that $f(u)=c$.
 Let $Q_u\subset V$ be an open subspace such that $u\ot Q_u\subset E$.
 Then $g(Q_u)\subset D$ is an open subspace, and
$c\ot g(Q_u)\subset (f\ot g)(E)$.
\end{proof}

\begin{lem} \label{tensor-star-induced-topology}
 Let $U$ and $V$ be topological vector spaces, and let $K\subset U$ and
$L\subset V$ be vector subspaces endowed with the induced topologies.
 Then the topology of $K\ot^*L$ coincides with the induced topology on
$K\ot_kL$ as a vector subspace in $U\ot^*V$.
\end{lem}

\begin{proof}
 We have $K\ot_kL\subset U\ot_kL\subset U\ot_kV$ and the functor
$\ot^*$ is commutative; so it suffices to show that the topology of
$U\ot^*L$ coincides with the induced topology on $U\ot_kL$ as
a vector subspace in $U\ot^*V$.
 By Lemma~\ref{tensor-continuous}, the inclusion $U\ot^*L\rarrow
U\ot^*V$ is a continuous map.

 Let $F\subset U\ot^*L$ be an open subspace.
 Then there exist open subspaces $P\subset U$ and $S\subset L$ such that
$P\ot_kS\subset F$.
 Choose an open subspace $Q\subset V$ such that $Q\cap L=S$ and
$Q+L=V$.
 Furthermore, choose a basis $\{u_i\}_{i\in I}$ in $U$ such that
a subset of $\{u_i\}$ is a basis in $P\subset U$.
 For every $i\in I$ such that $u_i\notin P$, there exists an open
subspace $S_i\subset L$ such that $u_i\ot S_i\subset F$.
 Choose an open subspace $Q_i\subset V$ such that $Q_i\cap L=S_i$.
 Then
$$
 E = F + P\ot_k Q + \sum\nolimits_{i\in I : u_i\notin P}
 u_i\ot Q_i \,\subset\, U\ot_kV
$$
is an open subspace in $U\ot^*V$ such that $E\cap(U\ot_kL)=F$.

 Indeed, let $v\in V$ be a vector.
 Choose vectors $q\in Q$ and $l\in L$ such that $q+l=v$.
 Then $P\ot q\subset E$, and there exists an open subspace
$P_l\subset U$ such that $P_l\ot l\subset F$.
 Hence $P\cap P_l$ is an open subspace in $U$ for which
$(P\cap P_l)\ot v\subset E$.
\end{proof}

\begin{lem} \label{tensor-arrow-induced-topology}
 Let $U$ and $V$ be topological vector spaces, and let $K\subset U$ and
$L\subset V$ be vector subspaces endowed with the induced topologies.
 Then the topology of $K\ot^\la L$ coincides with the induced topology
on $K\ot_kL$ as a vector subspace in $U\ot^\la V$.
\end{lem}

\begin{proof}
 By Lemma~\ref{tensor-continuous}, the inclusion $K\ot^\la L\rarrow
U\ot^\la V$ is a continuous map.
 Using the inclusions $K\ot_kL\subset U\ot_kL\subset U\ot_kV$
(or $K\ot_kL\subset K\ot_kV\subset U\ot_kV$), it suffices to consider
two cases separately.

 Let us show that the topology of $U\ot^\la L$ coincides with
the induced topology on $U\ot_kL$ as a vector subspace in $U\ot^\la V$.
 Let $F\subset U\ot^\la L$ be an open subspace.
 Then there exists an open subspace $P\subset U$ such that
$P\ot_kL\subset F$.
 Choose a basis $\{u_i\}_{i\in I}$ in $U$ such that a subset of
$\{u_i\}$ is a basis in $P\subset U$.
 For every $i\in I$ such that $u_i\notin P$, there exists an open
subspace $S_i\subset L$ such that $u_i\ot S_i\subset F$.
 Choose an open subspace $Q_i\subset V$ such that $Q_i\cap L=S_i$.
 Then
$$
 E = F + P\ot_k V + \sum\nolimits_{i\in I : u_i\notin P}
 u_i\ot Q_i \,\subset\, U\ot_kV
$$
is an open subspace in $U\ot^\la V$ such that $E\cap(U\ot_kL)=F$.

 Let us show that the topology of $K\ot^\la V$ coincides with
the induced topology on $K\ot_kV$ as a vector subspace in $U\ot^\la V$.
 Let $F\subset K\ot^\la V$ be an open subspace.
 Then there exists an open subspace $R\subset K$ such that
$R\ot_kV\subset F$.
 Choose an open subspace $P\subset U$ such that $P\cap K=R$ and
$P+K=U$.
 Then $E=F+P\ot_kV\subset U\ot_kV$ is an open subspace in $U\ot^\la V$
such that $E\cap(K\ot_kV)=F$.
\end{proof}

\begin{lem} \label{tensor-shriek-induced-topology}
 Let $U$ and $V$ be topological vector spaces, and let $K\subset U$ and
$L\subset V$ be vector subspaces endowed with the induced topologies.
 Then the topology of $K\ot^!L$ coincides with the induced topology on
$K\ot_kL$ as a vector subspace in $U\ot^!V$.
\end{lem}

\begin{proof}
 By Lemma~\ref{tensor-continuous}, the inclusion $K\ot^!L\rarrow
U\ot^!V$ is a continuous map.
 Let $F\subset K\ot^!L$ be an open subspace.
 Then there exist open subspaces $R\subset K$ and $S\subset L$
such that $R\ot_kL+K\ot_kS\subset F$.
 Choose open subspaces $P\subset U$ and $Q\subset V$ such that
$P\cap K=R$ and $Q\cap L=S$.
 Put $E=F+P\ot_kV+U\ot_kQ\subset U\ot_kV$.
 Then $E$ is an open subspace in $U\ot^!V$ such that $E\cap(K\ot_kL)=F$.
\end{proof}

\begin{lem} \label{closed-subspaces-tensor-lemma}
 Let $U$ and $V$ be topological vector spaces, and let $K\subset U$
and $L\subset V$ be closed vector subspaces.
 Then the vector subspace $K\ot_kL\subset U\ot_kV$ is closed in
the topological vector spaces $U\ot^*V$, \ $U\ot^\la V$, and $U\ot^!V$.
\end{lem}

\begin{proof}
 It suffices to consider the case of the $!$\+topology, as it is
the coarsest one of the three.
 Let $w\in U\ot_kV$ be a vector not belonging to $K\ot_kL$.
 Let $U_w\subset U$ and $V_w\subset V$ be the minimal
(finite-dimensional) vector subspaces such that $w\in U_w\ot_k V_w$.
 Then either $U_w\not\subset K$, or $V_w\not\subset L$ (or both).
 Suppose that $U_w\not\subset K$, and choose a vector $u\in U_w$
such that $u\notin K$.
 Since $K$ is a closed vector subspace in $U$, there exists an open
vector subspace $P\subset U$ such that the coset $u+P\subset U$ does
not intersect $K$ (or equivalently, $u\notin P+K$).
 Consider the open subspace $P\ot_k V\subset U\ot^!V$.
 Then we have $P\ot_kV+K\ot_kL\subset (P+K)\ot_kV$ and
$w\notin (P+K)\ot_kV$ (since $U_w\not\subset P+K$).
 Hence $w$~does not belong to the closure of $K\ot_kL$ in $U\ot_kV$.
\end{proof}

\begin{lem} \label{dense-subspaces-tensor-lemma}
 Let $U$ and $V$ be topological vector spaces, and let $K\subset U$
and $L\subset V$ be dense vector subspaces.
 Then the vector subspace $K\ot_kL\subset U\ot_kV$ is dense in
the topological vector spaces $U\ot^*V$, \ $U\ot^\la V$, and $U\ot^!V$.
\end{lem}

\begin{proof}
 It suffices to consider the case of the $*$\+topology, as it is
the finest one of the three.
 Let $u\in U$ and $v\in V$ be arbitrary vectors; let us show that
the element $u\ot v\in U\ot^*V$ belongs to the closure of
the subspace $K\ot_k L\subset U\ot^*V$.
 Let $E\subset U\ot^*V$ be an open subspace.
 Then there exists an open subspace $Q_u\subset V$ such that
$u\ot Q_u\subset E$.
 Since the subspace $L$ is dense in $V$, we have $Q_u+L=V$.
 Let $q\in Q_u$ and $l\in L$ be vectors such that $v=q+l$.
 There exists an open subspace $P_l\subset U$ such that
$P_l\ot l\subset E$.
 Since the subspace $K$ is dense in $U$, we have $P_l+K=U$.
 Let $p\in P_l$ and $k'\in K$ be vectors such that $u=p+k'$.
 Then $u\ot v=u\ot q+p\ot l+k'\ot l\in E+k'\ot l\subset E+K\ot_kL$.
 Thus $E+K\ot_kL=U\ot_kV$.
\end{proof}

\begin{cor} \label{tensor-preserves-separated-cor}
 Let $U$ and $V$ be separated topological vector spaces.
 Then the topological vector spaces $U\ot^*V$, \ $U\ot^\la V$, and
$U\ot^!V$ are separated, too.
\end{cor}

\begin{proof}
 Follows from Lemma~\ref{closed-subspaces-tensor-lemma} (take $K=0=L$).
\end{proof}

\begin{thm} \label{nonseparated-tensor-exactness-properties}
\textup{(a)} The functor\/ $\ot^*\:\Top_k\times\Top_k\rarrow\Top_k$
preserves the kernels and cokernels of morphisms (in each of its
arguments).
 Fixing a topological vector space in one of the arguments, it becomes
an exact endofunctor\/ $\Top_k\rarrow\Top_k$ in the other argument
(with respect to the quasi-abelian exact structure on\/ $\Top_k$). \par
\textup{(b)} The functor\/ $\ot^\la\:\Top_k\times\Top_k\rarrow\Top_k$
preserves the kernels and cokernels of morphisms (in each of its
arguments).
 Fixing a topological vector space in any one of the arguments, it
becomes an exact endofunctor\/ $\Top_k\rarrow\Top_k$ in the other
argument (with respect to the quasi-abelian exact structure on\/
$\Top_k$). \par
\textup{(c)} The functor\/ $\ot^!\:\Top_k\times\Top_k\rarrow\Top_k$
preserves the kernels and cokernels of morphisms (in each of its
arguments).
 Fixing a topological vector space in one of the arguments, it becomes
an exact endofunctor\/ $\Top_k\rarrow\Top_k$ in the other argument
(with respect to the quasi-abelian exact structure on\/ $\Top_k$). 
\end{thm}

\begin{proof}
 Let us explain part~(a); parts~(b) and~(c) are similar.
 The functor~$\ot^*$ is well-defined by Lemma~\ref{tensor-continuous}.
 For any fixed topological vector space $U$, the functor $U\ot^*{-}\,\:
\Top_k\rarrow\Top_k$ preserves the kernels of morphisms by
Lemma~\ref{tensor-star-induced-topology}, and it preserves
the cokernels of morphisms by Lemma~\ref{tensor-preserves-open-maps}(a).
 Any additive functor between quasi-abelian categories which preserves
the kernels and cokernels is exact with respect to the quasi-abelian
exact structures.
\end{proof}

\begin{thm} \label{incomplete-tensor-exactness-properties}
\textup{(a)} The functor\/ $\ot^*\:\Top^\s_k\times\Top^\s_k\rarrow
\Top^\s_k$ preserves the kernels and cokernels of morphisms (in each of
its arguments).
 Fixing a topological vector space in one of the arguments, it becomes
an exact endofunctor\/ $\Top^\s_k\rarrow\Top^\s_k$ in the other argument
(with respect to the quasi-abelian exact structure on\/ $\Top^\s_k$).
\par
\textup{(b)} The functor\/ $\ot^\la\:\Top^\s_k\times\Top^\s_k\rarrow
\Top^\s_k$ preserves the kernels and cokernels of morphisms (in each of
its arguments).
 Fixing a topological vector space in any one of the arguments, it
becomes an exact endofunctor\/ $\Top^\s_k\rarrow\Top^\s_k$ in the other
argument (with respect to the quasi-abelian exact structure on\/
$\Top^\s_k$). \par
\textup{(c)} The functor\/ $\ot^!\:\Top^\s_k\times\Top^\s_k\rarrow
\Top^\s_k$ preserves the kernels and cokernels of morphisms (in each of
its arguments).
 Fixing a topological vector space in one of the arguments, it becomes
an exact endofunctor\/ $\Top^\s_k\rarrow\Top^\s_k$ in the other
argument (with respect to the quasi-abelian exact structure on\/
$\Top^\s_k$). 
\end{thm}

\begin{proof}
 Let us explain part~(a); parts~(b) and~(c) are similar.
 The functor~$\ot^*$ is well-defined by Lemma~\ref{tensor-continuous}
and Corollary~\ref{tensor-preserves-separated-cor}.
 For any fixed separated topological vector space $U$, the functor
$U\ot^*{-}\,\:\Top^\s_k\rarrow\Top^\s_k$ preserves the kernels by
Lemma~\ref{tensor-star-induced-topology}.
 Furthermore, let $f\:V'\rarrow V''$ be a continuous linear map of
separated topological vector spaces.
 Then the cokernel of the map~$f$ in the category $\Top^\s_k$ is
the quotient space $V''/\overline{f(V')}_{V''}$ endowed with
the quotient topology.
 By Lemmas~\ref{tensor-star-induced-topology},
\ref{closed-subspaces-tensor-lemma},
and~\ref{dense-subspaces-tensor-lemma}, the closure of the vector
subspace $U\ot_k f(V')$ in the topological vector space $U\ot^*V''$
is the subspace $\overline{U\ot_kf(V')}_{U\ot^*V''}=
U\ot_k\overline{f(V')}_{V''}\subset U\ot_kV''$.
 It remains to use Lemma~\ref{tensor-preserves-open-maps}(a) in order
to conclude that the functor $U\ot^*{-}$ preserves the cokernels in
$\Top^\s_k$.
 The second assertion in part~(a) follows from the first one.
\end{proof}

 The following result is an uncompleted version
of~\cite[Corollary in Section~1.1]{Beil}.

\begin{prop} \label{three-uncompleted-topologies-sequence-prop}
 Let $U$ and $V$ be topological vector spaces.
 Then the short exact sequence of (abstract, nontopological)
vector spaces
$$
 0\lrarrow U\ot_k V\lrarrow U\ot_k V\oplus U\ot_k V\lrarrow
 U\ot_k V\lrarrow0,
$$
where the left arrow is the diagonal map and the right one is
the difference of two projections, is exact as a short sequence
\begin{equation} \label{three-uncompleted-topologies-sequence-eqn}
 0\lrarrow U\ot^*V\lrarrow U\ot^\la V\oplus U\ot^\to V\lrarrow
 U\ot^!V\lrarrow0
\end{equation}
in the quasi-abelian exact structure on the category\/ $\Top_k$.
\end{prop}

\begin{proof}
 We recall the notation $U\ot^\to V=V\ot^\la U$.
 The maps involved in the desired short exact
sequence~\eqref{three-uncompleted-topologies-sequence-eqn}
are continuous, since the $*$\+topology is the finest of the three
topologies on $U\ot_kV$, while the $!$\+topology is the coarsest one.
 On top of this observation, the proposition essentially claims two
assertions:
\begin{enumerate}
\renewcommand{\theenumi}{\alph{enumi}}
\item the surjective linear map $U\ot^\la V\oplus U\ot^\to V
\overset p\rarrow U\ot^!V$ is open;
\item the topology of $U\ot^*V$ coincides with the induced topology
on $U\ot_kV$ as the diagonal subspace in
$U\ot^\la V\oplus U\ot^\to V$.
\end{enumerate}

 Proof of~(a): let $F\subset U\ot^\la V$ and $G\subset U\ot^\to V$
be open subspaces.
 Then there exists an open subspace $P\subset U$ such that $P\ot_kV
\subset F$.
 Similarly, there exists an open subspace $Q\subset V$ such that
$U\ot_kQ\subset G$.
 Now $P\ot_kV+U\ot_kQ\subset F+G=p(F\oplus G)$, hence $p(F\oplus G)
\subset U\ot_kV$ is an open subspace in $U\ot^!V$.

 Proof of~(b): let $E\subset U\ot^*V$ be an open subspace.
 It suffices to find two open subspaces $F\subset U\ot^\la V$ and
$G\subset U\ot^\to V$ such that $F\cap G\subset E$ in $U\ot_kV$.

 By the definition of the $*$\+topology, there exist two open
subspaces $P\subset U$ and $Q\subset V$ such that $P\ot_k Q\subset E$.
 Choose a basis $\{u_i\}_{i\in I}$ in $U$ such that a subset of
$\{u_i\}$ is a basis in $P\subset U$.
 Similarly, choose a basis $\{v_j\}_{j\in J}$ in $V$ such that
a subset of $\{v_j\}$ is a basis in $Q\subset V$.

 For every $i\in I$ such that $u_i\notin P$, there exists an open
subspace $Q_i\subset V$ such that $u_i\ot Q_i\subset E$.
 Similarly, for every $j\in J$ such that $v_j\notin Q$, there exists
an open subspace $P_j\subset U$ such that $P_j\ot v_j\subset E$.
 Put
$$
 F = P\ot_kV + \sum\nolimits_{i\in I:u_i\notin P} u_i\ot (Q_i\cap Q)
$$
and
$$
 G = U\ot_kQ + \sum\nolimits_{j\in J:v_j\notin Q} (P_j\cap P)\ot v_j.
$$
 Then $F\subset U\ot^\la V$ and $G\subset U\ot^\to V$ are open
subspaces, and
$$
 F\cap G = P\ot_kQ +
 \sum\nolimits_{i\in I:u_i\notin P} u_i\ot (Q_i\cap Q) +
 \sum\nolimits_{j\in J:v_j\notin Q} (P_j\cap P)\ot v_j
 \,\subset\,E,
$$
as desired.
\end{proof}

\Section{Refined Exact Category Structures on VSLTs}
\label{refined-exact-struct-secn}

 It is claimed in~\cite[Section~1.1, Exercise on page~2]{Beil} that
the completed tensor product operations are exact in the category
of complete, separated topological vector spaces.
 As we will see, the validity of these assertions depends on what is
understood by exactness; but they do not hold if one presumes
exactness with respect to the maximal exact category structure on
$\Top^\scc_k$.

 Let $\fU$ and $\fV$ be two complete, separated topological vector
spaces (with linear topology).
 Let us introduce notation for three completed topological tensor
products of $\fU$ and $\fV$:
\begin{itemize}
\item $\fU\cot^*\fV=(\fU\ot^*\fV)\sphat\,$;
\item $\fU\cot^\la\fV=(\fU\ot^\la\fV)\sphat\,$;
\item $\fU\cot^!\fV=(\fU\ot^!\fV)\sphat\,$.
\end{itemize}

 The simplest examples of complete, separated topological vector
spaces for which the completed tensor products differ from
the uncompleted ones were mentioned in
Remark~\ref{tensor-product-not-complete-remark}.
 Here are some slightly more sophisticated examples. 

\begin{exs} \label{tensor-with-discrete-examples}
 (1)~Let $U$ be a discrete vector space and $V$ be a topological
vector space.
 Then, by the definition, a vector subspace $E\subset U\ot_kV$ is
open in $U\ot^*V$ if and only if it is open in $U\ot^\la V$, and if
and only if for every vector $u\in U$ there exists an open subspace
$Q_u\subset V$ such that $u\ot Q_u\subset E$.
 
 Let $\{u_i\}_{i\in I}$ be a basis in~$U$.
 Identifying the tensor product $U\ot_kV$ with the direct sum
$\bigoplus_{i\in I}V$, one observes that the topology of
$U\ot^*V=U\ot^\la V$ coincides with the coproduct topology of
$\bigoplus_{i\in I}V$.

 By Lemma~\ref{coproduct-topology-lemma}, the direct sum of a family
of complete, separated topological vector spaces is separated and
complete in the coproduct topology.
 Thus, for any discrete vector space $\fU$ and any complete, separated
topological vector space $\fV$, one has $\fU\cot^*\fV=\fU\ot^*\fV=
\bigoplus_{i\in I}\fV=\fU\ot^\la\fV=\fU\cot^\la\fV$, where $I$ is
a set indexing a basis in~$\fU$.

\smallskip
 (2)~Let $U$ be a topological vector space and $V$ be a discrete
vector space.
 Then, by the definition, a vector subspace $E\subset U\ot_kV$ is
open in $U\ot^\la V$ if and only if it is open in $U\ot^!V$, and if
and only if there exists an open subspace $P\subset U$ such that
$P\ot_kV\subset E$.

 Let $\{v_x\}_{x\in X}$ be a basis in~$V$.
 Identifying the tensor product $U\ot_kV$ with the direct sum
$\bigoplus_{x\in X}U=U^{(X)}=U[X]$, one observes that a base of
open subspaces in $U\ot^\la V=U\ot^!V$ is formed by the subspaces
$P[X]$, where $P$ ranges over the open vector subspaces in~$U$.
 This is the topology relevant in the context of the construction
in the beginning of Section~\ref{strong-exact-structure-secn}.

 For a complete, separated topological vector space $\fU$ and
a set $X$, the completion of the vector space $\fU^{(X)}=\fU[X]$
in the above topology is the topological vector space denoted
by $\fU[[X]]$ in Section~\ref{strong-exact-structure-secn}.
 Thus, for any complete, separated topological vector space $\fU$
and a discrete vector space $\fV$, one has $\fU\cot^\la\fV=\fU[[X]]=
\fU\cot^!\fV$, where $X$ is a set indexing a basis in~$\fV$.

\smallskip
 (3)~Let $\fU$ and $\fV$ be complete, separated topological vector
spaces.
 For any open subspace $\fP\subset\fU$, let us view the quotient
space $\fU/\fP$ as a discrete vector space and endow the tensor product
$(\fU/\fP)\ot_k\fV$ with the (complete, separated) topology
described in~(1).
 Then the preimages in $\fU\ot_k\fV$ of open subspaces in
$(\fU/\fP)\ot_k\fV$ form a base of neighborhoods of zero in
$\fU\ot^\la\fV$.
 Hence one has $\fU\cot^\la\fV=\varprojlim_{\fP\subset\fU}
(\fU/\fP\ot_k\fV)$, where the projective limit over the poset of open
subspaces $\fP\subset\fU$ is taken in the category $\Top^\scc_k$,
or equivalently, in any one of the categories $\Top^\s_k$ or $\Top_k$.

 As the projective limit in the categories of topological vector
spaces agrees with the one in $\Vect_k$, it follows that the underlying
vector space of the topological vector space $\fU\cot^\la\fV$ only
depends on the underlying vector space of the topological vector
space $\fV$, and does not depend on a (complete, separated) topology
on~$\fV$.
 So there is a well-defined functor of completed tensor product
$$
\cot^\la\:\Top^\scc_k\times\Vect_k\lrarrow\Vect_k,
$$
defined by the rule $\fU\cot^\la V=\varprojlim_{\fP\subset\fU}
(\fU/\fP\ot_k V)$ for all $\fU\in\Top^\scc_k$, \ $V\in\Vect_k$,
and agreeing with the functor
$\cot^\la\:\Top^\scc_k\times\Top^\scc_k\rarrow\Top^\scc_k$.
\end{exs}

 The following lemma holds for topological abelian groups just as well,
but we will only use it for topological vector spaces.

\begin{lem} \label{completion-exactness-properties-lemma}
\textup{(a)} The completion functors $V\longmapsto V\sphat\,\:
\Top_k\rarrow\Top^\scc_k$ and $V\longmapsto V\sphat\,\:
\Top^\s_k\rarrow\Top^\scc_k$ preserve the cokernels of morphisms. \par
\textup{(b)} For any injective morphism of topological vector spaces\/
$\iota\:U\rarrow V$ such that the subspace\/ $\iota(U)$ is dense in $V$
and the topology of $U$ is induced from the topology of $V$ via~$\iota$,
the induced morphism of the completions
$\iota\sphat\,\:U\sphat\,\rarrow V\sphat\,$ is an isomorphism of
topological vector spaces. \par
\textup{(c)} The completion functors $V\longmapsto V\sphat\,\:
\Top_k\rarrow\Top^\scc_k$ and $V\longmapsto V\sphat\,\:
\Top^\s_k\rarrow\Top^\scc_k$ take short sequences satisfying Ex1 to
short sequences satisfying Ex1.
\end{lem}

\begin{proof}
 Part~(a): the completion functors are the reflectors, i.~e., they are
left adjoint to the inclusion functors $\Top^\scc_k\rarrow\Top_k$ and
$\Top^\scc_k\rarrow\Top^\s_k$, respectively (cf.\
Section~\ref{top-abelian-secn}).
 All left adjoint functors preserve all colimits, and in particular,
cokernels.

 Part~(b): first let us assume that $V$ is separated; then $U$ is
separated, too.
 For any separated topological vector space (or abelian group)
$V$ and its completion $V\sphat\,$, the topology of $V$ is induced from
the topology of $V\sphat\,$ via the injective completion map
$V\rarrow V\sphat\,$, which makes $V$ a dense subspace/subgroup
in~$V\sphat\,$.
 In the situation at hand, $U$ is dense in $V$ and $V$ is dense in
$V\sphat\,$, hence $U$ is dense in~$V\sphat\,$.
 The topology of $U$ is induced from the topology of $V$ and
the topology of $V$ is induced from the topology of $V\sphat\,$, hence
the topology of $U$ is induced from the topology of~$V\sphat\,$.
 It remains to apply Lemma~\ref{closure-completion-lemma} to
the embedding $U\rarrow V\sphat\,$.

 The general case is reduced to the separated case by the passage to
the maximal separated quotient spaces $U/\overline{\{0\}}_U$ and
$V/\overline{\{0\}}_V$ of the topological vector spaces $U$ and $V$
(see the proof of part~(c) below for some additional details).

 Part~(c): let $0\rarrow K\overset i\rarrow V\overset p\rarrow C
\rarrow0$ be a short sequence satisfying Ex1 in $\Top^\s_k$ (i.~e.,
a short exact sequence in the quasi-abelian exact structure
on $\Top^\s_k$).
 By part~(a), the morphism $p\sphat\,\:V\sphat\,\rarrow C\sphat\,$
is the cokernel of the morphism $i\sphat\,\:K\sphat\,\rarrow
V\sphat\,$.
 Applying Lemma~\ref{closure-completion-lemma} to the embedding
$K\rarrow V\rarrow V\sphat\,$, one can see that the morphism
$i\sphat\,=K\sphat\,\rarrow V\sphat\,$ is an injective closed map.
 By Proposition~\ref{top-groups-spaces-kernels-cokernels-prop}(b),
this means that $i\sphat\,$~is a kernel in $\Top^\scc_k$; hence
$i\sphat\,$ is a kernel of its cokernel~$p\sphat\,$.

 Similarly, let $0\rarrow K\overset i\rarrow V\overset p\rarrow C
\rarrow0$ be a short sequence satisfying Ex1 in $\Top_k$ (i.~e.,
a short exact sequence in the quasi-abelian exact structure
on $\Top_k$).
 By part~(a), the morphism $p\sphat\,\:V\sphat\,\rarrow C\sphat\,$
is the cokernel of the morphism $i\sphat\,\:K\sphat\,\rarrow
V\sphat\,$.
 Let $K_0=\overline{\{0\}}_K\subset K$ and 
$V_0=\overline{\{0\}}_V\subset V$ be the closures of the zero subgroup
in $K$ and in~$V$.
 Then one has $K_0=V_0\cap K$, so the induced map $i'\:K/K_0\rarrow
V/V_0$ is injective.
 Moreover, the quotient topology on $K/K_0$ is induced from
the quotient topology of $V/V_0$ via~$i'$.
 Applying Lemma~\ref{closure-completion-lemma} to the embedding
$K/K_0\rarrow V/V_0\rarrow (V/V_0)\sphat\,=V\sphat\,$, one can see
that the morphism $i\sphat\,=i'{}\sphat\,\:(K/K_0)\sphat\,=K\sphat\,
\rarrow V\sphat\,$ is an injective closed map; and the argument
finishes similarly to the previous case.
\end{proof}

\begin{ex} \label{kernel-created-by-completion}
 Very simple counterexamples show that the completion functor does
\emph{not} preserve the kernels of morphisms of topological vector
spaces.
 In fact, it can transform an injective continuous linear map of
separated topological vector spaces into a noninjective map of
the completions.
 For example, let $C$ be an incomplete separated topological vector
space and $\fC=C\sphat\,$ be its completion.
 Choose a vector $x\in\fC\setminus C$; then $x\:k\rarrow\fC$ is
a split monomorphism of topological vector spaces (where
the one-dimensional vector space~$k$ is endowed with the discrete
topology) by Corollary~\ref{countable-base-complete-subspace-splits}.
 Hence the quotient space $\fC/kx$ is a direct summand in $\fC$,
so it is separated and complete.
 Thus the completion functor transforms the injective map $C\rarrow
\fC/kx$ into the split epimorphism $\fC\rarrow\fC/kx$ with a nonzero
kernel~$kx$.
\end{ex}

\begin{prop} \label{topological-tensor-associativity}
 The three topological tensor products\/ $\cot^*$, \ $\cot^\la$,
and\/ $\cot^!$ are associative: for any complete, separated topological
vector spaces $\fU$, $\fV$, and $\fW$ there are natural isomorphisms
of (complete, separated) topological vector spaces \par
\textup{(a)} $(\fU\cot^*\fV)\cot^*\fW\simeq\fU\cot^*(\fV\cot^*\fW$);
\par
\textup{(b)} $(\fU\cot^\la\fV)\cot^\la\fW\simeq
\fU\cot^\la(\fV\cot^\la\fW$); \par
\textup{(c)} $(\fU\cot^!\fV)\cot^!\fW\simeq\fU\cot^!(\fV\cot^!\fW)$.
\end{prop}

\begin{proof}
 First one needs to establish associativity of each one of
the three uncompleted tensor products $\ot^*$, \ $\ot^\la$, and $\ot^!$
(as functors $\Top_k\times\Top_k\rarrow\Top_k$).
 The best way to do it is to describe the triple (and multiple)
tensor products, i.~e., to formulate the definitions of
$*$\+-topology, the $\la$\+topology, and the $!$\+topology on
the tensor product of several topological vector spaces.
 Such explicit definitions are given in~\cite[Section~1.1]{Beil}.

 Having convinced oneself that the uncompleted tensor products are
associative, the associativity of the completed tensor products becomes
a formal corollary based on
Lemmas~\ref{tensor-star-induced-topology}\+-%
\ref{tensor-shriek-induced-topology},
\ref{dense-subspaces-tensor-lemma},
and~\ref{completion-exactness-properties-lemma}(b).
 Let us discuss part~(a).
 We want to show that both the iterated completed tensor products in
question are naturally isomorphic to the completion of the uncompleted 
triple tensor product, $(\fU\ot^*\fV\ot^*\fW)\sphat\,$.

 Indeed, following the discussion in the proof of
Lemma~\ref{completion-exactness-properties-lemma}(b),
the uncompleted tensor product $\fU\ot^*\fV$ is a dense
subspace in the completed tensor product $\fU\cot^*\fV$, and
the topology of $\fU\ot^*\fV$ is induced from the topology of
$\fU\cot^*\fV$.
 By Lemmas~\ref{tensor-star-induced-topology}
and~\ref{dense-subspaces-tensor-lemma}, it follows that
$(\fU\ot^*\fV)\ot^*\fW$ is a dense subspace in $(\fU\cot^*\fV)\ot^*\fW$,
and the topology of $(\fU\ot^*\fV)\ot^*\fW$ is induced from
the topology of $(\fU\cot^*\fV)\ot^*\fW$.
 Applying Lemma~\ref{completion-exactness-properties-lemma}(b),
one can conclude that the embedding
$(\fU\ot^*\fV)\ot^*\fW\rarrow(\fU\cot^*\fV)\ot^*\fW$
becomes an isomorphism after the passage to the completions.
\end{proof}

\begin{qst}
 Let $\fU$ and $\fV$ be complete, separated topological vector
spaces.
 By Proposition~\ref{three-uncompleted-topologies-sequence-prop},
there is a short exact sequence of uncompleted tensor products
in the category $\Top_k$, and consequently also $\Top^\s_k$
$$
 0\lrarrow\fU\ot^*\fV\lrarrow\fU\ot^\la\fV\oplus\fU\ot^\to\fV
 \lrarrow\fU\ot^!\fV\lrarrow0.
$$
 Applying the completion functor, we obtain a short sequence
of complete, separated topological vector spaces
\begin{equation} \label{three-completed-topologies-sequence-eqn}
 0\lrarrow\fU\cot^*\fV\lrarrow\fU\cot^\la\fV\oplus\fU\cot^\to\fV
 \lrarrow\fU\cot^!\fV\lrarrow0,
\end{equation}
which satisfies Ex1 in $\Top^\scc_k$ by
Lemma~\ref{completion-exactness-properties-lemma}(c).
 (Here $\fU\cot^\to\fV$ is an alternative notation for
$\fV\cot^\la\fU$.)
 
 Is the map $\fU\cot^\la\fV\oplus\fU\cot^\to\fV\rarrow\fU\cot^!\fV$
surjective?
 In other words, is~\eqref{three-completed-topologies-sequence-eqn}
a short exact sequence in the maximal exact structure
on $\Top^\scc_k$\,? 
\end{qst}

\begin{rem}
 Let $\fU$, $\fV$, and $\fW$ be three complete, separated topological
vector spaces, and let $\phi\:\fU\times\fV\rarrow\fW$ be a continuous
bilinear map.
 Then it is clear from
Remark~\ref{three-incomplete-topologies-motivation-remark} that
the linear map $\phi^\ot\:\fU\ot_k\fV\rarrow\fW$ extends uniquely to
a continuous linear map $\fU\cot^*\fV\rarrow\fW$.
 So continuous bilinear pairings $\fU\times\fV\rarrow\fW$ correspond
bijectively to continuous linear maps $\fU\cot^*\fV\rarrow\fW$.

 Let $\fR$ be a complete, separated topological vector space with
a topological (associative) algebra structure.
 Following the same remark, if open two-sided ideals form a base of
neighborhoods of zero in $\fR$, then the multiplication in $\fR$
gives rise to a continuous linear map $\fR\cot^!\fR\rarrow\fR$.
 If open right ideals form a base of neighborhoods of zero in $\fR$,
then the multiplication in $\fR$ can be described as a continuous
linear map $\fR\cot^\la\fR\rarrow\fR$.
 Using Proposition~\ref{topological-tensor-associativity},
the associativity of multiplication in $\fR$ can be formulated in
terms of the topological tensor products.

 Assume from now on that open right ideals form a base of neighborhoods
of zero in~$\fR$.
 Let $N$ be a discrete $k$\+vector space.
 Then discrete right $\fR$\+module structures on $N$ (in the sense
of~\cite[Section~VI.4]{St}, \cite[Sections~2.3\+-2.4]{Pcoun},
\cite[Section~2.4]{Pproperf}) can be described as continuous linear
maps $N\cot^\la\fR\rarrow N$ (satisfying the appropriate associativity
equation).
 See~\cite[Section~1.4]{Beil} for a discussion.
 
 Furthermore, assume that the algebra $\fR$ has a unit.
 Let $B$ be an abstract (nontopological) vector space.
 Following Example~\ref{tensor-with-discrete-examples}(3),
an abstract (nontopological) vector space $\fR\cot^\la B$ is
well-defined.
 According to~\cite[Section~1.10]{Pweak} or~\cite[Section~2.3]{Prev},
left $\fR$\+contramodule structures on $B$ (in the sense
of~\cite{PR,PS,Pcoun,Pproperf,PS3}) can be described as
linear maps $\fR\cot^\la B\rarrow B$ (satisfying
suitable contraassociativity and contraunitality equations).
 This definition of $\fR$\+contramodules for a topological algebra $\fR$
over a field~$k$ appeared already in~\cite[Section~D.5.2]{Psemi}.
\end{rem}

\begin{prop} \label{complete-tensor-exactness-properties}
\textup{(a)} The functor\/ $\cot^*\:\Top^\scc_k\times\Top^\scc_k\rarrow
\Top^\scc_k$ preserves the cokernels of morphisms (in each
of its arguments).
 Fixing a topological vector space in one of the arguments, it becomes
an endofunctor\/ $\Top^\scc_k\rarrow\Top^\scc_k$ taking short
sequences satisfying Ex1 to short sequences satisfying Ex1. \par
\textup{(b)} The functor\/ $\cot^\la\:\Top^\scc_k\times\Top^\scc_k
\rarrow\Top^\scc_k$ preserves the cokernels of morphisms
(in each of its arguments). 
 Fixing a topological vector space in any one of the arguments,
it becomes an endofunctor\/ $\Top^\scc_k\rarrow\Top^\scc_k$ taking short
sequences satisfying Ex1 to short sequences satisfying Ex1. \par
\textup{(c)} The functor\/ $\cot^!\:\Top^\scc_k\times\Top^\scc_k\rarrow
\Top^\scc_k$ preserves the cokernels of morphisms (in each
of its arguments).
 Fixing a topological vector space in one of the arguments, it becomes
an endofunctor\/ $\Top^\scc_k\rarrow\Top^\scc_k$ taking short
sequences satisfying Ex1 to short sequences satisfying Ex1.
\end{prop}

\begin{proof}
 Let us explain part~(a); parts~(b) and~(c) are similar.
 Let $f\:\fU'\rarrow\fU''$ be a morphism in $\Top^\scc_k$, and
$C=\fU''/\overline{f(\fU')}_{\fU''}$ be the cokernel of~$f$
in the category $\Top^\s_k$.
 Let $\fV$ be a complete, separated topological vector space.
 By Theorem~\ref{incomplete-tensor-exactness-properties}(a),
the uncompleted tensor product $C\ot^*\fV$ is the cokernel of
the morphism $f\ot^*\fV\:\allowbreak
\fU'\ot^*\nobreak\fV\rarrow\fU''\ot^*\fV$ in the category $\Top^\s_k$.
 According to Lemma~\ref{completion-exactness-properties-lemma}(a),
it follows that the completion $(C\ot^*\fV)\sphat\,$ is
the cokernel of the morphism of completed tensor products
$f\cot^*\fV\:\fU'\cot^*\fV\rarrow\fU''\cot^*\fV$ in the category
$\Top^\scc_k$.
 Finally, $\fC=C\sphat\,$ is the cokernel of the morphism~$f$
in $\Top^\scc_k$; and by Lemmas~\ref{tensor-star-induced-topology}
and~\ref{dense-subspaces-tensor-lemma}, the topological vector
space $C\ot^*\fV$ is a dense subspace in $\fC\ot^*\fV$ with
the topology of $C\ot^*\fV$ induced from the topology of
$\fC\ot^*\fV$.
 By Lemma~\ref{completion-exactness-properties-lemma}(b), we can
conclude that $(C\ot^*\fV)\sphat\,\simeq(\fC\ot^*\fV)\sphat\,=
\fC\cot^*\fV$.

 Let $0\rarrow\fK\rarrow\fU\rarrow\fC\rarrow0$ be a short sequence
satisfying Ex1 in $\Top^\scc_k$, and let $\fV$ be a complete, separated
topological vector space.
 Then $\fC=C\sphat\,$, where $0\rarrow\fK\rarrow\fU\rarrow C\rarrow0$
is a short exact sequence in $\Top^\s_k$.
 By Theorem~\ref{incomplete-tensor-exactness-properties}(a),
$0\rarrow\fK\ot^*\fV\rarrow\fU\ot^*\fV\rarrow C\ot^*\fV\rarrow0$
is a short exact sequence in $\Top^\s_k$, too.
 Passing to the completions, by
Lemma~\ref{completion-exactness-properties-lemma}(c) we get
a short sequence $0\rarrow\fK\cot^*\fV\rarrow\fU\cot^*\fV\rarrow
(C\ot^*\fV)\sphat\,\rarrow0$ satisfying Ex1 in $\Top^\scc_k$.
 Following the argument in the previous paragraph,
$(C\ot^*\fV)\sphat\,\simeq\fC\cot^*\fV$, and we are done.
\end{proof}

\begin{qst}
 Do the functors $\cot^*$, \ $\cot^\la$, and/or
$\cot^!\:\Top^\scc_k\times\Top^\scc_k\rarrow\Top^\scc_k$ preserve
the kernels of morphisms (in any one of their arguments, with
the other argument fixed) in the category $\Top^\scc_k$\,?
 Notice that the completion, generally speaking, does not preserve
kernels; see Example~\ref{kernel-created-by-completion}.
\end{qst}

\begin{cor} \label{tensor-not-exact-in-the-maximal}
\textup{(a)} For any infinite-dimensional discrete vector space\/ $\fV$,
the functor\/ ${-}\ot^\la\nobreak\fV\:\allowbreak
\Top^\scc_k\rarrow\Top^\scc_k$ is \emph{not} exact in the maximal exact
structure on\/ $\Top^\scc_k$. \par
\textup{(b)} For any infinite-dimensional discrete vector space\/ $\fV$,
the functor\/ ${-}\ot^!\nobreak\fV\:\allowbreak
\Top^\scc_k\rarrow\Top^\scc_k$ is \emph{not} exact in the maximal exact
structure on\/ $\Top^\scc_k$.
\end{cor}

\begin{proof}
 In fact, the functor in parts~(a) and~(b) is one and the same,
according to Example~\ref{tensor-with-discrete-examples}(2).
 It is not exact with respect to the maximal exact structure on
$\Top^\scc_k$ by
Corollary~\ref{converging-combinations-not-exact-for-maximal}(b).
\end{proof}

\begin{prop}
 There exists a unique exact category structure on the category\/
$\Top^\scc_k$ which is maximal among all the exact structures having
the property that the functors\/ $\cot^*$, \ $\cot^\la$, and\/
$\cot^!\:\Top^\scc_k\times\Top^\scc_k\rarrow\Top^\scc_k$ are exact
in it (as functors of any one argument with the topological vector
space in the other argument fixed).
\end{prop}

\begin{proof}
 A short sequence $0\rarrow\fK\overset i\rarrow\fU\overset p\rarrow
\fC\rarrow0$ is said to be exact in the desired exact category
structure (which can be called the \emph{tensor-refined exact
structure} on $\Top^\scc_k$) if it is exact in the maximal exact
category structure and remains exact in the maximal exact category
structure after any finite number of iterated applications of
the functors ${-}\cot^*\fV$, \ ${-}\cot^\la\fV$, \ $\fV\cot^\la{-}$,
and ${-}\cot^!\fV$.

 This means that the maps $i$ and~$p$ are each other's kernel and
cokernel in $\Top^\scc_k$, the map~$p$ is surjective, and the map~$p$
stays surjective after any repeated application of the completed
tensor product functors.
 Any short exact sequence in the tensor-refined exact structure is
also exact in the strong exact structure of
Theorem~\ref{strong-exact-structure-theorem}
(by Example~\ref{tensor-with-discrete-examples}(2)), but we do
\emph{not} know whether the converse holds.

 To convince oneself that the above definition actually produces
an exact category structure on $\Top^\scc_k$, one can use
the construction of Example~\ref{psi-exact-structure} (cf.\
the proof of Theorem~\ref{strong-exact-structure-theorem}).
 One only needs to observe that the completed tensor product functors
(viewed as functors of any one argument with the other argument fixed)
preserve both the cokernels of all morphisms and the kernels of
admissible epimorphisms in the maximal exact structure on
$\Top^\scc_k$ by Proposition~\ref{complete-tensor-exactness-properties}.
 We omit the straightforward details.
\end{proof}

\begin{conc} \label{final-conclusion}
 The uncompleted tensor product functors $\ot^*$, \ $\ot^\la$,
and $\ot^!$, acting in the categories of incomplete topological
vector spaces $\Top_k$ and $\Top^\s_k$, have nice exactness properties.
 They preserve the kernels and cokernels of all morphisms, and are exact
in the quasi-abelian exact structures of these additive categories
(by Theorems~\ref{nonseparated-tensor-exactness-properties}
and~\ref{incomplete-tensor-exactness-properties}).

 The completed tensor product functors $\cot^*$, \ $\cot^\la$,
and $\cot^!$, acting in the category of complete, separated
topological vector spaces $\Top^\scc_k$, have seemingly paradoxical
properties.
 They preserve the cokernels of all morphisms, and take short sequences
satisfying Ex1 (the ``kernel-cokernel pairs'') to short sequences
satisfying~Ex1; see
Proposition~\ref{complete-tensor-exactness-properties}.
 But the category $\Top^\scc_k$ is \emph{not} quasi-abelian, and
the class of all short sequences satisfying Ex1 is not well-behaved
in it.
 More specifically, the class of all cokernels in $\Top^\scc_k$ is not
well-behaved, because it includes nonsurjective cokernels (see
Corollary~\ref{not-left-quasi-abelian-cor}
and Conclusion~\ref{maximal-exact-conclusion}).

 The class of all stable short exact sequences, forming the maximal
exact structure on $\Top^\scc_k$, is better behaved.
 But the functors $\cot^\la$ and $\cot^!$ do \emph{not} preserve
the semi-stability (\,$=$~stability) of cokernels.
 Accordingly, these functors are \emph{not} exact in the maximal exact
structure on $\Top^\scc_k$ (see
Corollary~\ref{tensor-not-exact-in-the-maximal}).

 One can construct an exact structure on $\Top^\scc_k$ by imposing
the condition of preservation of short exact sequences by the tensor
product operations.
 But it is not clear how to describe this exact structure more
explicitly.

 No exactness problems arise in connection with the category
$\Top^{\omega,\scc}_k$ of complete, separated topological vector spaces
\emph{with a countable base of neighborhoods of zero}, as this category
is quasi-abelian and all the short exact sequences in its quasi-abelian
exact structure are split.
 But the tensor product operations $\cot^*$ and $\cot^\la$ are
\emph{not} defined on the full subcategory $\Top^{\omega,\scc}_k
\subset\Top^\scc_k$, as they do not preserve the countable base of
neighborhoods of zero
(compare Example~\ref{tensor-with-discrete-examples}(1)
with Lemma~\ref{countable-coproducts-not-preserve-countable-base}).
\end{conc}

\end{document}